\documentclass[12pt]{amsart}

\makeatletter
\newcommand\mathcircled[1]{%
	\mathpalette\@mathcircled{#1}%
}
\newcommand\@mathcircled[2]{%
	\tikz[baseline=(math.base)] \node[draw,ellipse,inner sep=1pt] (math) {$\m@th#1#2$};%
}

\usepackage{multicol}
\usepackage{latexsym}
\usepackage{psfrag}
\usepackage{amsmath}
\usepackage{amssymb}
\usepackage{epsfig}
\usepackage{amsfonts}
\usepackage{amscd}
\usepackage{mathrsfs}
\usepackage{graphicx}
\usepackage{enumerate}
\usepackage[autostyle=false, style=english]{csquotes}
\MakeOuterQuote{"}
\usepackage{ragged2e}
\usepackage[all]{xy}
\usepackage{mathtools}

\newlength\ubwidth

\usepackage[dvipsnames]{xcolor}

\usepackage{placeins}
\usepackage{tikz}
\usetikzlibrary{shapes, positioning}
\usetikzlibrary{matrix,shapes.geometric,calc,backgrounds}

\newtheorem{theorem}{Theorem}%[section]
\newtheorem{lemma}[theorem]{Lemma}

\newtheorem{cor}[theorem]{Corollary}
\newtheorem{prop}[theorem]{Proposition}

\newtheorem{defn}[theorem]{Definition}

\usepackage{amsmath,amssymb,amsthm}
\usepackage[dvipsnames]{xcolor}
\usepackage[colorlinks=true,citecolor=RoyalBlue,linkcolor=red,breaklinks=true]{hyperref}

\usepackage{graphicx}
\usepackage{mathdots} % for anti diagonal dots
\usepackage[margin=2.5cm]{geometry}
\usepackage{ytableau}
\newcommand*\circled[1]{\tikz[baseline=(char.base)]{
		\node[shape=circle,draw,inner sep=2pt] (char) {#1};}}

\DeclareMathOperator*{\aast}{ \hspace{-1mm} 
\scalebox{1.5}{ \raisebox{-0.3mm}{\ensuremath{\ast}} } \hspace{-1mm} }

\usepackage{cancel}

\usepackage{blkarray}
\usepackage{booktabs}

\begin{document}%%%%%%%%%%%%%%%%%%%%%%%%%%%%%%%%%%%%%%%%%%%%%%%%%%%%%%%%
	%%%%%%%%%%%%%%%%%%%%%%%%%%%%%%%%%%%%%%%%%%%%%%%%%%%%%%%%%%%%%%%%%%%%%%%%

\title[Decomposition of and Distinct parts in Cylindric Partitions]
{A Decomposition of Cylindric Partitions and Cylindric Partitions into Distinct Parts}

\author[Kur\c{s}ung\"{o}z]{Ka\u{g}an Kur\c{s}ung\"{o}z}
\address{Ka\u{g}an Kur\c{s}ung\"{o}z, Faculty of Engineering and Natural Sciences, 
    Sabanc{\i} University, Tuzla, Istanbul 34956, Turkey}
\email{kursungoz@sabanciuniv.edu}

\author[\"{O}mr\"{u}uzun Seyrek]{Hal\.{ı}me \"{O}mr\"{u}uzun Seyrek}
\address{Hal\.{ı}me \"{O}mr\"{u}uzun Seyrek, Faculty of Engineering and Natural Sciences, 
    Sabanc{\i} University, Tuzla, Istanbul 34956, Turkey}
\email{halimeomruuzun@alumni.sabanciuniv.edu}



\date{2023}

\begin{abstract}
We show that cylindric partitions are in one-to-one correspondence 
with a pair which has an ordinary partition and 
a colored partition into distinct parts.  
Then, we show the general form of the generating function for 
cylindric partitions into distinct parts and give some examples.  
We prove part of a conjecture by Corteel, Dousse, and Uncu.  
The approaches and proofs are elementary and combinatorial.
\end{abstract}

\maketitle

\section{Introduction and statement of results}
\label{secIntro}

Cylindric partitions were introduced by Gessel and Krattenthaler \cite{GesselKrattenthaler}. 

\begin{defn}\label{def:cylin} Let $r$ and $\ell$ be positive integers.
Let $c=(c_1,c_2,\dots, c_r)$ be a composition, where $c_1+c_2+\dots+c_r=\ell$.
A \emph{cylindric partition with profile $c$} is a vector partition
$\Lambda = (\lambda^{(1)},\lambda^{(2)},\dots,\lambda^{(r)})$,
where each $\lambda^{(i)} = \lambda^{(i)}_1+\lambda^{(i)}_2 + \cdots +\lambda^{(i)}_{s_i}$ is a partition,
such that for all $i$ and $j$,
\[
  \lambda^{(i)}_j\geq \lambda^{(i+1)}_{j+c_{i+1}} \quad \text{and} \quad \lambda^{(r)}_{j}\geq\lambda^{(1)}_{j+c_1}.
\]
The integers $r$ and $\ell$ are called
the \emph{rank}, and the \emph{level}, of the cylindric partition, 
respectively.  
\end{defn}

For example, the sequence 
$\Lambda=((10,5,4,1),(12,8,5,3),(7,6,4,2))$
is a cylindric partition with profile $(1,2,0)$. 
One can check that for all $j$, $\lambda^{(1)}_j\ge \lambda^{(2)}_{j+2}$, $\lambda^{(2)}_j\ge \lambda^{(3)}_{j}$
and $\lambda^{(3)}_j\ge \lambda^{(1)}_{j+1}$.
We can visualize the required inequalities by writing the partitions 
in subsequent rows repeating the first row below the last one, 
and shifting the rows below as much as necessary to the left.  
Thus, the inequalities become the weakly decreasing of the parts to the left in each row, 
and downward in each column.  
\[
\begin{array}{ccc ccc ccc}
	& & & 10 & 5 & 4 & 1\\
	& 12 & 8 & 5 & 3 & \\
	& 7 & 6 & 4 & 2& \\
	\textcolor{lightgray}{10} & \textcolor{lightgray}{5} 
	& \textcolor{lightgray}{4} & \textcolor{lightgray}{1}	
\end{array}
\]
The repeated first row is shown in gray.  

The size $\vert \Lambda \vert$ of a cylindric partition
$\Lambda = (\lambda^{(1)},\lambda^{(2)},\dots,\lambda^{(r)})$
is defined to be the sum of all the parts in the partitions
$\lambda^{(1)},\lambda^{(2)},\dots,\lambda^{(r)}$.
The largest part of a cylindric partition $\Lambda$
is defined to be the maximum part among all the partitions in $\Lambda$,
and it is denoted by $\max(\Lambda)$.
For the cylindric partition $\Lambda$ given above, 
$\vert \Lambda \vert = 67$, and $\mathrm{max}(\Lambda) = 12$.  

The following generating function
$$F_c(z,q):=\sum_{\Lambda\in \mathcal{P}_c} z^{\max{(\Lambda)}}q^{\vert\Lambda \vert}$$
is the generating function for cylindric partitions, where $\mathcal{P}_c$ denotes the set of all cylindric partitions with profile $c$. 

In 2007, Borodin \cite{Borodin} showed that when one sets $z=1$ to this generating function, it turns out to be a very nice infinite product. 

\begin{theorem}[Borodin, 2007]	
\label{theorem-Borodin}
	Let $r$ and $\ell$ be positive integers,
	and let $c=(c_1,c_2,\dots,c_r)$ be a composition of $\ell$. Define $t:=r+\ell$
	and $s(i,j) := c_i+c_{i+1}+\dots+ c_j$.
	Then,
	\begin{equation} 
		\label{BorodinProd}
		F_c(1,q) = \frac{1}{(q^t;q^t)_\infty}
				\prod_{i=1}^r \prod_{j=i}^r \prod_{m=1}^{c_i}
				\frac{1}{(q^{m+j-i+s(i+1,j)};q^t)_\infty}
				\prod_{i=2}^r \prod_{j=2}^i \prod_{m=1}^{c_i} \frac{1}{(q^{t-m+j-i-s(j,i-1)};q^t)_\infty}.
	\end{equation}
\end{theorem}

The identity (\refeq{BorodinProd}) is a very strong tool to find product representation of generating functions of cylindric partitions with a given profile explicitly.  

Here and throughout, 
we use the following standard $q$-Pochhammer symbols~\cite{GR}.  
\begin{align}
\nonumber 
  (a; q)_n := \prod_{j = 1}^n (1 - aq^{j-1}) 
  \quad \textrm{ and } \quad 
  (a; q)_\infty := \lim_{n \to \infty} (a; q)_n, 
\end{align}
for any $n \in \mathbb{N}$, $a, q \in \mathbb{C}$, 
and $\vert q \vert < 1$.  
The last condition on $q$ will ensure that 
all series and infinite products in this note converge absolutely\cite{NT-Rama, GR}.  

Decompositions of cylindric partitions into several combinatorial components 
has been studied widely.  
To name a few, by Borodin~\cite{Borodin}, 
Corteel~\cite{Corteel-RR-RSK}, Corteel, Savelief, and Vuletić~\cite{CSV}, 
Langer~\cite{Langer-I, Langer-II}, 
van Leeuwen~\cite{Leeuwen},
and Tingley~\cite{Tingley, Tingley-Correction}.
In particular, Borodin~\cite{Borodin} proved that 
cylindric partitions with a fixed profile are in one-to-one correspondence 
with a pair consisting of an ordinary partition and 
an arbitrarily labeled cylindric diagram.  
A bijective proof of this correspondence is given by Langer~\cite{Langer-II}, 
based on the notion of growth diagrams 
developed by Fomin~\cite{Fomin88, Fomin95}.  
Langer~\cite{Langer-II} also cites Krattenthaler~\cite{Krattenthaler-growth-diag} 
for a similar bijection.  
Tingley, too, proved that each cylindric partition corresponds to a pair
consisting of an ordinary partition and a labeled partition into distinct parts.
Our first general result is the following theorem.  

\begin{theorem}
\label{thmCylPtnVsPtnPairsFullCase}
  For any profile $c = (c_1, \ldots, c_r)$, 
  the cylindric partitions with profile $c$ 
  are in one-to-one correspondence with pairs of partitions $(\mu, \beta)$ 
  in which $\mu$ is an unrestricted partition 
  and $\beta$ is a partition into labeled distinct parts.  
  Moreover, the possible labels of parts of $\beta$
  are determined by the profile $c$, 
  and the length of runs in $\beta$ cannot exceed
  the smaller of $(r-1)$ and $(\ell-1)$.
\end{theorem}

The $\beta$'s mentioned in Theorem do not seem to be directly related 
to Borodin's arbitrarily labeled cylindric diagrams~\cite{Borodin},
and the underlying idea is based on growing Ferrers boards~\cite{GEA-E}.  
rather than growing diagrams~\cite{Fomin88, Fomin95}.  
Tingley's bijection~\cite{Tingley, Tingley-Correction} is much more similar to ours,
but there are still significant differences.
Tingley also constructs a bijection between his pairs and Borodin's pairs.
% Theorem \ref{thmCylPtnVsPtnPairsFullCase} is applied for (for what?).  

Cylindric partitions with certain distinctness conditions 
were first considered by Bridges and Uncu~\cite{BU}.  
The authors gave generating functions for cylindric partitions into distinct parts 
for profiles $(0,2)$ and $(1,1)$~\cite{KO}.  
For a non-example, the cylindric partition 
$\Lambda=((10,5,4,1),$ $(12,8,5,3),$ $(7,6,4,2))$ with profile $(1,2,0)$ 
does not have distinct parts.  
4 and 5 are repeated.  

Let $\mathcal{D}_c(n)$ denote the number of cylindric partitions of $n$ 
with profile $c$ into distinct parts. 
Our next finding is the generalization of Theorem 10 in~\cite{KO}
to all profiles.  

\begin{theorem}
\label{thmDistPartsGeneral}
  Set 
  \begin{align}
  \nonumber 
    D_c(q) = \sum_{n \geq 0} \mathcal{D}_c(n) q^n.  
  \end{align}
  Then, there exist a polynomial $rg_c(q)$ which depends on the profile $c$,
  $k \in \mathbb{N}$ and doubly indexed complex numbers
  $\alpha_{\cdots}^{\cdots}$, $\beta_{\cdots}^{\cdots}$ 
  such that $D_c(q)$ is the following finite linear combination of 
  infinite products and infinite products multiplied by Lambert series,
  up to a correction by the polynomial $rg_c(q)$.
  \begin{align}
  \nonumber
   D_c(q) = rg_c(q) + & \sum_{i = 1}^{B_0} \alpha_i^0 \; ( - \beta_i^0 q ; q)_\infty \\
  \nonumber
   & + \sum_{i = 1}^{B_1} \alpha_i^1 \; 
    \left[ \left( z \frac{\mathrm{d}}{\mathrm{d}z} \right) 
    ( - \beta_i^1 z q ; q)_\infty \right]_{z = 1} \\ 
  \nonumber 
   & \vdots \\ 
  \nonumber 
%   \label{eqDistPartsGenFunc}
   & + \sum_{i = 1}^{B_k} \alpha_i^k \; 
    \left[ \left( z \frac{\mathrm{d}}{\mathrm{d}z} \right)^k 
    ( - \beta_i^k z q ; q)_\infty \right]_{z = 1} 
  \end{align}

\end{theorem}

It is routine to calculate 
$\left( z \frac{\mathrm{d}}{\mathrm{d}z} \right)^k ( - \beta z q ; q)_\infty$ 
as the product of $( - \beta z q ; q)_\infty$ and a polynomial in 
\begin{align}
\nonumber 
  \sum_{n \geq 0} \left( \frac{\beta z q^n}{ 1 + \beta z q^n } \right)^j
\end{align}
for $j = 1, 2, \ldots, k$ using logarithmic differentiation.  
The displayed series are examples of Lambert series~\cite{NT-Rama}.  
In fact, this polynomial in Lambert series will be a homogeneous polynomial 
with degree $k$ in some umbral sense~\cite{Umbral}.  

We also prove that for rank $r = 2$, $k$ in Theorem \ref{thmDistPartsGeneral} 
is necessarily zero, hence there are no Lambert series in that case 
(Corollary \ref{corDistPartsRank2}).  

Constructing evidently positive series for $F_c(z, q)$ or $F_c(1, q)$
in the spirit of Andrews-Gordon identities~\cite{GEA-AG}
has been another challenging follow-up problem in cylindric partitions.  
Some examples, and by no means a complete list of, are 
Andrews, Schilling, and Warnaar~\cite{ASW}, %99 
Feigin, Foda, and Welsh~\cite{FFW}, %08
Corteel and Welsh~\cite{CW}, %19 
Corteel, Dousse, and Uncu~\cite{CDU}, %22 
Kanade and Russell~\cite{KR-completeASW}, %22
Tsuchioka~\cite{Tsu}, %23 
Uncu~\cite{AU23}, %23 
and Warnaar~\cite{Warnaar-A2AG, Warnaar-BaileyTree}. %23
For more complete lists of references and other connections, e.g. lie-theoretic, 
we refer the reader to~\cite{ASW, CDU, KR-completeASW, Warnaar-A2AG, Warnaar-BaileyTree}.

Our next result is the proof of part of a conjecture in
Corteel, Dousse and Uncu~\cite{CDU}.  
Their approach in~\cite{CDU} is using the functional equation 
derived by Corteel and Welsh~\cite{CW} 
with some computational aid by Ablinger and Uncu~\cite{AbU}.  
We prove that in 
\begin{align}
\label{eqClyPtnGenFuncWithPs}
  F_c(z, q)
  = (1 - z) \sum_{n \geq 0} \frac{ P_{n, c}(q) z^n }{ (q^r; q^r)_n }
  = \sum_{n \geq 0} \frac{ P_{=n, c}(q) z^n }{ (q^r; q^r)_n }, 
\end{align}
both $P_{n, c}(q)$'s and $P_{=n, c}(q)$'s
are polynomials with positive coefficients, 
and $P_{n, c}(1) = P_{=n, c}(1) = \binom{\ell + r-1}{r-1}^n$.  
We also give an effective way of computing $P_{n, c}(q)$'s and $P_{=n, c}(q)$'s.  

Finally, as an application of Theorem \ref{thmCylPtnVsPtnPairsFullCase}, we prove that 
\begin{align}
\label{eqQconjPrepGenFunc}
  F_c(q; z) = \frac{1}{(zq; q)_\infty} 
  \sum_{n \geq 0} \frac{ 
  \left( \widetilde{P}_{n, c}(q) 
    + \displaystyle \sum_{ \Lambda \in \mathcal{MJL}_{n, c} }
      \left( \prod_{j \in \iota(\Lambda)} (1 - q^{rj}) \right) q^{\vert \Lambda \vert} \right)
  z^n }{ (q^r; q^r)_n } 
\end{align}
for effectively calculable $\widetilde{P}_{n, c}(q)$ and $\iota(\Lambda)$.  
$\widetilde{P}_{n, c}(q)$ has positive coefficients and 
$\widetilde{P}_{n, c}(1) = \left(\binom{\ell + r-1}{r-1} - r \right)^n$.  
The collection $\mathcal{MJL}_{n, c}$ of cylindric partitions
is finite for any given pair of $c$ and $n$.  
Although \eqref{eqQconjPrepGenFunc} comes nowhere near 
proving the conjectures in Corteel, Dousse, and Uncu~\cite[Conjecture 4.2, part 2]{CDU} 
and Warnaar~\cite[Conjecture 8.4]{Warnaar-A2AG}
% for rank $r = 3$ and level $\ell \not\equiv 0 \pmod{3}$,
it certainly makes them more intriguing because of the negative 
coefficients in the sum over $\Lambda \in \mathcal{MJL}_{n, c}$.
% The said conjectures also seem to extend to the following.
% \begin{conj}
% \label{conjQConjExtended}
%   For any profile $c = (c_1, c_2, \ldots, c_r)$, let
%   \begin{align}
%    F_c(z, q) = \frac{1}{(zq; q)_\infty} \sum_{n \geq 0} \frac{ Q_{n, c}(q) z^n }{ (q; q)_n }.
%   \end{align}
%   If the rank $r \geq 3$, and is relatively prime to the level $\ell = c_1 + c_2 + \cdots + c_r$,
%   i.e. $\mathrm{gcd}(r, \ell) = 1$,
%   then $Q_{n, c}(q)$'s are polynomials, they have positive coefficients,
%   and
%   \begin{align}
%   \nonumber
%     Q_{n, c}(1) = \left( \frac{ \binom{ \ell + r - 1 }{ r - 1 } - r }{r} \right)^n.
%   \end{align}
% \end{conj}

The rest of the paper is organized as follows.  
In Section \ref{secPrelim}, 
we rephrase the introduction of cylindric partitions 
in Gessel and Krattenthaler~\cite{GesselKrattenthaler}, 
and develop the notions we will use in the proofs.  
In Section \ref{secUnrestrictedCylPtn}, 
we prove Theorem \ref{thmCylPtnVsPtnPairsFullCase} in several steps.  
The reason why the proof is not given all at once is also 
explained there.  
The difference with Tingley's approach~\cite{Tingley, Tingley-Correction}
is also emphasized.
In Section \ref{secDist}, Theorem \ref{thmDistPartsGeneral} is proven, 
followed by some ramifications.  
In section \ref{secCDU}, we prove the positivity in \eqref{eqClyPtnGenFuncWithPs}
and in \eqref{eqQconjPrepGenFunc}, 
along with all necessary constructions.  
The constructions in Section \ref{secCDU} is more reminiscent of
Tingley's construction~\cite{Tingley, Tingley-Correction},
but we state the fundamental difference between the two approaches.
% We also indicate how computer aid is used in guessing the Conjecture \ref{conjQConjExtended}.
% Another conjecture (Conjecture \ref{conjQConjPrep}) is given.
All proofs in this paper are elementary.  
We conclude with some ideas for further research in Section \ref{secComments}.  

\section{Preliminaries}
\label{secPrelim}

This section is mostly a rewriting of the introduction in 
Gessel and Krattenthaler~\cite{GesselKrattenthaler}.  

A visualization of cylindric partitions is
replacing each number in the cylindric partition by a vertical stack
consisting of that many unit cubes~\cite{Warnaar-A2AG}, 
then aligning and shifting as imposed by the profile.  
The cylindric partition $\Lambda=((10,5,4,1),(12,8,5,3),(7,6,4,2))$
with profile $(1,2,0)$ will be shown as in Figure \ref{figStackOfCubes}.
The repeated first row is shown as faded.
\begin{figure}
  \centering
    \includegraphics[scale=0.3]{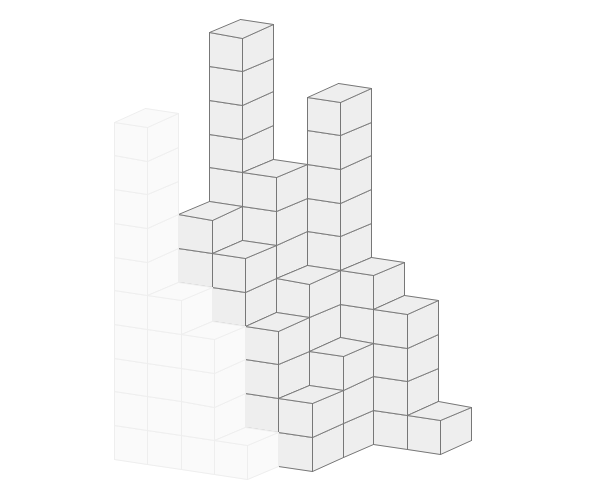}
  \caption{ A visualization of a cylindric partition, the repeated first row is shaded. }
\label{figStackOfCubes}
\end{figure}

Given two cylindric partitions $\Lambda$ and $M$ having the same profile, 
we can declare $\Lambda \geq M$ if 
$\lambda^{(i)}_j \geq \mu^{(i)}_j$ for all pairs $(i, j)$.  
The missing entries may be filled in with zeros, as necessary to make comparisons.  
This defines a partial ordering on cylindric partitions.  
One also finds it convenient to consider the empty cylindric partition $E$ 
which is the unique partition satisfying $\Lambda \geq E$ for all cylindric partitions $\Lambda$.  

One defines $\Lambda+M$ as the componentwise addition, 
taking the missing entries as zeros.  
Is is easy to see that $\Lambda+M$ is another cylindric partition with the same profile,
because the inequalities $\lambda^{(\cdots)}_{\cdots} \geq \lambda^{(\cdots)}_{\cdots}$
and $\mu^{(\cdots)}_{\cdots} \geq \mu^{(\cdots)}_{\cdots}$ 
can be added side by side while preserving the direction of the inequality.  
It is also easy to see that 
$\mathrm{max}(\Lambda+M) \leq \mathrm{max}(\Lambda) + \mathrm{max}(M)$.

One has to be more careful with the componentwise subtraction $\Lambda - M$.  
Even if the resulting object happens to have non-negative entries, 
it has no reason to be a cylindric partition with the same attributes.  
Easy counterexamples may be found, aided by skew diagrams.  
An important exception where the difference is still a cylindric partition 
is as follows.  

Choose and fix a profile $c=(c_1, \ldots, c_r)$.
Given a cylindric partition $\Lambda$ 
with profile $c$ and $m = \mathrm{max}(\Lambda)$; 
set $\Lambda^{(m)} = \Lambda$
and consider the positions $(i, j)$ such that $\lambda_{(i, j)} = m$.  
Form the cylindric partition ${}_m\Sigma$ which consists of 1's in these positions, 
and then subtract 1 from the entries in each of these positions.  
After this operation $\Lambda^{(m)}$ becomes $\Lambda^{(m-1)}$, 
and $\vert {}_m\Sigma \vert + \vert \Lambda^{(m-1)} \vert = \vert \Lambda \vert$.  

To see that ${}_m\Sigma$ and $\Lambda^{(m-1)}$
are cylindric partitions with the same profile $c$,
replace $m$'s with 1's and the strictly smaller parts with zeros in the 
inequalities $\lambda^{(...)}_{(...)} \geq \lambda^{(...)}_{(...)}$, 
and examine case by case.  
We will call cylindric partitions consisting of only 1's \emph{slices}.  
Tingley~\cite{Tingley, Tingley-Correction} also momentarily uses the term slice,
but for vertical slices rather than horizontal slices.
The vertical slices can be read off directly
from the rudimentary representation of a cylindric partition.

We apply the same procedure on $\Lambda^{(m-1)}$ to extract ${}_{m-1}\Sigma$.  
Since all $m$'s in $\Lambda^{(m)}$ became $(m-1)$'s 
after the extraction of ${}_{m}\Sigma$, 
we have ${}_{m-1}\Sigma \geq {}_{m}\Sigma$.  
We keep doing this until we reach $\Lambda^{(0)} = E$, 
the empty cylindric partition with profile $c$, 
and collect the slices ordered as

\begin{align}
\label{PO_slices}
  {}_{1}\Sigma \geq {}_{2}\Sigma \geq \cdots \geq {}_{m}\Sigma > E
\end{align}

% At the moment, we use lower left indices to tell the slices apart,
% because the subscripts indicate the components of the vector partition
% making up the cylindric partition.
% In the proofs, we will use the slices decomposition exclusively,
% and not the vector partition nature of a cylindric partition.
% Then, we will be able to use subscripts in place of
% the awkward lower left indices without confusion.

Now, fix a $j$ with $1 \leq j \leq m$ and consider the slice 
${}_j\Sigma = ( {}_j\sigma^{(1)}, \ldots, {}_j\sigma^{(r)} )$.
All ${}_j\sigma^{(i)}$'s are partitions into 1's, 
so $\vert {}_j\sigma^{(i)} \vert = l({}_j\sigma^{(i)})$.  
It follows from \eqref{PO_slices} that for each fixed $i = 1, \ldots, r$
\begin{align}
\nonumber 
  l( {}_m\sigma^{(i)} ) \geq \cdots \geq l( {}_1\sigma^{(i)} ) \geq 0.   
\end{align}
In fact, the above inequalities are equivalent to \eqref{PO_slices}.  
At the moment, $l( {}_1\sigma^{(i)} ) > 0$ for at least one $i$.  

It is easier to follow the procedure visually.  
We consider the cylindric partition with profile $c = (1,1,1)$
$\Lambda = ( (5, 4), (8, 2), (7, 5, 1) )$ and find its slices in Figure \ref{figSlices}.
\begin{figure}
  \centering
\includegraphics[scale=0.3]{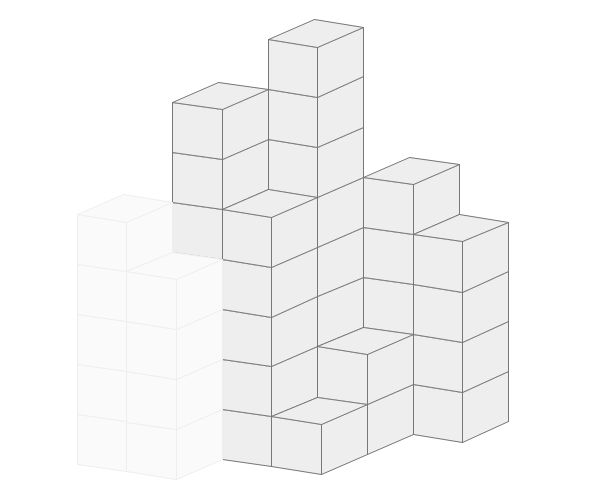}
\raisebox{2cm}{$\longrightarrow$}
\includegraphics[scale=0.3]{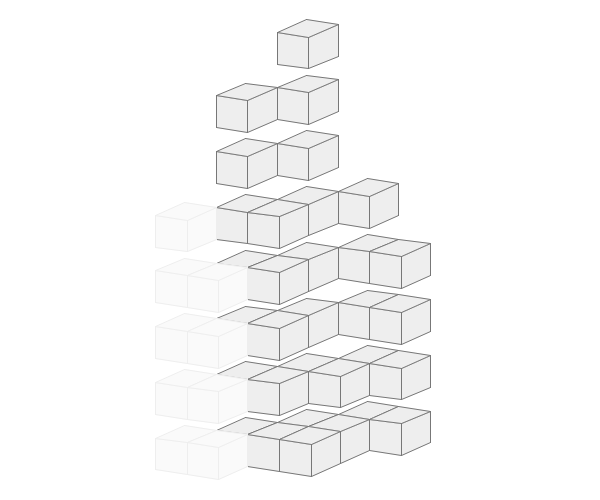}
  \caption{ Slices of a cylindric partition. }
\label{figSlices}
\end{figure}
Since all non-zero entries are ones, 
it is convenient to view the slices as skew diagrams.  
The eight slices are ordered as follows.  
We omit the repeated first row, 
because the profile is clear.  
\begin{align}
\nonumber 
\vcenter{ \hbox { \resizebox{0.4\width}{0.4\height}{ 
  \ydiagram{2+2, 1+2, 0+3} } } } 
\geq 
\vcenter{ \hbox { \resizebox{0.4\width}{0.4\height}{ 
  \ydiagram{2+2, 1+2, 0+2} } } } 
\geq 
\vcenter{ \hbox { \resizebox{0.4\width}{0.4\height}{ 
  \ydiagram{2+2, 1+1, 0+2} } } } 
\geq 
\vcenter{ \hbox { \resizebox{0.4\width}{0.4\height}{ 
  \ydiagram{2+2, 1+1, 0+2} } } } 
\geq 
\vcenter{ \hbox { \resizebox{0.4\width}{0.4\height}{ 
  \ydiagram{2+1, 1+1, 0+2} } } } 
\geq 
\vcenter{ \hbox { \resizebox{0.4\width}{0.4\height}{ 
  \ydiagram{2+0, 1+1, 0+1} } } } 
\geq 
\vcenter{ \hbox { \resizebox{0.4\width}{0.4\height}{ 
  \ydiagram{2+0, 1+1, 0+1} } } } 
\geq 
\vcenter{ \hbox { \resizebox{0.4\width}{0.4\height}{ 
  \ydiagram{2+0, 1+1, 0+0} } } } 
>
E
\end{align}
In the presence of empty rows, 
we can emphasize the profile as the following.  
We rewrite the inequality for the three smallest slices.  
\begin{align}
\nonumber 
\vcenter{ \hbox{ \resizebox{0.4\width}{0.4\height}{
  \ydiagram[*(white)]{ 3+0, 2+1, 1+1 } 
  *[*(darkgray)]{3,3,2}
  } } }
  \geq 
\vcenter{ \hbox{ \resizebox{0.4\width}{0.4\height}{
  \ydiagram[*(white)]{ 3+0, 2+1, 1+1 } 
  *[*(darkgray)]{3,3,2}
  } } }
  \geq 
\vcenter{ \hbox{ \resizebox{0.4\width}{0.4\height}{
  \ydiagram[*(white)]{ 3+0, 2+1, 1+0 } 
  *[*(darkgray)]{3,3,1}
  } } }
  >
\vcenter{ \hbox{ \resizebox{0.4\width}{0.4\height}{
  \ydiagram[*(white)]{ 3+0, 2+0, 1+0 } 
  *[*(darkgray)]{3,2,1}
  } } }
  = E
\end{align}
The dark gray boxes are actually not there, 
and the white boxes are 1's.  

Next, we describe the \emph{shape} of a slice.  
For the sake of this part of discussion, 
the weight per se is immaterial.  
Therefore; we augment all slices 
with the dark gray Young diagram in place for the empty cylindric partition
as in the last displayed chain of containment, 
discard the colors of boxes, 
and end up with an intermediate Young diagram of a partition 
into exactly $r$ parts, where $r$ is the rank.
Then, we delete the last part, 
and subtract this part from all other parts.  
This partition into at most $r-1$ parts is called the \emph{shape} of the slice.
We find it convenient to keep the rank information, 
so we pad these partitions with zeros as necessary.  
For example, in the second to the last containment,
the slices, including the empty slice, 
have shapes
\begin{align}
\nonumber 
  (1, 0), (2, 1), (2, 0), (2, 0), (1, 0), (1, 1), (1, 1), (2, 2), (2, 1),  
\end{align}
respectively.  
% The shape of the empty slice ($(2,1)$ here) 
% will henceforth be called \emph{the shape of zero}.  

As the profile governs the ``left ends'' of slices, 
so does the shapes their ``right ends''.  
The convention has been to represent profiles by compositions (called $c$, $d$, etc.).  
We represented the shapes by partitions (called $\sigma$, $\tau$, etc.).  

Notice that the empty cylindric partition $E$ has shape implied by the profile $c$.  
We will call this the \emph{shape of zero}.  
Shape of zero is unique in any fixed profile.  
In particular, if the profile is $c = (c_1, c_2, \ldots, c_r)$,
then the shape $\sigma$ of the empty partition is 
given by
\begin{align}
\nonumber 
  \sigma = ( c_r + c_{r-1} + \cdots + c_2,
   c_r + c_{r-1} + \cdots + c_3, \ldots, c_r + c_{r-1}, c_r ).
\end{align}
% Considering the possibility that $c_k = \cdots = c_j = 0 \neq c_{j-1}$ 
% for some $j = 2, \ldots, k$, 
% we see that $\sigma$ is a partition with at most $(k-1)$ parts.  
Since $c_1 + c_2 \cdots + c_r = \ell$,
parts of $\sigma$ are at most $\ell$.  
As such, $\sigma$ is one of the $\binom{\ell+r-1}{r-1}$
partitions enumerated by $\begin{bmatrix} \ell+r-1 \\ r-1 \end{bmatrix}$
\cite{TheBlueBook, GEA-E}.
Here, $\begin{bmatrix} n \\ k \end{bmatrix}$ $= \frac{ (q; q)_n }{ (q; q)_k (q; q)_{n-k} }$
is the $q$-binomial coefficient~\cite{TheBlueBook, GEA-E, GR}.
% This is ??? in Gessel and Krattenthaler's paper~\cite{GesselKrattenthaler}.  
It follows that the shapes of zeros are in 1-1 correspondence 
with the set of possible profiles.  
This correspondence carries over to the set of possible shapes.  

As long as we know the level $\ell$, we can switch back and forth 
between compositions and partitions, 
as done in the above example.  
Either may be preferred, and it will be clear from the context which one is used.  

% \begin{align*}
%   \ydiagram[*(white)]{ 3+0, 2+0, 1+0 } 
%   *[*(darkgray)]{3,2,1} 
%   \; 
%   \ydiagram[*(white)]{ 3+0, 2+0, 1+1 } 
%   *[*(darkgray)]{3,2,2}
%   \; 
%   \ydiagram[*(white)]{ 3+0, 2+1, 1+0 } 
%   *[*(darkgray)]{3,3,1}
%   \; 
%   \ydiagram[*(white)]{ 3+1, 2+0, 1+0 } 
%   *[*(darkgray)]{4,2,1} \\
%   \ydiagram[*(white)]{ 3+0, 2+1, 1+1 } 
%   *[*(darkgray)]{3,3,2}
%   \; 
%   \ydiagram[*(white)]{ 3+1, 2+0, 1+1 } 
%   *[*(darkgray)]{4,2,2}
%   \; 
%   \ydiagram[*(white)]{ 3+1, 2+1, 1+0 } 
%   *[*(darkgray)]{4,3,1}
%   \; 
%   \ydiagram[*(white)]{ 3+0, 2+1, 1+2 } 
%   *[*(darkgray)]{3,3,3} \\
%   \ydiagram{ 2+1, 1+1, 0+1 } 
%   \; 
%   \ydiagram[*(white)]{ 3+2, 2+0, 1+1 } 
%   *[*(darkgray)]{5,2,2}
%   \; 
%   \ydiagram[*(white)]{ 3+1, 2+2, 1+0 } 
%   *[*(darkgray)]{4,4,1}
%   \; 
%   \ydiagram{ 2+1, 1+1, 0+2 } \\
%   \ydiagram{ 2+1, 1+2, 0+1 }
%   \; 
%   \ydiagram{ 2+2, 1+1, 0+1 }
%   \; 
%   \ydiagram{ 2+1, 1+2, 0+2 }
%   \; 
%   \ydiagram{ 2+2, 1+1, 0+2 } \\
%   \ydiagram{ 2+2, 1+2, 0+1 }
% \end{align*}

We recall once more that the profile $c$ is fixed.  
We order the slices by inclusion, and draw their 
Hasse diagram~\cite{Birkhoff-Lattices}
sideways in such a way that slices which have the same weight are vertically aligned, 
and slices having the same shape are horizontally aligned.  
Moreover, the shapes are lexicographically ordered.  
For convenience, we include the shape of zero.  
This carries the profile information, as well.  
For the profile $c = (1,1,1)$, the beginning of this Hasse diagram
is shown in Figure \ref{figHasse}.
\begin{figure}
  \centering
    \includegraphics[scale=0.8]{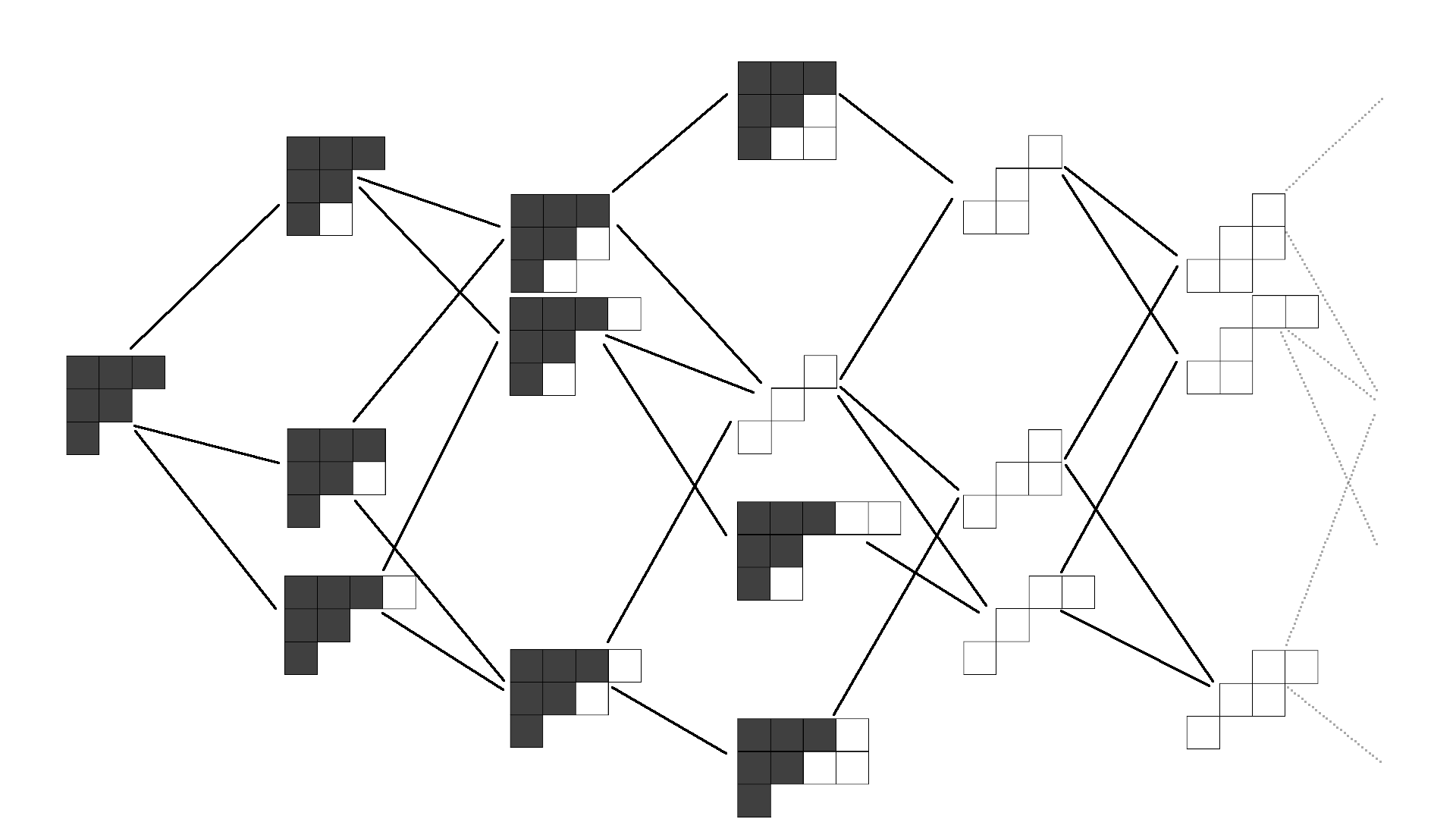}
  \caption{The Hasse diagram corresponding to the profile $c=(1,1,1)$. }
\label{figHasse}
\end{figure}
It is not difficult to see that there can be at most one slice
having a given shape and weight.  
For instance, the slices with weight 2 
in the above examples can only have shapes $(1, 1)$, $(2, 0)$, and $(3, 2)$; 
but not, say, $(2, 1)$.  
We will represent any cylindric partition in this diagram 
with the multiplicities of its slices.  
To do that, we replace the slices with nodes, 
indicate the weights on the horizontal direction, 
and the shapes on the vertical direction. 
The last example will look like Figure \ref{figPathDiag}.
\begin{figure}
  \centering
    \includegraphics[scale=0.1]{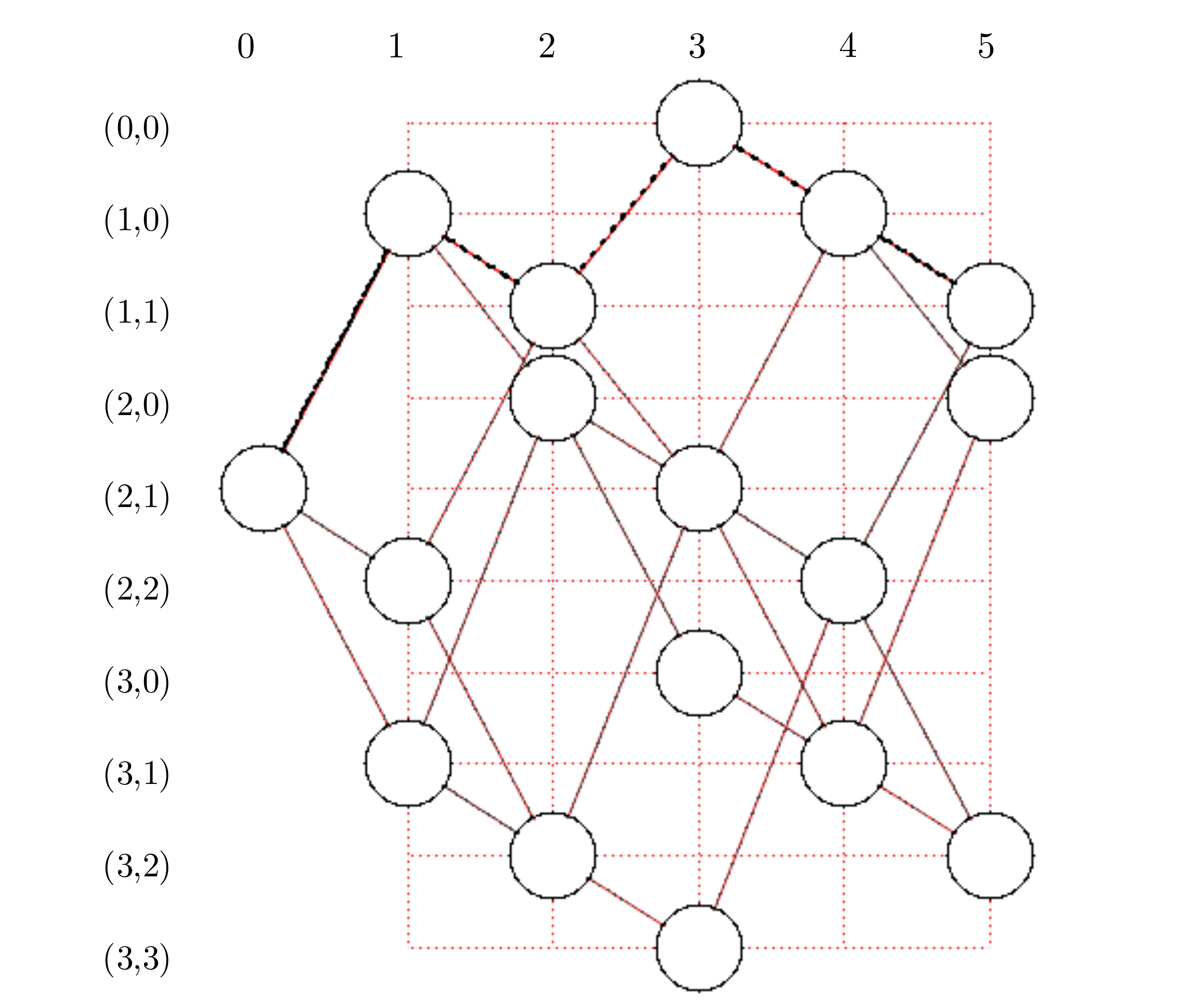}
    \raisebox{3cm}{\ldots}
  \caption{The path diagram corresponding to the profile $c=(1,1,1)$. }
\label{figPathDiag}
\end{figure}
We will call this the \emph{path diagram} associated with the profile $c$.  
The reason for the term is that we will draw and examine paths on these diagrams.  
The nodes will be used to record the (non-zero) multiplicities of the slices.  
For example, the cylindric partition $\Lambda = ( (5, 4), (8, 2), (7, 5, 1) )$ 
with profile $c=(1,1,1)$, the slice decomposition of which is shown before, 
will be registered as in Figure \ref{figCylPtnInPathDiag}.
\begin{figure}
  \centering
  \includegraphics[scale=0.1]{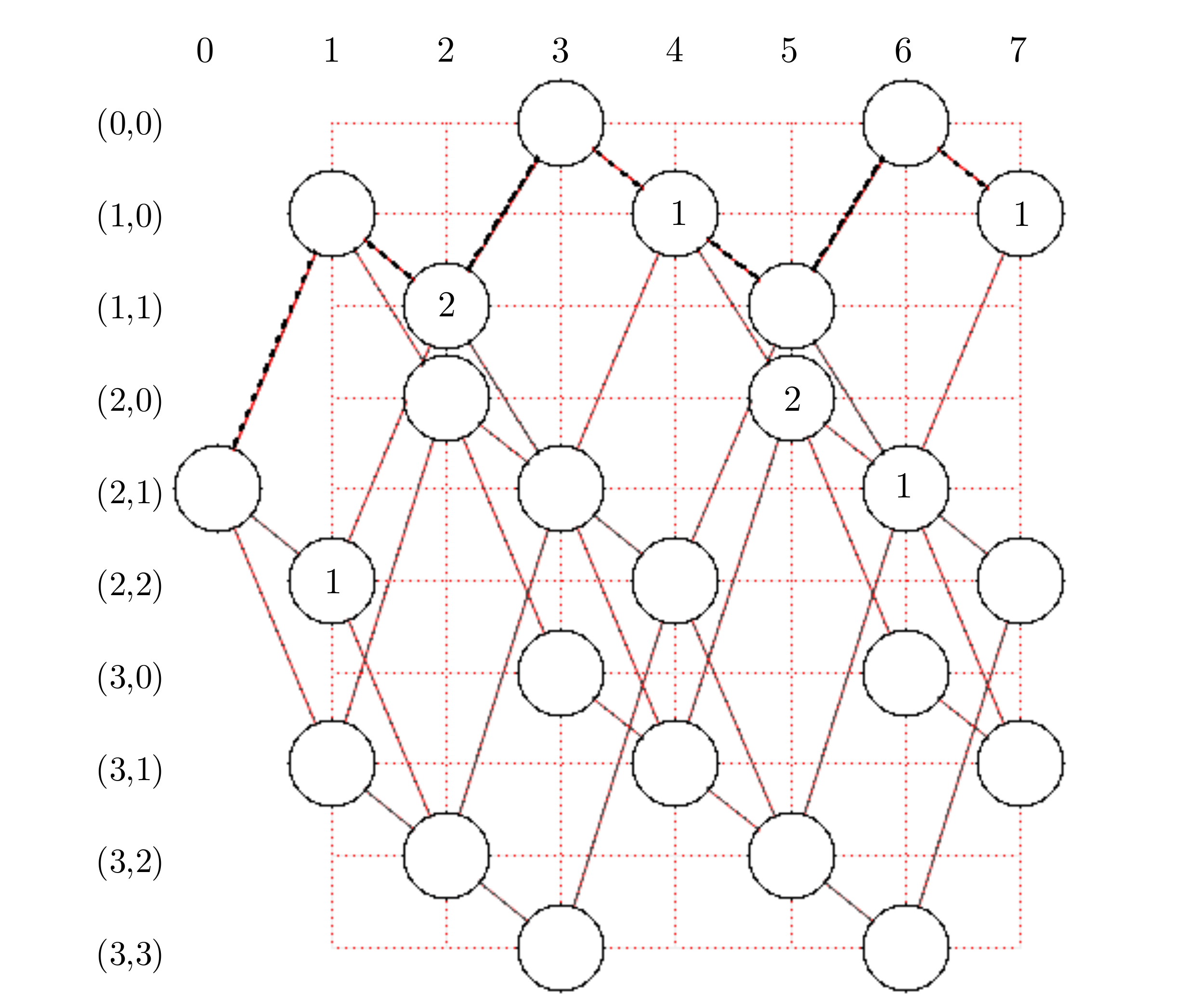}
  \raisebox{3cm}{\ldots}
  \caption{A cylindric partition shown in the path diagram. }
\label{figCylPtnInPathDiag}
\end{figure}

It is straightforward to extend the definition 
of an \emph{outer corner} of a Young diagram~\cite{GEA-E, CW} 
to slices of a cylindric partition with profile $c$.  
All one has to do is to adjust according to the interplay between the 
top and the bottom rows.  
Then, one can see that there is an edge between 
slice ${}_{1}\Sigma$ and ${}_{2}\Sigma$ 
if and only if ${}_{2}\Sigma$ can be obtained from ${}_{1}\Sigma$ 
by an addition of an outer corner.  
It is also clear that 
$\vert {}_{2}\Sigma \vert - \vert {}_{1}\Sigma \vert = 1$ 
in this case.  
Since adding outer corners concerns shapes only, 
we can have a directed graph with shapes as nodes.  
For each possible addition of an outer corner will an edge be drawn.  
When one highlights the shape of zero, 
this finite graph compactly carries the same information as the path diagram.
The described graph will be called the \emph{shape transition graph} 
of slices of cylindric partitions with profile $c$.  
For instance, the shape transition graph 
associated with profile $c=(1,1,1)$ is as in Figure \ref{figShapeTransitionGraph}.
\begin{figure}
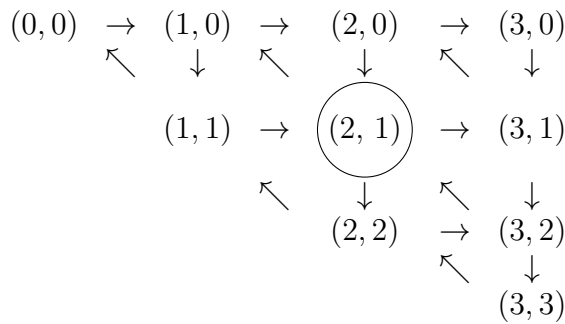

  \centering
  $\begin{array}{ccccccc}
  (0, 0) & \rightarrow & (1, 0) & \rightarrow & (2, 0) & \rightarrow & (3, 0) \\
  & \nwarrow & \downarrow & \nwarrow & \downarrow & \nwarrow & \downarrow \\
  & & (1, 1) & \rightarrow & \circled{(2, 1)} & \rightarrow & (3, 1) \\
  & & & \nwarrow & \downarrow & \nwarrow & \downarrow \\
  & & & & (2, 2) & \rightarrow & (3, 2) \\
  & & & & & \nwarrow & \downarrow \\
  & & & & & & (3, 3)
  \end{array}$
  \caption{The shape transition graph associated with profile $c=(1,1,1)$.  }
\label{figShapeTransitionGraph}
\end{figure}
The construction of the shape transition graph, with a little extra effort, 
also shows that the path diagram is eventually periodic
in the increasing weight direction with period equal to the rank $r$.
It is difficult, if not impossible, to record the weights 
in a shape transition graph.  
Therefore, we still need path diagrams.  
We will use the shape transition graphs for other purposes
in Sections \ref{secDist} and \ref{secCDU}.

It must be mentioned that one cannot avoid crossing edges in 
the path diagrams for rank $r \geq 3$
and in the shape transition graphs for rank $r \geq 4$,
but for small levels $\ell$.  
However, the rank $r = 2$ case has distinguishedly nice path diagrams
and shape transition graphs.  
For instance, the path diagram associated with profile $c = (4, 3)$
is shown in Figure \ref{figPathDiag-4-3}.
\begin{figure}
  \centering
  \includegraphics[scale=0.1]{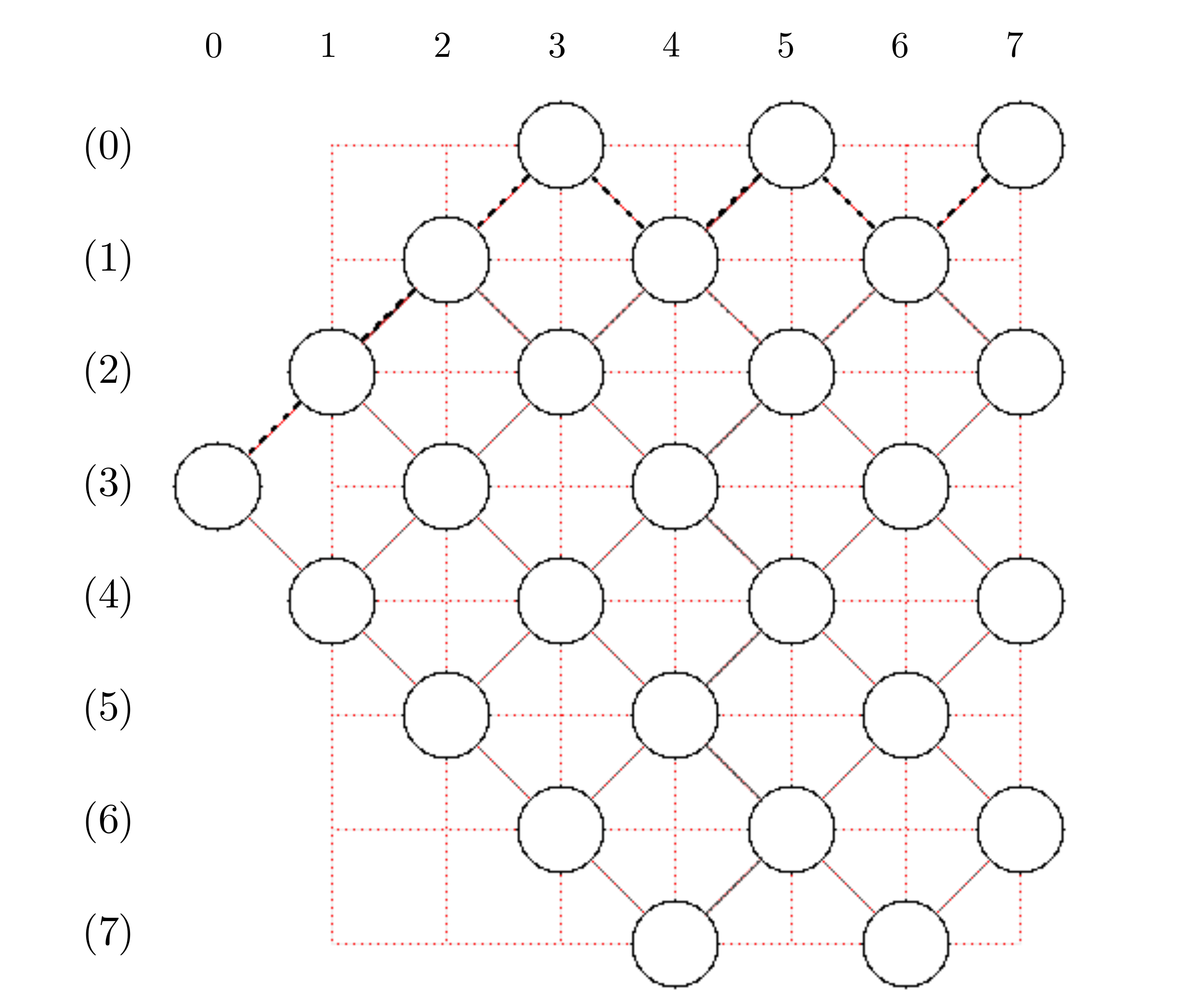}
  \raisebox{3cm}{\ldots}
  \caption{The path diagram corresponding to the profile $c=(4,3)$. }
\label{figPathDiag-4-3}
\end{figure}
The shape transition graph corresponding to profile $c=(4,3)$ is
\begin{align}
\nonumber 
  (0) {\curvearrowleft \above 0pt \text{\rotatebox[origin=c]{180}{$\curvearrowleft$}} }
  (1) {\curvearrowleft \above 0pt \text{\rotatebox[origin=c]{180}{$\curvearrowleft$}} }
  (2) {\curvearrowleft \above 0pt \text{\rotatebox[origin=c]{180}{$\curvearrowleft$}} }
  \circled{(3)} 
    {\curvearrowleft \above 0pt \text{\rotatebox[origin=c]{180}{$\curvearrowleft$}} }
  (4) {\curvearrowleft \above 0pt \text{\rotatebox[origin=c]{180}{$\curvearrowleft$}} }
  (5) {\curvearrowleft \above 0pt \text{\rotatebox[origin=c]{180}{$\curvearrowleft$}} }
  (6) {\curvearrowleft \above 0pt \text{\rotatebox[origin=c]{180}{$\curvearrowleft$}} }
  (7). 
\end{align}

To find the cylindric partition with profile $c$, shape $\tau$, 
and the minimum possible weight, we do the following.  
Suppose the profile $c$ corresponds to shape $\sigma$ (shape of zero).  
We want to cut out the shape $\sigma$ from inside $\tau$.  
If $\tau_j \geq \sigma_j$ for all $j = 1, 2, \ldots, (r-1)$, no problem.
We can excise the squares making up $\sigma$.
The weight of this minimal partition is $\vert \tau \vert - \vert \sigma \vert$.  
If not, we retry after appending a column of height $k$ on the left of $\tau$.  
Parts of $\tau$ are practically increased by 1, 
but in order to retain the shape, an extra square is put in the $r$th row.
Then, we check if $\tau_j+1 \geq \sigma_j$ for all $j = 1, 2, \ldots, (r-1)$.
If so, again, we're done.  
In this case, the weight of this minimal cylindric partition is 
$r + \vert \tau \vert - \vert \sigma \vert$.
If not, we append a second column of height $r$ on the left of $\tau$ etc.
It is clear that we need the smallest non-negative integer $h$ such that 
$\tau_j + h \geq \sigma_j$ for all $j = 1, 2, \ldots, (r-1)$.
Thus, $h = $ $\mathrm{max}\left\{ 0, \sigma_1 - \tau_1 \right.$, 
\ldots, $\left. \sigma_{r-1} - \tau_{r-1} \right\}$.
Then, the minimal partition has weight $rh + \vert \tau \vert - \vert \sigma \vert$.
We define this as $\Delta(\sigma, \tau)$, 
or by switching to the corresponding compositions, $\Delta(c, d)$.  
To be precise, for compositions $c = (c_1, \ldots, c_r)$ and $d = (d_1, \ldots, d_r)$,
\begin{align}
\nonumber 
  \Delta(c, d) & = (r-1)(d_r - c_r) + (r-2)(d_{r-1} - c_{r-1}) + \cdots + (d_2 - c_2)  \\
\label{defDeltaOp}
    & + r \; \mathrm{max}\left\{ 0, c_r - d_r, (c_r + c_{r-1}) - (d_r + d_{r-1}),
    \ldots, (c_r + \cdots + c_2) - ( d_r + \cdots + d_2 ) \right\}.
\end{align}
One can readily verify that $\Delta(c,c) = 0$, 
and $\Delta(c,d)$ does not necessarily equal to $\Delta(d,c)$.  
The careful reader will have noticed that this is nothing but a micromanagement 
on shapes (of plane partitions) in Gessel and Krattenthaler~\cite{GesselKrattenthaler}.  
For example, for rank $r = 3$, 
$\Delta((1,1,1),$ $(0,0,2)) = 1$ and $\Delta((0,0,2),$ $(1,1,1)) = 2$ 
is shown below, where the excess squares are marked by $\times$.  
\begin{align}
\nonumber 
\vcenter{ \hbox { \resizebox{0.4\width}{0.4\height}{ 
    \ytableausetup{notabloids}
    \begin{ytableau}
    \none & \none & \\
    \none &  & \times \\
    \phantom{0}
    \end{ytableau}
   } } },  
\vcenter{ \hbox { \resizebox{0.4\width}{0.4\height}{ 
    \ytableausetup{notabloids}
    \begin{ytableau}
    \none &  &  & \times \\
    \none &  &  \\
    \phantom{0} & \times 
    \end{ytableau}
   } } }. 
\end{align}

\section{Decomposition of unrestricted cylindric partitions}
\label{secUnrestrictedCylPtn}

Although we will prove the general case later, 
we introduce the approach on cylindric partitions with profile $(2, 1)$.  
Rank 2 cylindric partitions yield clearer pictures, 
and the proofs can mostly be done on the path diagrams.  
This will no longer be the case for rank $r \geq 3$ cylindric partitions.
Also, the profile $(2, 1)$ case immediately generalizes to the 
profile $(a, b)$ case.  
% One does not even have to indicate the places to update in the proof.  
Lastly, $(2, 1)$ is a tribute 
to Rogers-Ramanujan identities~\cite{RR, Borodin}.  

\begin{prop} \label{c=(2,1)RR1}
	Let $\Lambda$ be a cylindric partition with profile $c=(2,1)$.
	Then, it can be decomposed into a unique pair of partitions $(\mu,\beta)$ 
	such that $\mu$ is an ordinary partition 
	and $\beta$ is a partition in which adjacent parts differ by at least $2$. 
	Conversely, if $(\mu,\beta)$ is a pair of partitions 
	as described in the previous sentence, 
	then it corresponds to a unique cylindric partition of profile $c=(2,1)$. 
	Hence, we have:
	\begin{equation} \label{gen.func.c=(2,1)}
		F_{(2,1)}(1,q)=\frac{1}{(q;q)_\infty} 
          \left( \sum_{n\geq0}\frac{q^{n^2}}{(q;q)_n}\right) .
	\end{equation}
\end{prop}

\begin{proof}

Given an arbitrary but fixed cylindric partition $\Lambda$ with profile $(2, 1)$, 
decompose $\Lambda$ into its slices, and tally the slices in the path diagram.  
The path diagram associated with the profile $c=(2, 1)$
is shown in Figure \ref{figRR_diag0}.
Notice that $(1)$ is the shape of zero, 
slices of odd weight may have shapes $(0)$ or $(2)$, 
and slices of even weight may have shapes $(1)$ or $(3)$.  
This is not counterintuitive if we remember that 
slices are in fact skew shapes, and the cutout partition $(1)$ 
from the left end switches parities.  
\begin{figure}
  \centering
  \includegraphics[scale=0.1]{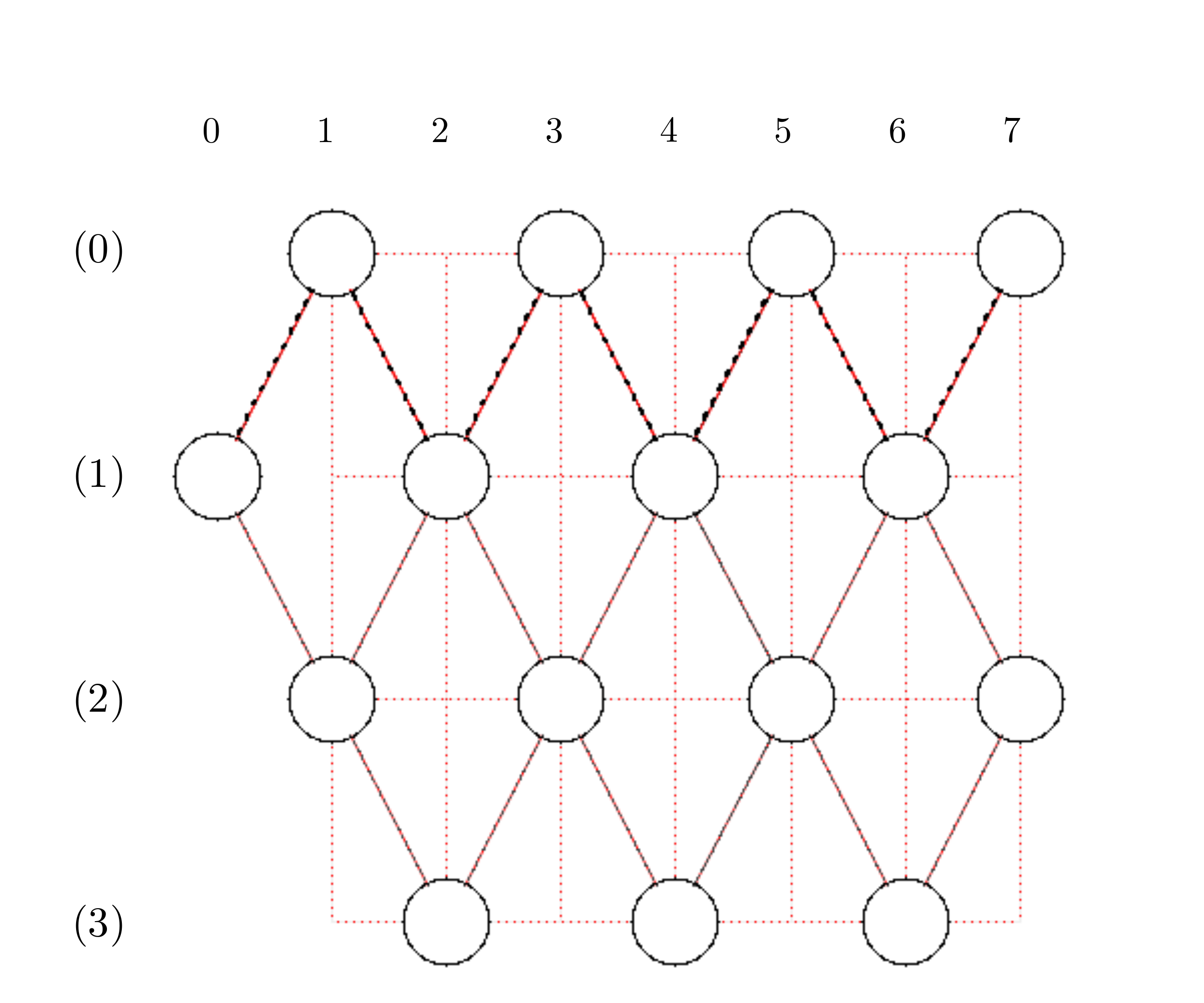}
  \raisebox{3cm}{\ldots}
  \caption{The path diagram corresponding to the profile $c=(2, 1)$. }
\label{figRR_diag0}
\end{figure}
% We remind the reader that the indices of slices indicate individual slices,
% rather than the components of the vector partition that they are.
Let ${}_{1}\Sigma$, \ldots, ${}_{m}\Sigma$ be the appearing slices.
They are distinct, and they may be used one or more times.  
In addition, take ${}_{0}\Sigma$ to be the zero slice,
and ${}_{m+1}\Sigma$ to be a slice of shape $(0)$ ahead of ${}_{m}\Sigma$,
and connected to it in the path diagram.  
For all $j = 0, 1, \ldots, m$, temporarily draw the mesh of paths 
connecting ${}_{j}\Sigma$ to ${}_{j+1}\Sigma$,
and declare the path between ${}_{j}\Sigma$ and ${}_{j+1}\Sigma$ the upper envelope of it.
This is depicted in Figure \ref{figRR_diag_UE}.
The successive non-empty nodes are indicated by asterisks.
The mesh of paths is the union of the gray and the blue edges.
The upper envelope is the blue path.
\begin{figure}
  \centering
  \includegraphics[scale=0.1]{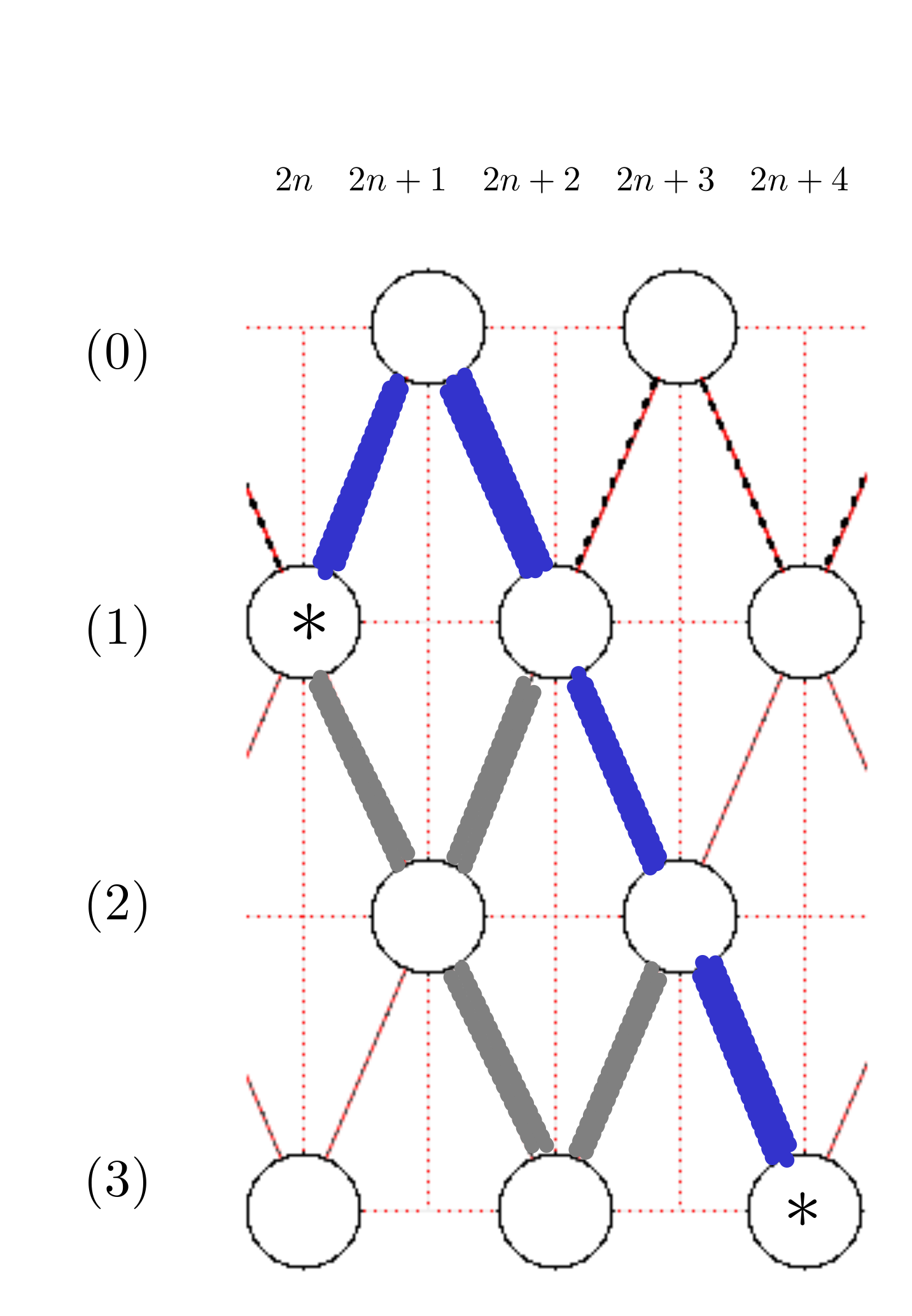}
  \caption{The upper envelope of a mesh of paths between two nodes.  }
\label{figRR_diag_UE}
\end{figure}
To the right of ${}_{m+1}\Sigma$,
draw the path oscillating between shapes $(0)$ and $(1)$.  
This path is uniquely determined by $\Lambda$, 
but not vice versa.  
Different cylindric partitions may produce the same path.  
Any such path will have a unique node of weight $n$ for each $n$.  

For example, the cylindric partition with profile $c=(2,1)$
\begin{align}
\nonumber 
  \Lambda = \begin{array}{ccccccccc}
      & & & 15 & 15 & 10 & 10 & 6 & 5 \\ 
      & & 18 & 13 & 6 & 6 & & & \\
      \textcolor{lightgray}{15} & \textcolor{lightgray}{15} 
        & \textcolor{lightgray}{10} & \textcolor{lightgray}{10} 
        & \textcolor{lightgray}{6} & \textcolor{lightgray}{5} & & & 
    \end{array}
\end{align}
has the slice decomposition
\begin{align}
\nonumber
	3 \times \vcenter{ \hbox{ \resizebox{0.4\width}{0.4\height}{
		\begin{ytableau}
			*(darkgray) & *(darkgray) & *(darkgray) \\
			*(darkgray) & *(darkgray) &
		\end{ytableau}
		} } }
	+ 2 \times \vcenter{ \hbox { \resizebox{0.4\width}{0.4\height}{
		\ydiagram{1+2, 0+1} } } }
	+ 3 \times \vcenter{ \hbox { \resizebox{0.4\width}{0.4\height}{
		\ydiagram{1+2, 0+2} } } }
	+ 4 \times \vcenter{ \hbox { \resizebox{0.4\width}{0.4\height}{
		\ydiagram{1+4, 0+2} } } }
	+ \vcenter{ \hbox { \resizebox{0.4\width}{0.4\height}{
		\ydiagram{1+5, 0+4} } } }
	+ 5 \times \vcenter{ \hbox { \resizebox{0.4\width}{0.4\height}{
		\ydiagram{1+6, 0+4} } } }.
\end{align}

It will look like Figure \ref{figRR_diag1} in the path diagram,
along with the described path.  
The slices ${}_{0}\Sigma$ and ${}_{m+1}\Sigma = {}_{7}\Sigma$
are indicated by $\times$'s.
\begin{figure}
  \centering
  \includegraphics[scale=0.1]{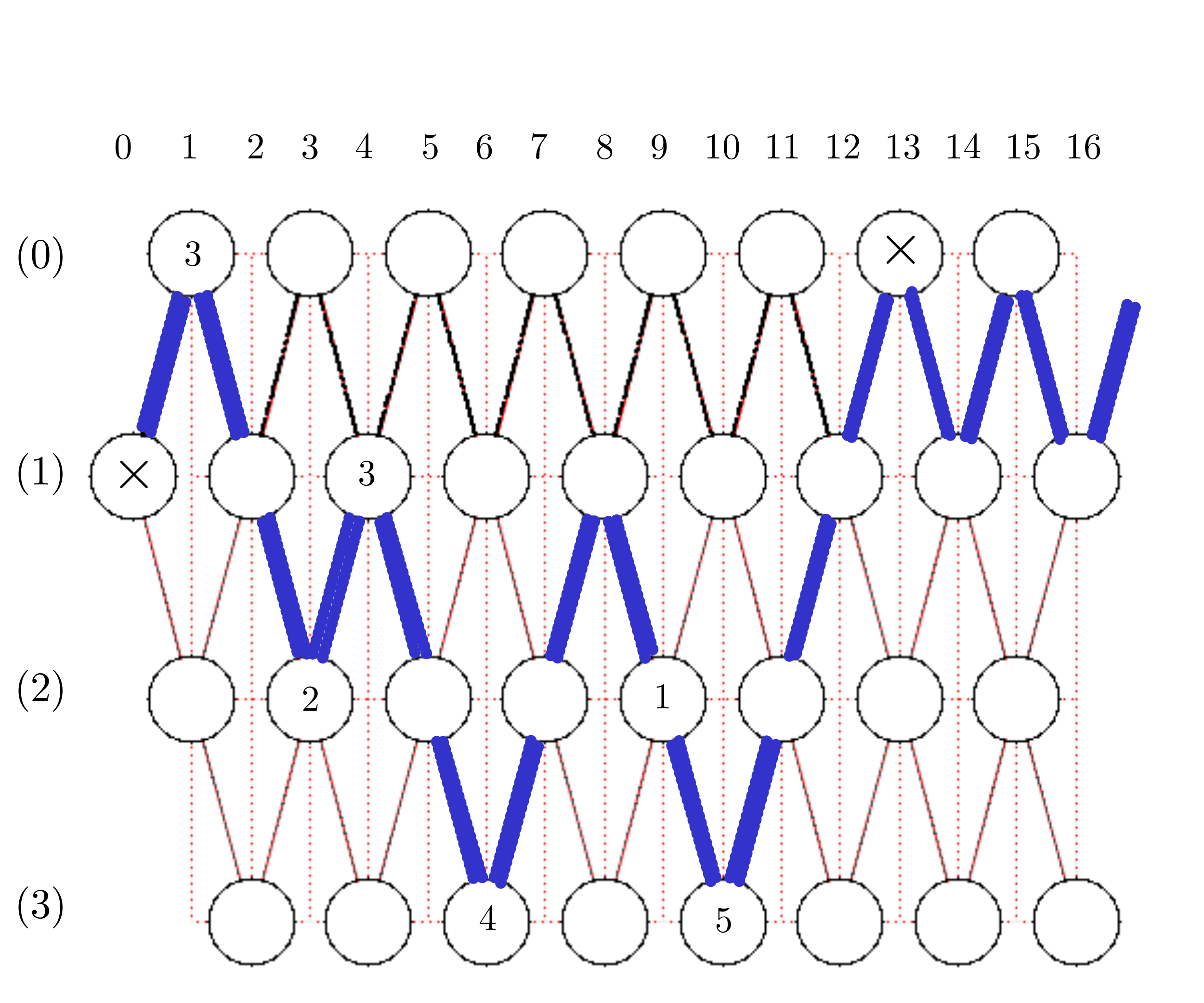}
  \caption{The path determined by a cylindric partition.  }
\label{figRR_diag1}
\end{figure}

When the path passes through a slice of shape $(2)$, 
that slice does not have to appear in $\Lambda$.  
However, if the slice corresponds to a downward spike (a local minimum, or a stalactite) 
in the drawn path, then it has to appear in $\Lambda$.  
This is because upper envelope of paths connecting two nodes 
cannot contain downward spikes with shapes $(2)$ or $(3)$.  
Slices of shape $(3)$ in the drawn path are necessarily downward spikes, 
therefore they must appear in $\Lambda$.
Please compare the red and the blue path in Figure \ref{figRR_diag_spikes},
where the asterisks indicate the non-empty nodes.
\begin{figure}
  \centering
  \includegraphics[scale=0.1]{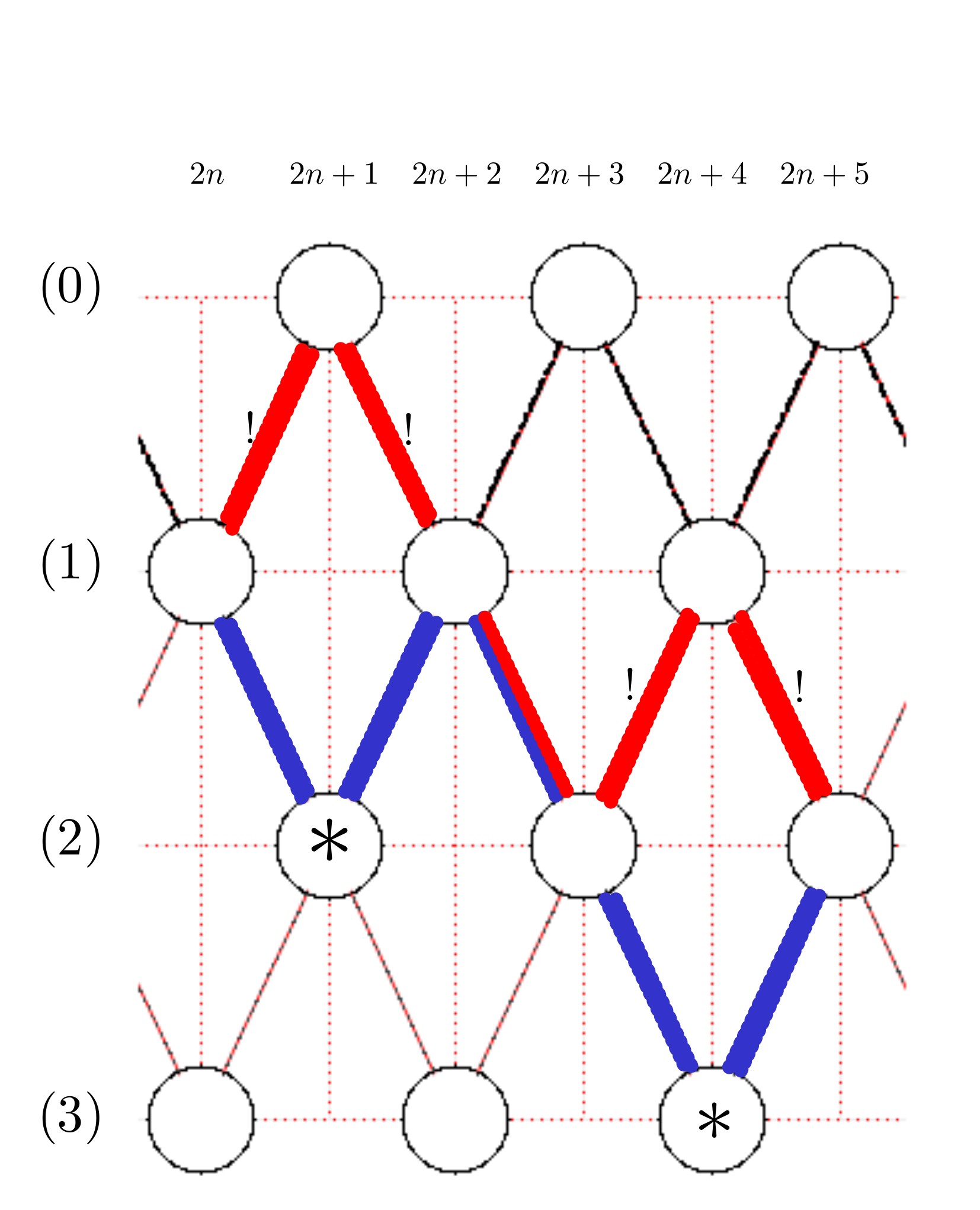}
  \caption{Downward spikes or stalactites.  }
\label{figRR_diag_spikes}
\end{figure}

In both of the only two possibilities, 
if the slice of shape $(2)$ or $(3)$ is absent in $\Lambda$, 
then the upper envelope connecting the slices before and after 
cannot pass through the nodes of shape $(2)$ or $(3)$.  
Also, % in the picture on the right, 
if the slice with shape $(3)$ appears in $\Lambda$, 
the path necessarily passes through the neighboring slices of shape $(2)$, 
regardless of whether they appear in $\Lambda$.  
In summary, downward spikes of shape $(2)$ or $(3)$ 
correspond to non-empty nodes in the drawn path, 
i.e. $\Lambda$ possesses those slices.

The slices of shape $(1)$ are not included in this discussion, 
because they may make downward spikes 
even if $\Lambda$ does not have the corresponding slice.  
This is most clearly seen in the part of the drawn path after ${}_{m+1}\Sigma$.

We next argue that the indicated downward spikes of shapes $(2)$ or $(3)$ 
completely determine the infinite path drawn above; 
in the sense that if we forgot about other slices in $\Lambda$, 
and drew the path as described above, we would have drawn the same path.  

For a contradiction, assume otherwise.  
That is, we are given a cylindric partition $\Lambda$ with profile $(2, 1)$, 
and drew the corresponding infinite path in the path diagram.  
Then, we erased everything in $\Lambda$ but one copy of each of the slices 
that are the downward spikes of shapes $(2)$ or $(3)$.  
Finally, we redrew the path corresponding to this reduced $\Lambda$ 
to end up with a different path.  

In this situation, there is a slice with the smallest weight 
up to which both paths agree, and after which the two paths differ.  
This slice must have shape $(1)$ or $(2)$, 
as slices with shapes $(0)$ or $(3)$ have only one outgoing edge.  
In the former possibility, one of the paths misses a downward spike 
of shape $(2)$ or $(3)$ registered by the other.  
This is impossible either because we did not delete all shapes 
corresponding to such nodes, 
or because the path cannot go through such slices unless $\Lambda$ has them.  
Please compare the blue, red and the gray paths
in the picture on the left in Figure \ref{figRR_diag_after}.
\begin{figure}
  \centering
  \includegraphics[scale=0.1]{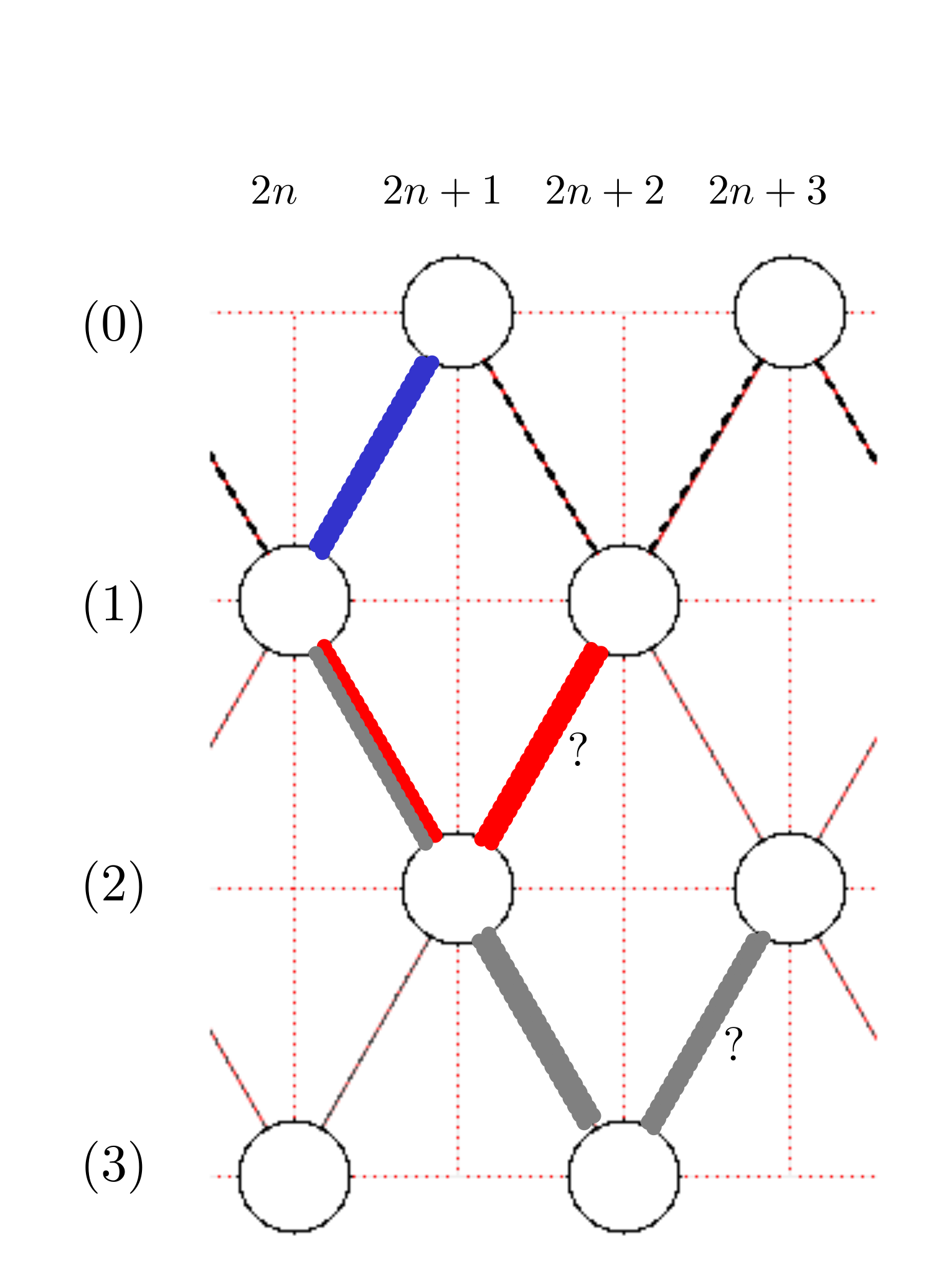}
  \hspace{10mm}
  \includegraphics[scale=0.1]{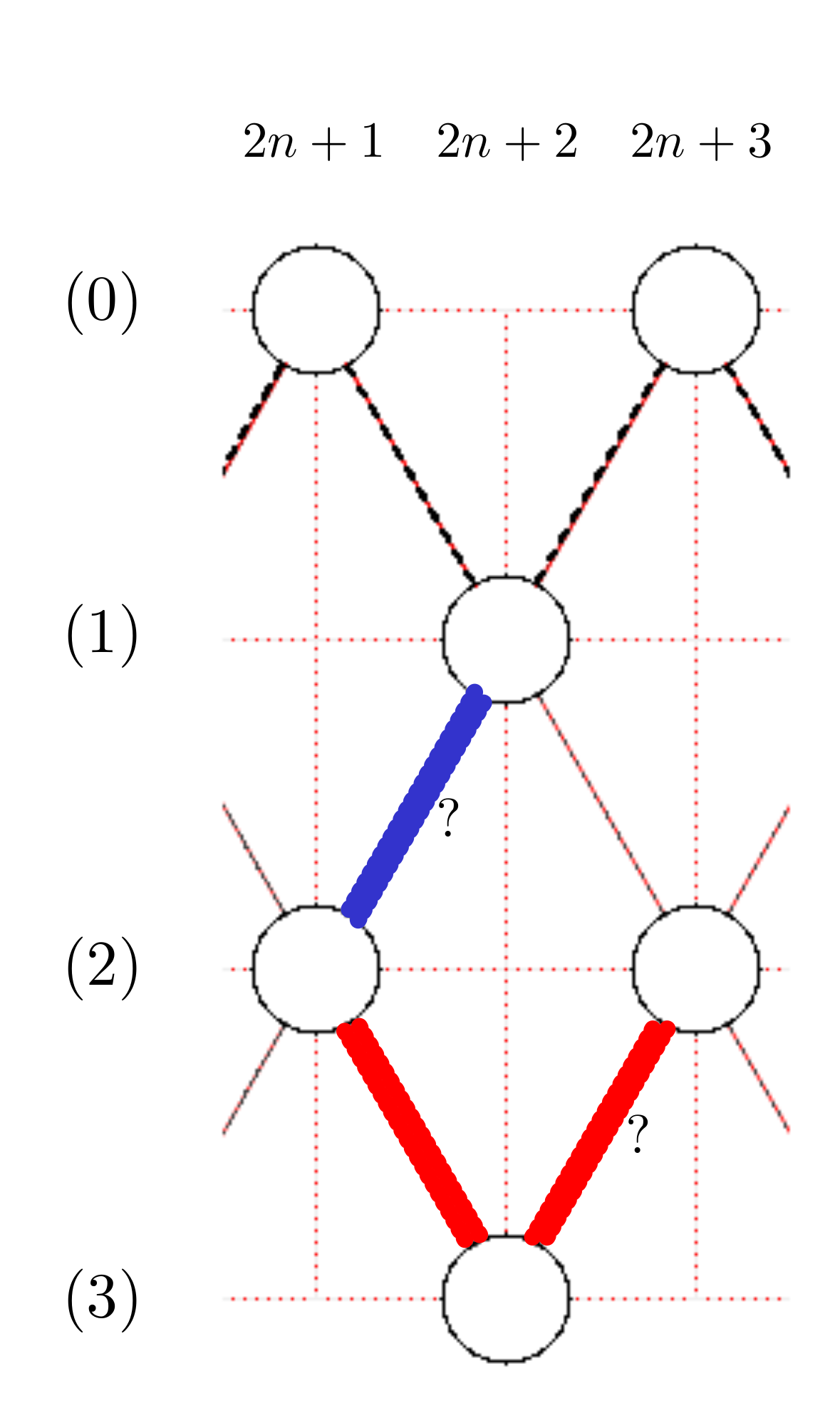}
  \caption{Diverging paths after slices of shapes $(1)$ or $(2)$.  }
\label{figRR_diag_after}
\end{figure}
In the latter possibility, we lose a downward spike of shape $(3)$, 
or create a downward spike of shape $(3)$.  
This is likewise impossible,
please compare the blue and the red paths
in the picture on the right in Figure \ref{figRR_diag_after}.
% \begin{figure}
%   \centering
%   \includegraphics[scale=0.1]{RR_diag_after2r.png}
%   \caption{Diverging paths after a slice of shape $(2)$.  }
% \label{figRR_diag_after2}
% \end{figure}

Therefore, each cylindric partition with profile $(2, 1)$ 
corresponds to an infinite path on the path diagram 
which is completely determined by its downward spikes of shape $(2)$ or $(3)$.  
Moreover, the slices corresponding to these slices appear in $\Lambda$.  
We call these slices \emph{pivots}.  
We denote pivot slices by $\Pi$'s rather than $\Sigma$'s.
A more general, but still consistent, definition of a pivot 
will be developed in the proof of Theorem \ref{thmCylPtnVsPtnPairsFullCase}
and given afterwards.

Next, please observe that two pivots cannot have consecutive weights 
(in the profile $(2, 1)$).  
This is because each pivot must be preceded by a downward edge, 
and followed by an upward edge.  

Take one slice each from each of the pivot slices or pivot nodes.  
Their weights will constitute $\beta$.  
The weights of the remaining slices will constitute $\mu$.  
By construction, $\beta$ has distinct and non-consecutive parts, 
and $\mu$ is an unrestricted partition.  
Also, $\beta$ may have parts as small as 1, 
because the infinite path drawn as in the beginning of the proof 
may have a pivot slice having weight as small as 1.  
Continuing with the above example, we find:
\begin{align}
\nonumber
	& \beta = 10, 6, 3, \\
\nonumber
	& \mu = 10, 10, 10, 10, 9, 6, 6, 6, 4, 4, 4, 3, 3, 1, 1, 1.
\end{align}

Conversely, given a pair $(\mu, \beta)$ in which 
$\mu$ is an unrestricted partition, 
and $\beta$ is a partition into distinct and non-consecutive parts, 
we will construct a unique cylindric partition $\Lambda$ with profile $(2, 1)$.  
First, draw the path diagram of profile $(2, 1)$.  
Then, tally parts of $\beta$ in nodes of shapes $(2)$ or $(3)$ 
corresponding to their size.  
Even parts will necessarily have shape $(3)$, 
and odd parts will have shape $(2)$.  

Next, draw the path which is the upper envelope of the mesh of paths 
between the zero and the smallest non-empty slice, 
that between the smallest and the second smallest 
slice, \ldots; that between the largest non-empty node 
and the node with shape $(0)$
having the smallest weight greater than the largest part of $\beta$, 
and connected to the slice with the greatest weight.  
Pad this path after this last slice with shape $(0)$
with oscillation between shapes $(0)$ and $(1)$ 
to make it an infinite path.  
This path can have no other pivots than the already registered non-empty nodes 
that emanated from the parts of $\beta$ by construction.  

After this, we add the count of parts of $\mu$ 
in the nodes with respective weights.  
This is uniquely possible if we remember that the constructed infinite path 
has exactly one node of each weight $n \in \mathrm{N}$.  

To finish the proof, 
we notice that both cylindric partitions $\Lambda$ with profile $(2, 1)$ 
and the pairs $(\mu, \beta)$ in which $\mu$ is an unrestricted partition 
and $\beta$ is a partition into distinct and non-consecutive parts 
are shown to correspond to unique admissible infinite paths in the path diagram 
together with number of slices in the path.  
By an admissible path, we mean that only finitely many nodes 
in the path diagram have positive slice counts in them, 
non-empty nodes are connected via upper envelopes, 
consequently pivot nodes are non-empty,
and the paths are completely determined by the pivot positions.  
Then, the transformations are seen to be one-to-one.  
They are also onto, 
and we made the constructions so that they are inverses of each other.

\end{proof}

We note that we can write the generating function of cylindric partitions of profile $c=(2,1)$ by using the formula \eqref{BorodinProd} and it is given by
\begin{align}
% \label{gen.func.c=(2,1)Borodin}
\nonumber
	F_{(2,1)}(1,q)=\frac{1}{(q;q)_\infty} \times \frac{1}{(q;q^5)_\infty(q^4;q^5)_\infty}.
\end{align}

If we compare Proposition \ref{c=(2,1)RR1} and the above identity,
we see one of the famous Rogers-Ramanujan identities \cite{TheBlueBook}, namely
\begin{align}
\nonumber
	\sum_{n\geq0}\frac{q^{n^2}}{(q;q)_n}= \frac{1}{(q;q^5)_\infty(q^4;q^5)_\infty}.
\end{align}

Some remarks are in order.  
The above proof has ideas that will be elaborated on in the discussion below.
If $\beta = \varepsilon$, the empty partition, 
then the resulting infinite path will be called the default path.  
No node in the default path can be a pivot.
This is shown as the black path in Figure \ref{figRR_remark_r}.
Each pivot requires a deviation from the default path.  
The path branches to the pivot ``as late as possible'', 
and returns back to the default path ``as quickly as possible'', 
as shown as the red path in Figure \ref{figRR_remark_r},
the non-empty node corresponding to a pivot slice denoted by a red asterisk.
This observation will carry over to cylindric partitions with greater rank.  
\begin{figure}
  \centering
 \includegraphics[scale=0.1]{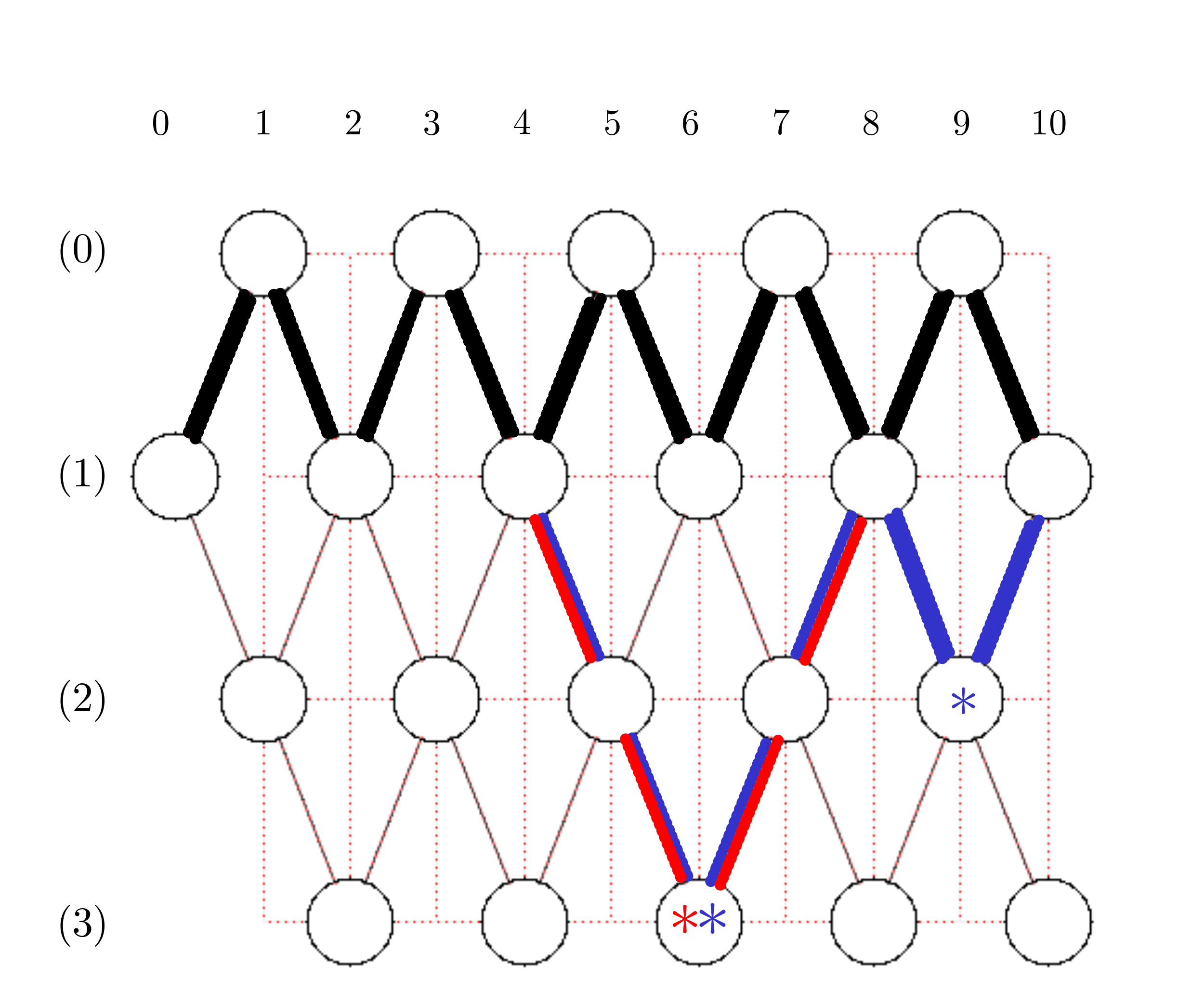}
  \caption{ Leaving and coming back to the default path.  }
\label{figRR_remark_r}
\end{figure}
Shape $(a)$ is declared smaller than shape $(b)$ if $a < b$.  
Between two successive pivots, 
the upper envelope of the mesh of paths connecting them 
amounts to picking the edge towards a smaller shape, or a ``left turn'', whenever possible.
This is shown as the blue path in Figure \ref{figRR_remark_r},
the non-empty nodes corresponding to successive pivot slices indicated by blue asterisks.
For example, in going from the $(1)$ shaped $4$ to the $(3)$ shaped 6, 
one has to go through the $(2)$ shaped 5.  
So, the $(2)$ shaped 5 gives no extra information about the path.  
In this sense, the implicit default path together with the pivots 
are necessary and sufficient data to navigate along a unique path 
in the path diagram.  
The said data is encoded in $\beta$.  
This observation, too, will carry over to cylindric partitions with greater rank, 
but there will be substantial refinements in argumentation.  

We could get away with declaring $\beta$ simply as a Rogers-Ramanujan partition
in cylindric partitions with profiles $(2,1)$ or $(3, 0)$,
because the pivots may be labeled $(2)$ or $(3)$,
which is dictated by the parity of the weight.
If we want the same construction to work for rank $r = 2$ and level $\ell \geq 4$,
we must attach labels or colors to parts in $\beta$,
please see Figure \ref{figRank2_gen_pivots_2}.
\begin{figure}
  \centering
  \includegraphics[scale=0.1]{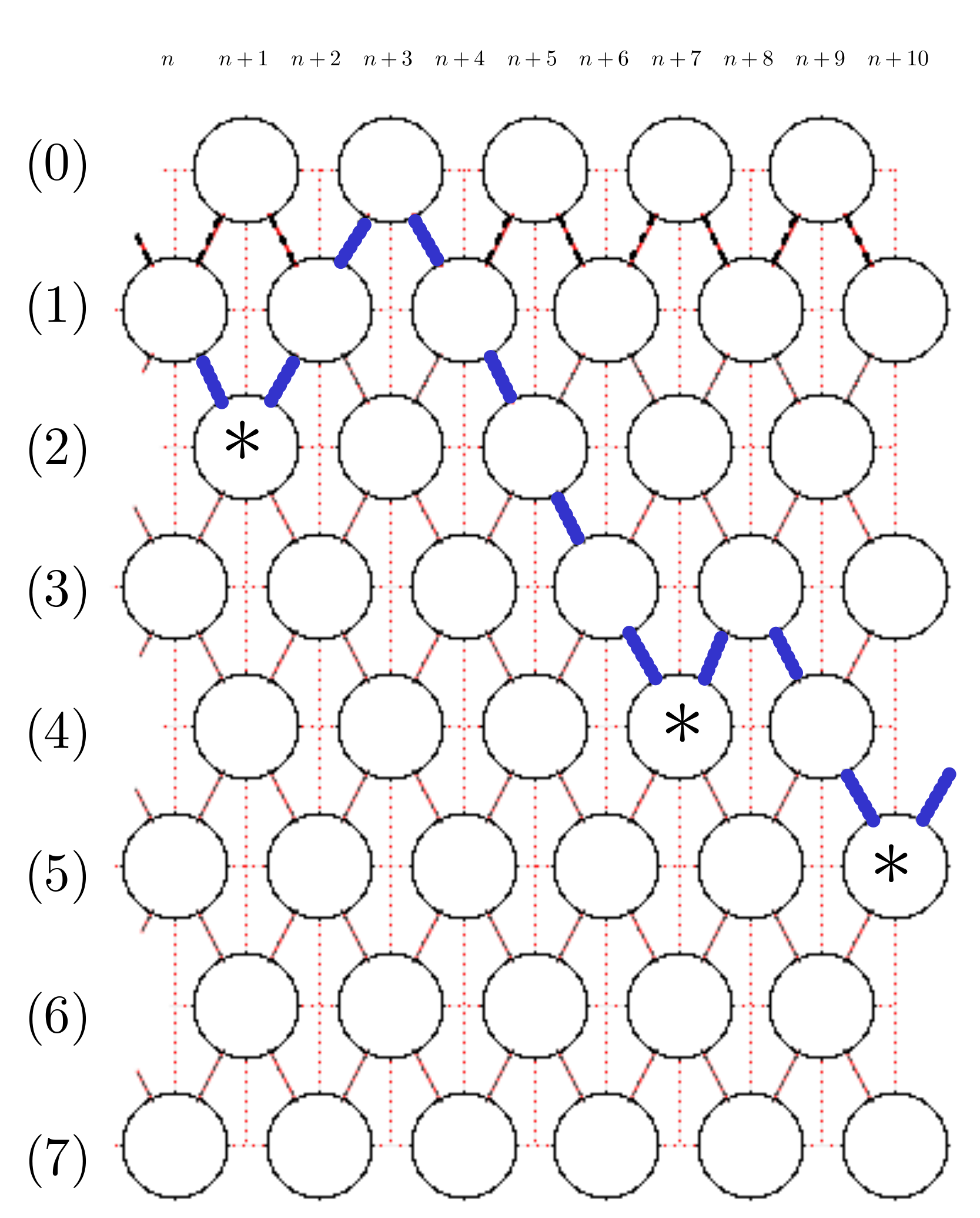}
  \caption{ Three successive pivots in profile $(a, b)$, where $a+b=7$.  }
\label{figRank2_gen_pivots_2}
\end{figure}
In the rank $r = 2$ and level $\ell \geq 4$ case,
pivots may have shape $(2)$ through $(\ell)$.
Depending on the parity of the weight,
the label is either from the set $\{(2), (4), \ldots\}$, or $\{ (3), (5), \ldots \}$.
When we construct the path corresponding to a cylindric partition $\Lambda$
in the path diagram as in the proof of Proposition \ref{c=(2,1)RR1},
the pivots will still be the downward spikes (stalactites)
with label $(2)$ or greater.
It is not difficult to see that if ${}_{j+1}\Pi < {}_{j}\Pi$ are two successive pivots
with weights $w_{j+1}$ and $w_{j}$, and labels $(s_{j+1})$ and $(s_{j})$, respectively,
then
\begin{align}
\nonumber
	w_{j} - w_{j+1} \geq \vert s_{j} - s_{j+1} \vert + 2.
\end{align}
For example, pivots with the same shape may have weights differing by two,
but pivots with shapes $(2)$ and $(4)$ must have weights differing by at least $\vert 2 - 4 \vert + 2 = 4$.
Also, depending on the shape of zero, let's say $(z)$,
not all weight-shape combinations can be pivots.
Some weight-shape combinations may not even be slices.
\begin{figure}
  \centering
  \includegraphics[scale=0.1]{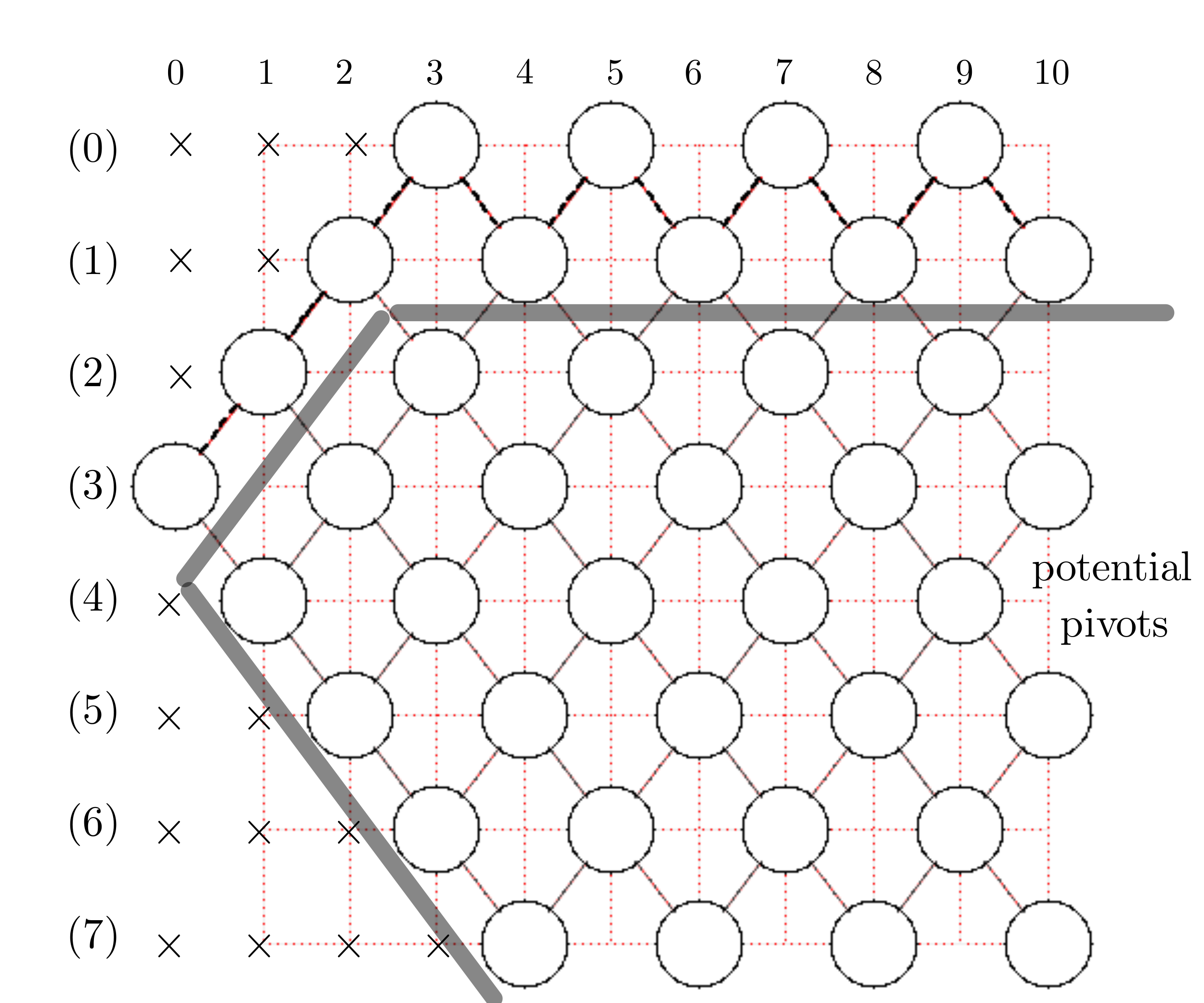}
  \caption{ Impossible slices and potential pivots.  }
\label{figRank2_gen_start}
\end{figure}
In order for a slice with weight $w$ and shape $(s)$, $s \geq 2$, to be a pivot,
both of the inequalities
\begin{align}
\nonumber
	w \geq s - z, \quad \textrm{ and } \quad
	w > z - s
\end{align}
must be satisfied.
Please see Figure \ref{figRank2_gen_start}.
The preceding discussion proves the rank $r = 2$ and arbitrary level $\ell$ case.

\begin{prop}
\label{propRankTwoLevelAnyUnrestricted}
	Let $a$ and $b$ be non-negative integers such that $a + b > 0$.
	Cylindric partitions with profile $(a, b)$ are in one-to-one correspondence
	with pairs $(\mu, \beta)$ in which $\mu$ is an unrestricted partition,
	and $\beta$ is a labeled Rogers-Ramanujan partition satisfying the following conditions.
	\begin{enumerate}[(i)]
	 \item Each part $w$ in $\beta$ is labeled by one of $(2)$, $(3)$, \ldots, $(a+b)$.
	 \item If $w$ is labeled $(s)$, then $w + s \equiv b \pmod{2}$.
	 \item If parts $w_1 \neq w_2$ of $\beta$ are labeled $(s_1)$ and $(s_2)$, respectively,
		then $\vert w_1 - w_2 \vert \geq \vert s_1 - s_2 \vert + 2$.
	 \item A part $w$ labeled $(s)$ in $\beta$ must satisfy both $w \geq s - b$ and $w > b-s$.
	\end{enumerate}
\end{prop}

The reader can verify that Proposition reduces to Theorems 3 and 6 in~\cite{KO}
when $a + b = 2$, and to Proposition \ref{c=(2,1)RR1} when $a + b = 3$ with significant simplifications.

It is possible to construct similar correspondences 
for cylindric partitions with rank $r \geq 3$.
However, we cannot give a corollary-like generalization of the above propositions.
The arguments must be improved using the structure of the path diagram 
coming from restricted partitions as shapes.
As indicated before, rank $r = 2$ case is presented separately,
because the pictures gave a very clear idea of the proofs.
We will enhance those ideas in another way.
For rank $r \geq 3$, as we will see below, $\beta$ will still have distinct parts,
but the parts may be consecutive sometimes.
% We suggest the reader to look at Figure \ref{figPathDiag} again.

The two main obstacles for a straight generalization can be seen 
in Figure \ref{figPathDiag}.
There are crossing edges, and there seem to be no natural way around it.
Therefore, the ``upper envelope'' shortcut needs more careful explanation.
In fact, it is impossible to do it by relying on graphs alone,
as easy counterexamples may be constructed.
Secondly; although it is fairly straightforward to describe the default path,
which is faintly indicated in Figure \ref{figPathDiag},
and argue that any node outside of the default path may be a pivot;
it is almost impossible to decide which nodes are pivot nodes on any given path in the path diagram.
To resolve both issues,
we must examine the structure of the path diagram resulting from adding boxes
to outer corners of slices more closely.
The flip side is that this analysis is an overkill in the rank $r = 2$ case,
in which everything is clearer on the path diagram.
We now prove the fully general case.

\begin{proof}[Proof of Theorem \ref{thmCylPtnVsPtnPairsFullCase}]
  Choose and fix a cylindric partition $\Lambda$ with profile $c$.  
  Then, draw the distinct shapes of slices of $\Lambda$, 
  discarding their multiplicities for a moment.  
  As indicated above, the cylindric partition with profile $c = (1,1,1)$
  $\Lambda = ( (5, 4), (8, 2), (7, 5, 1) )$ has the following distinct slices.

	\begin{align}
	\nonumber
	\vcenter{ \hbox { \resizebox{0.4\width}{0.4\height}{
		\ydiagram{2+2, 1+2, 0+3} } } }
	>
	\vcenter{ \hbox { \resizebox{0.4\width}{0.4\height}{
		\ydiagram{2+2, 1+2, 0+2} } } }
	>
	\vcenter{ \hbox { \resizebox{0.4\width}{0.4\height}{
		\ydiagram{2+2, 1+1, 0+2} } } }
	>
	\vcenter{ \hbox { \resizebox{0.4\width}{0.4\height}{
		\ydiagram{2+1, 1+1, 0+2} } } }
	>
	\vcenter{ \hbox{ \resizebox{0.4\width}{0.4\height}{
		\ydiagram[*(white)]{ 3+0, 2+1, 1+1 }
		*[*(darkgray)]{3,3,2}
		} } }
		>
	\vcenter{ \hbox{ \resizebox{0.4\width}{0.4\height}{
		\ydiagram[*(white)]{ 3+0, 2+1, 1+0 }
		*[*(darkgray)]{3,3,1}
		} } }
		>
	\vcenter{ \hbox{ \resizebox{0.4\width}{0.4\height}{
		\ydiagram[*(white)]{ 3+0, 2+0, 1+0 }
		*[*(darkgray)]{3,2,1}
		} } }
		= E
	\end{align}
  
  For all we know, each shape appears at least once in $\Lambda$.  
  In the path diagram, this corresponds to the nodes 
  with positive multiplicities recorded in them.  
  By the construction of the path diagram, 
  we know that each pair of successive nodes are connected.  
  Please see Figure \ref{figGen_thm_ex_1_path_diag},
  where we indicated each of the non-zero multiplicities with an asterisk.
  \begin{figure}
    \centering
    \includegraphics[scale=0.1]{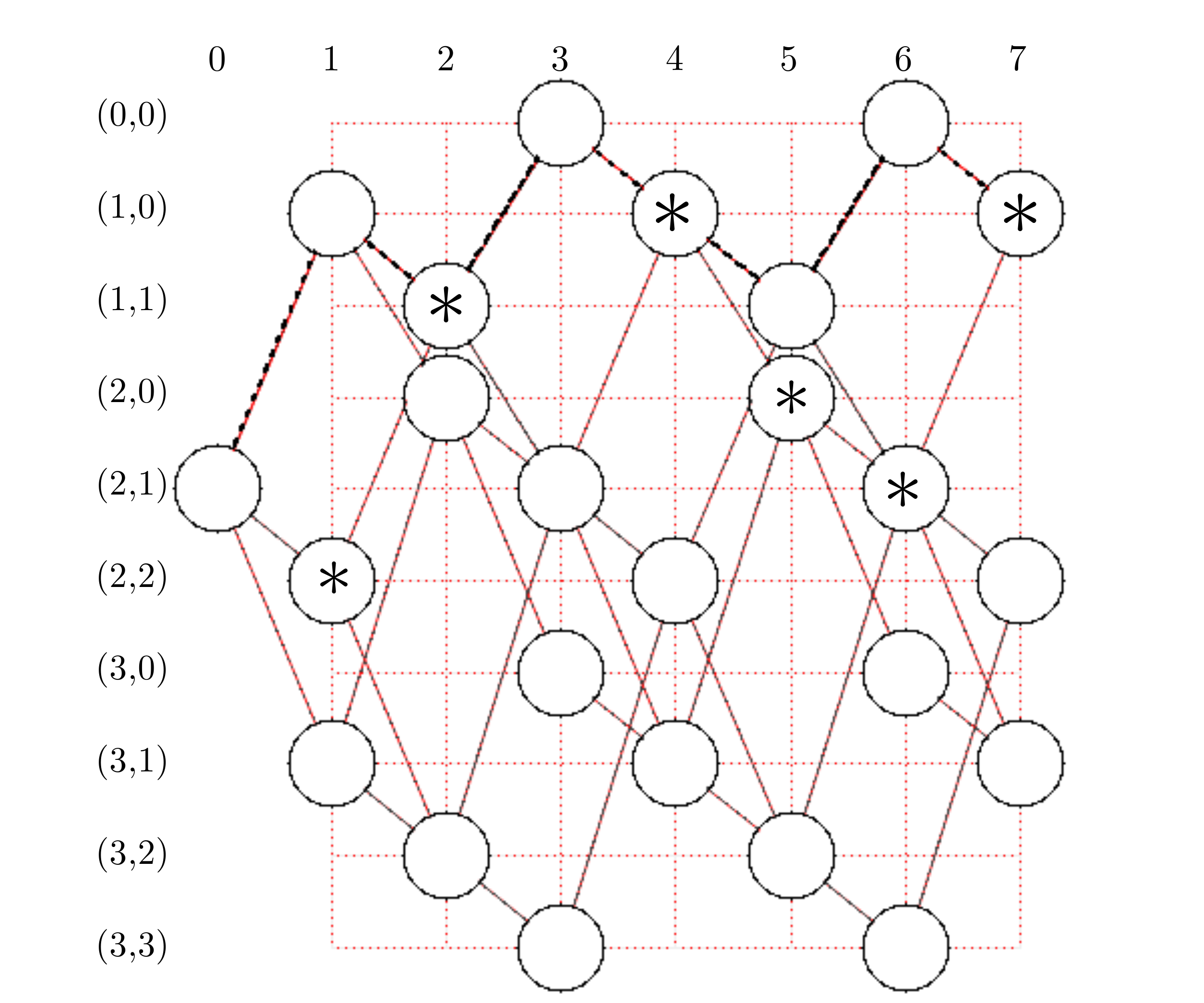}
    \caption{ Nodes with positive multiplicities.  }
  \label{figGen_thm_ex_1_path_diag}
  \end{figure}

  We imagine tiling the space between each successive pairs of slices
  by $1 \times 1$ squares, one at a time, 
  working on the same space until it is completely filled up.  
  We begin with the space between the zero slice and the smallest slice.  
  Then, we continue with the space between the smallest and the next smallest slice, etc.  
  Finally, we fill the infinite strip to the right of the largest slice.  
  \begin{center}
   \includegraphics[scale=0.5]{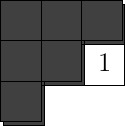}
   \raisebox{8mm}{$\quad \longrightarrow \quad$}
   \includegraphics[scale=0.5]{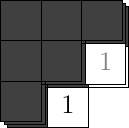}
   \raisebox{8mm}{$\quad \longrightarrow \quad$}
   \includegraphics[scale=0.5]{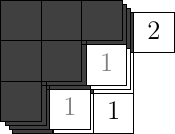}
   \raisebox{8mm}{$\quad \longrightarrow \cdots$}
  \end{center}
  We place the $1 \times 1$ boxes in a space beginning with the leftmost column,
  and filling that column from top to bottom,
  and work on the next leftmost empty column, etc.
  \begin{center}
   \raisebox{8mm}{$\cdots \longrightarrow \quad$}
   \includegraphics[scale=0.5]{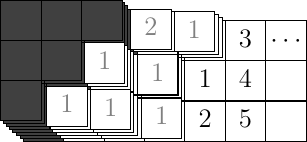}
  \end{center}

  After the placement of each box,
  hence obtaining a slice from the previous one by increasing its weight by one, 
  we record the new intermediate slice along with the edge connecting the previous 
  slice and the new one.  
  Below is the new look of the infinite path drawn on the path diagram.
  This infinite path has one and only one slice of each weight 
  $n = 0, 1, 2, \ldots$ 
  \begin{center}
   \includegraphics[scale=0.5]{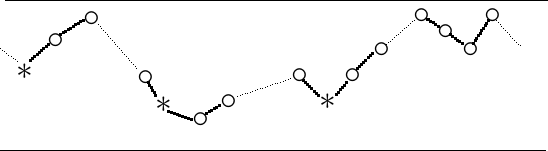}
  \end{center}

  For each slice $\Sigma$ in $\Lambda$,
  consider the last box $s$ we used for filling up the space just before $\Sigma$,
  and the first box $b$ we used for filling up the space just after $\Sigma$.
  If the column of $b$ is to the left of the column of $s$, 
  then we declare $\Sigma$ as a \emph{pivot} slice in $\Lambda$.
  In other words, $\Sigma$ is a pivot slice
  if it causes the filling up operation as described above
  to make a backward step.  
  By construction, no intermediate node we marked can be a pivot, 
  and not all shapes in $\Lambda$ will be pivots.  
  For example, in 
  \begin{center}
   \includegraphics[scale=0.5]{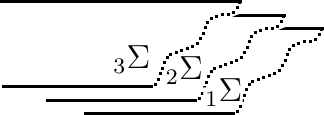}
  \end{center}
  
  if we would have encountered ${}_{2}\Sigma$ already
  in filling the gap between ${}_{1}\Sigma$ and ${}_{3}\Sigma$,
  as if ${}_{2}\Sigma$ was not there,
  ${}_{2}\Sigma$ will not be a pivot.
  After the declaration of pivothood, we denote the pivot slices by $\Pi$.
  
  In our running example,
  $\vcenter{ \hbox{ \resizebox{0.4\width}{0.4\height}{
		\begin{ytableau}
		 *(darkgray) \phantom{0} & *(darkgray) \phantom{0} & *(darkgray) \phantom{0} \\
		 *(darkgray) \phantom{0} & *(darkgray) \phantom{0} & *(white) \textcolor{gray}{\phantom{1}} \\
		 *(darkgray) \phantom{0}
		\end{ytableau}
		} } }$
	is a pivot slice, because to fill up the space for the next slice
	$\vcenter{ \hbox{ \resizebox{0.4\width}{0.4\height}{
		\begin{ytableau}
		 *(darkgray) \phantom{0} & *(darkgray) \phantom{0} & *(darkgray) \phantom{0} \\
		 *(darkgray) \phantom{0} & *(darkgray) \phantom{0} & *(white) s \\
		 *(darkgray) \phantom{0} & *(white) b
		\end{ytableau}
		} } }$,
	we need to place the first box $b$ strictly to the left of the last box $s$ we used for fully tiling the former slice.
	The slice
	$	\vcenter{ \hbox{ \resizebox{0.4\width}{0.4\height}{
	\begin{ytableau}
		*(darkgray) \phantom{0} & *(darkgray) \phantom{0} & *(darkgray) \phantom{0} \\
		*(darkgray) \phantom{0} & *(darkgray) \phantom{0} & *(white) \phantom{0} \\
		*(darkgray) \phantom{0} & *(white) \textcolor{gray}{\phantom{1}}
	\end{ytableau}
	} } }$ is not a pivot slice,
	because to fill up the space in the next slice
	$\vcenter{ \hbox{ \resizebox{0.4\width}{0.4\height}{
	\begin{ytableau}
		*(darkgray) \phantom{0} & *(darkgray) \phantom{0} & *(darkgray) \phantom{0} & *(white) \textcolor{gray}{\phantom{2}} \\
		*(darkgray) \phantom{0} & *(darkgray) \phantom{0} & *(white) s \\
		*(darkgray) \phantom{0} & *(white) \phantom{0} & *(white) b
	\end{ytableau}
	} } }$,
	the first box $b$ we place is not strictly to the left of the last box $s$
	we used for tiling the previous slice.
	The only two pivot slices in this example are
	\begin{align}
	\nonumber
	\vcenter{ \hbox{ \resizebox{0.4\width}{0.4\height}{
		\begin{ytableau}
		 *(darkgray) \phantom{0} & *(darkgray) \phantom{0} & *(darkgray) \phantom{0} & *(white) \phantom{0}
			& *(white) \textcolor{gray}{\phantom{0}} \\
		 *(darkgray) \phantom{0} & *(darkgray) \phantom{0} & *(white) \phantom{0} \\
		 *(darkgray) \phantom{0} & *(white) \phantom{0} & *(white) \phantom{0}
		\end{ytableau}
		} } },
		\quad \textrm{ and } \quad
	\vcenter{ \hbox{ \resizebox{0.4\width}{0.4\height}{
		\begin{ytableau}
		 *(darkgray) \phantom{0} & *(darkgray) \phantom{0} & *(darkgray) \phantom{0} \\
		 *(darkgray) \phantom{0} & *(darkgray) \phantom{0} & *(white) \textcolor{gray}{\phantom{1}} \\
		 *(darkgray) \phantom{0}
		\end{ytableau}
		} } }.
  \end{align}
  They have shapes $(2,0)$ and $(2,2)$, and weights 5 and 1, respectively.

  For clear reasons, the zero slice cannot be a pivot.
  Nor any slice with shape $(0, \ldots, 0)$, $(1, 0, \ldots, 0)$, \ldots 
  $(1, \ldots, 1)$ can be a pivot.  
  By construction, none of the intermediate slices 
  that were not originally in $\Lambda$ but later added to the infinite path 
  in the path diagram are pivots.  
  In short, to determine whether or not a slice is a pivot slice,
  we need to examine the \emph{spaces} before and after it.
  The space before a slice is the gap between the slice and the previous strictly smaller slice,
  which can be the zero slice.
  The space after a slice is the gap between that and the next strictly larger slice,
  which may be the hypothetical slice of shape $(0,\ldots,0)$ and large enough weight.
  
  One copy from each pivot slice will make up $\beta$.  
  $\beta$ necessarily has distinct parts.  
  We can regard $\beta$ as a partition into distinct parts, 
  in which parts are labeled by shapes of respective pivot slices.  
  The labeling will satisfy some conditions implied by the construction.  
  Instead of listing all those conditions, 
  we keep in mind the dual outlook of $\beta$ as 
  a collection of pivot slices.  
  
  The unrestricted partition $\mu$ is formed by 
  the weights of the remaining slices after we take out one copy of each pivot.  
  We discard the information given by the shapes, 
  because that information is already stored in the pivots.  
  The shapes for the weights in $\mu$ can be recovered 
  by stacking the pivot slices and tiling the spaces in between.

  In our example, we find that the pair
  \begin{align}
	\nonumber
		& \beta = 5^{(2,0)}, 1^{(2,2)}, \\
	\nonumber
		& \mu = 7, 6, 5, 4, 2, 2,
  \end{align}
  corresponds to the cylindric partition $\Lambda = ( (5, 4), (8, 2), (7, 5, 1) )$.
  The shapes of the pivot slices recorded in $\beta$ are denoted by the superscripts.
  Recall that the profile is $c=(1,1,1)$.

  For the reverse direction,
  we are given a pair of partitions $(\mu, \beta)$ 
  in which $\mu$ is an unrestricted partition, 
  and $\beta$ is a partition into labeled distinct parts satisfying 
  some further conditions, i.e. the labeling must be admissible.  
  In particular, the labels are shapes in the form of 
  restricted partitions with at most $r-1$ parts, 
  each of which is at most $\ell$.  
  These numbers $\ell$ and $r$ are known.  
  When we consider the size of each part together with its shape in $\beta$, 
  we obtain slices with profile $c$.  
  These slices when stacked together must be successive pivots, 
  i.e. the tiling according to the rules above will cause 
  the first box after a slice to be placed in a column 
  to the left of the last placed box for filling the space just before that slice.  
  
  After marking the pivot slices in the path diagram 
  by tallying one in each of the corresponding nodes, 
  we construct the infinite path passing through these nodes 
  as described above.  
  Namely, we tile each space between the pivots one square at a time, 
  recording the new intermediate slice in the path diagram along with 
  the edge connecting it to the preceding slice.  
  We start with the space before the smallest pivot, 
  then continue with the space between the smallest and the next smallest pivot etc., 
  and finish with the infinite space to the right of the largest pivot.  
  
  By construction, the only pivot slices in this infinite path 
  are the ones coming from $\beta$.  
  We have not introduced any new, nor altered any existing pivots.  
  Again, we have exactly one node of each size $n \in \mathbb{N}$.  
  Next, for each part in $\mu$, 
  we register one slice at the node with weight equal to the part size.  
  This is possible in exactly one way.  
  Finally, we read the slice counts on the path diagram as a cylindric partition 
  $\Lambda$ with profile $c$.  
  
  Please observe that the constructions of $(\mu, \beta)$ from $\Lambda$ 
  and of $\Lambda$ from $(\mu, \beta)$ are inverses of each other.  
  
  The final part of the proof is the length of runs in $\beta$.  
  The determination of pivots depend on placing a box to the left of the previous one.  
  It is implicit that we need to place the new box to another row; 
  in particular, a row that is strictly below the last box.  
  This is due to the fact that the right end of a slice 
  is the shape of the Young diagram of a partition.  
  So, if the last box before a pivot is in the top row, 
  the highest the next box can go is the second row etc. 
  If the last box is placed in the bottom row, 
  then the slice cannot be a pivot.  
  Consequently, there is room for only $r-1$ 
  pivots with consecutive weights.  
  Therefore, the longest run in $\beta$ can have length at most $r-1$.  
  Similarly, the number of columns to the left of a box that has just been tiled
  in which a new outer corner may be placed is bounded by $\ell - 1$.
\end{proof}

We present another example.
This time, we are given
\begin{align}
\nonumber
  & \beta = 15^{(2,1)}, 11^{(3, 2)}, 10^{(3, 1)}, 1^{(2, 2)}, \textrm{ and } \\
\nonumber
  & \mu = 13, 10, 10, 9, 5, 5, 3, 2,
\end{align}
corresponding to a cylindric partition with profile $c=(1,1,1)$.
Thus, the shape of zero is $(2, 1)$.
We first verify that the slices described by $\beta$ are indeed pivots.
For that, we (almost) superimpose successive pairs or triples of the alleged pivots
${}_{1}\Pi$ $>{}_{2}\Pi$ $>{}_{3}\Pi$ $>{}_{4}\Pi$
with respective weights and shapes written above.
We also indicate the last box $s$ used for tiling the space before
the smaller pivot and the first box $b$ used for tiling the space after it.
\begin{center}
 \includegraphics[scale=1.2]{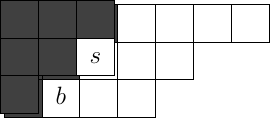} \raisebox{4mm}{, }
 \includegraphics[scale=1.2]{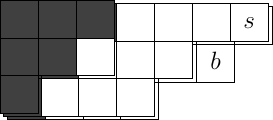} \raisebox{4mm}{, }
 \includegraphics[scale=1.2]{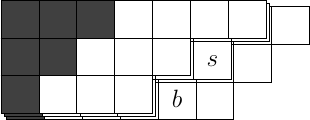} \raisebox{4mm}{, }
 \includegraphics[scale=1.2]{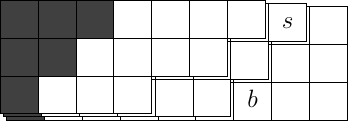} \raisebox{4mm}{. }
\end{center}
To find the intermediate slices between, say, ${}_{4}\Pi$ and ${}_{3}\Pi$,
we tile the space between the said pivots as described
in the proof of Theorem \ref{thmCylPtnVsPtnPairsFullCase},
and record slices after the insertion of each box.
The intermediate slices between ${}_{4}\Pi$ and ${}_{3}\Pi$ are:
\begin{align}
\nonumber
	\vcenter{ \hbox{ \resizebox{0.4\width}{0.4\height}{
		\begin{ytableau}
			*(darkgray) & *(darkgray) & *(darkgray) \\
			*(darkgray) & *(darkgray) & \\
			*(darkgray) &
		\end{ytableau}
		} } },
	\vcenter{ \hbox{ \resizebox{0.4\width}{0.4\height}{
		\begin{ytableau}
			*(darkgray) & *(darkgray) & *(darkgray) \\
			*(darkgray) & *(darkgray) & \\
			*(darkgray) & &
		\end{ytableau}
		} } },
	\vcenter{ \hbox { \resizebox{0.4\width}{0.4\height}{
		\ydiagram{2+1, 1+1, 0+2} } } },
	\vcenter{ \hbox { \resizebox{0.4\width}{0.4\height}{
		\ydiagram{2+1, 1+2, 0+2} } } },
	\vcenter{ \hbox { \resizebox{0.4\width}{0.4\height}{
		\ydiagram{2+1, 1+2, 0+3} } } },
	\vcenter{ \hbox { \resizebox{0.4\width}{0.4\height}{
		\ydiagram{2+2, 1+2, 0+3} } } },
	\vcenter{ \hbox { \resizebox{0.4\width}{0.4\height}{
		\ydiagram{2+2, 1+3, 0+3} } } },
	\vcenter{ \hbox { \resizebox{0.4\width}{0.4\height}{
		\ydiagram{2+3, 1+3, 0+3} } } },
	\vcenter{ \hbox { \resizebox{0.4\width}{0.4\height}{
		\ydiagram{2+4, 1+3, 0+3} } } }.
\end{align}
We call these intermediate slices $\Sigma$'s.
Observe that $\Pi$'s are still pivots after the insertion of these $\Sigma$'s
in the chain of containment,
and none of the $\Sigma$'s are pivots.
We do this for the space between ${}_{3}\Pi$ and ${}_{2}\Pi$,
that between ${}_{2}\Pi$ and ${}_{1}\Pi$,
and lastly for the space after ${}_{1}\Pi$.
The obtained chain of strict containment of slices now has
a unique slice of each positive weight.
At this point, we can incorporate parts of $\mu$.
To find the cylindric partition corresponding to the pair $(\beta, \mu)$,
we simply add the slice with the corresponding weight
for each part in $\mu$.
Counting the pivots as well,
we have the following slices for $\Lambda$.
\begin{align}
\nonumber
  \Lambda
  = \vcenter{ \hbox{ \resizebox{0.4\width}{0.4\height}{
		\begin{ytableau}
			*(darkgray) & *(darkgray) & *(darkgray) \\
			*(darkgray) & *(darkgray) & \\
			*(darkgray)
		\end{ytableau}
		} } }
	+ \vcenter{ \hbox{ \resizebox{0.4\width}{0.4\height}{
		\begin{ytableau}
			*(darkgray) & *(darkgray) & *(darkgray) \\
			*(darkgray) & *(darkgray) & \\
			*(darkgray) &
		\end{ytableau}
		} } }
	+ \vcenter{ \hbox{ \resizebox{0.4\width}{0.4\height}{
		\begin{ytableau}
			*(darkgray) & *(darkgray) & *(darkgray) \\
			*(darkgray) & *(darkgray) & \\
			*(darkgray) & &
		\end{ytableau}
		} } }
	+ 2 \times \vcenter{ \hbox { \resizebox{0.4\width}{0.4\height}{
		\ydiagram{2+1, 1+2, 0+2} } } }
	+ \vcenter{ \hbox { \resizebox{0.4\width}{0.4\height}{
		\ydiagram{2+3, 1+3, 0+3} } } }
	+ 3 \times \vcenter{ \hbox { \resizebox{0.4\width}{0.4\height}{
		\ydiagram{2+4, 1+3, 0+3} } } }
\end{align}
\begin{align}
\nonumber
	+ \vcenter{ \hbox { \resizebox{0.4\width}{0.4\height}{
		\ydiagram{2+4, 1+4, 0+3} } } }
	+ \vcenter{ \hbox { \resizebox{0.4\width}{0.4\height}{
		\ydiagram{2+4, 1+4, 0+5} } } }
	+ \vcenter{ \hbox { \resizebox{0.4\width}{0.4\height}{
		\ydiagram{2+5, 1+5, 0+5} } } }
  = \begin{array}{cccccccc}
     & & & 9 & 7 & 7 & 6 & 1 \\
     & & 12 & 9 & 7 & 3 & 1 & \\
     & 11 & 10 & 7 & 2 & 2 & & \\
     \textcolor{gray!50}{9} & \textcolor{gray!50}{7} & \textcolor{gray!50}{7} &
     \textcolor{gray!50}{6} & \textcolor{gray!50}{1} & & &
    \end{array}
\end{align}
We also show how $\Lambda$ is recorded in the path diagram
in Figure \ref{figRev_ex_pat_diag}.
The pivot slices are indicated.
\begin{figure}
  \centering
  \includegraphics[scale=0.1]{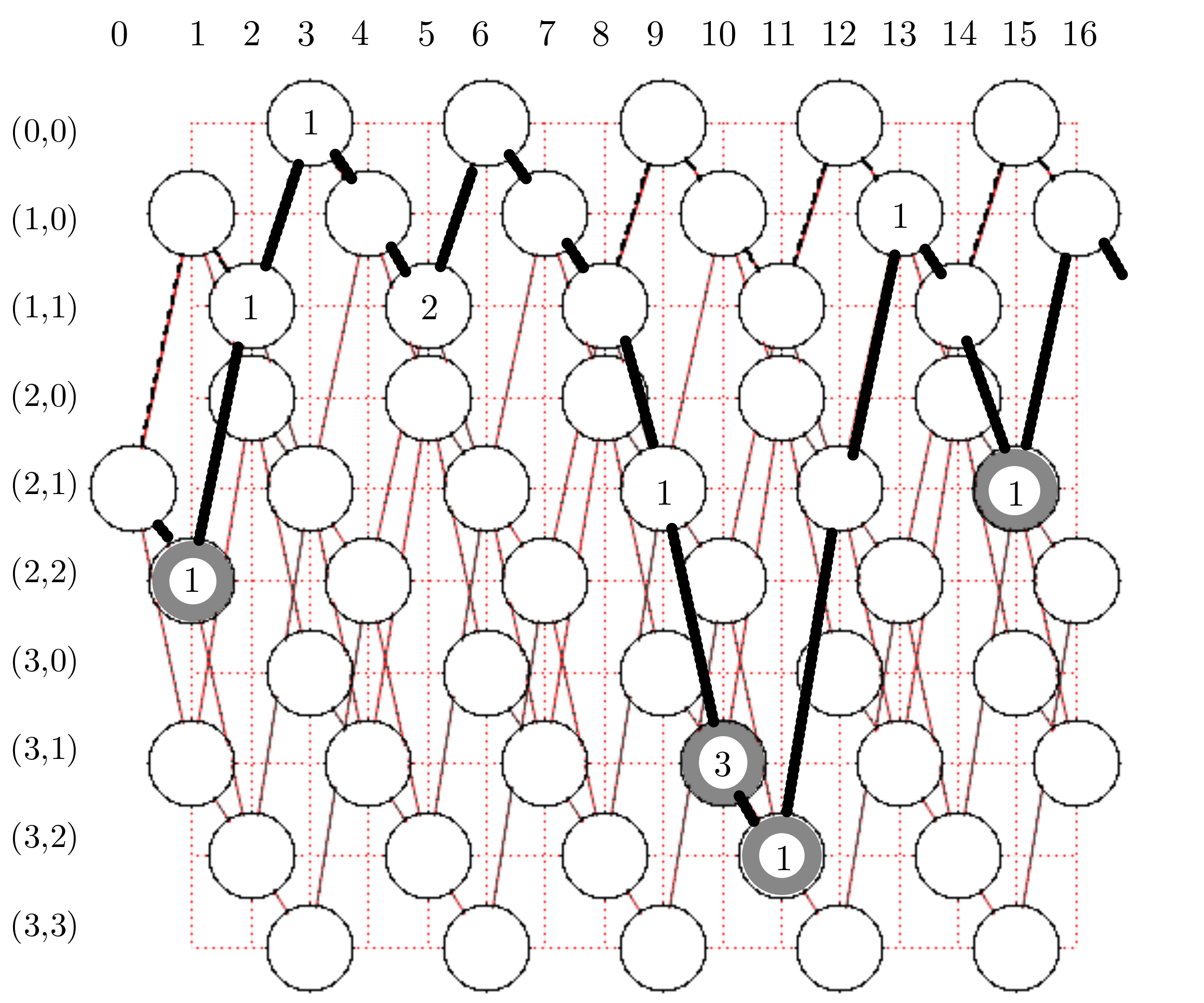}
  \caption{ A cylindric partition shown in the path diagram, the pivots are indicated.  }
\label{figRev_ex_pat_diag}
\end{figure}

% \begin{align*}
% 	\vcenter{ \hbox{ \resizebox{\width}{\height}{
% 		\begin{ytableau}
% 			*(darkgray) & *(darkgray) & *(darkgray) & & & & & & \\
% 			*(darkgray) & *(darkgray) & & & & & & & \\
% 			*(darkgray) & & & & & & & &
% 		\end{ytableau}
% 		} } } \\
% 	\vcenter{ \hbox{ \resizebox{\width}{\height}{
% 		\begin{ytableau}
% 			*(darkgray) & *(darkgray) & *(darkgray) & & & & & \\
% 			*(darkgray) & *(darkgray) & & & & & \\
% 			*(darkgray) & & & & &
% 		\end{ytableau}
% 		} } }
% 	\vcenter{ \hbox{ \resizebox{\width}{\height}{
% 		\begin{ytableau}
% 			*(darkgray) & *(darkgray) & *(darkgray) & & & & \\
% 			*(darkgray) & *(darkgray) & & & & \\
% 			*(darkgray) & & &
% 		\end{ytableau}
% 		} } } \\
% 	\vcenter{ \hbox{ \resizebox{\width}{\height}{
% 		\begin{ytableau}
% 			*(darkgray) & *(darkgray) & *(darkgray) & & & & \\
% 			*(darkgray) & *(darkgray) & & & \\
% 			*(darkgray) & & &
% 		\end{ytableau}
% 		} } }
% 	\vcenter{ \hbox{ \resizebox{\width}{\height}{
% 		\begin{ytableau}
% 			*(darkgray) & *(darkgray) & *(darkgray) \\
% 			*(darkgray) & *(darkgray) & \\
% 			*(darkgray)
% 		\end{ytableau}
% 		} } }
% \end{align*}
%

One can endow the outlined construction in the above proof 
with much fine print.  
For instance, if $\beta = \varepsilon$, the empty partition, 
then the infinite path in the path diagram is found 
by tiling the infinite horizontal strip to the right of the zero slice.  
\begin{align}
\nonumber
\vcenter{ \hbox{ \resizebox{0.6\width}{0.6\height}{
	\begin{ytableau}
		*(darkgray) \phantom{0} & *(darkgray) \phantom{0} & *(darkgray) \phantom{0} & 4 \\
		*(darkgray) \phantom{0} & *(darkgray) \phantom{0} & 2 & 5 \\
		*(darkgray) \phantom{0} & 1 & 3 & 6
	\end{ytableau}
	} } } \!\! \cdots
\end{align}

We will call this the \emph{default path}.  
As an example, the default path on the path diagram of cylindric partitions
with profile $c=(0,3,0)$ is shown in Figure \ref{figDef_path_ex_30}.
Notice that the shape of zero is $(3,0)$.
\begin{figure}
  \centering
  \includegraphics[scale=0.1]{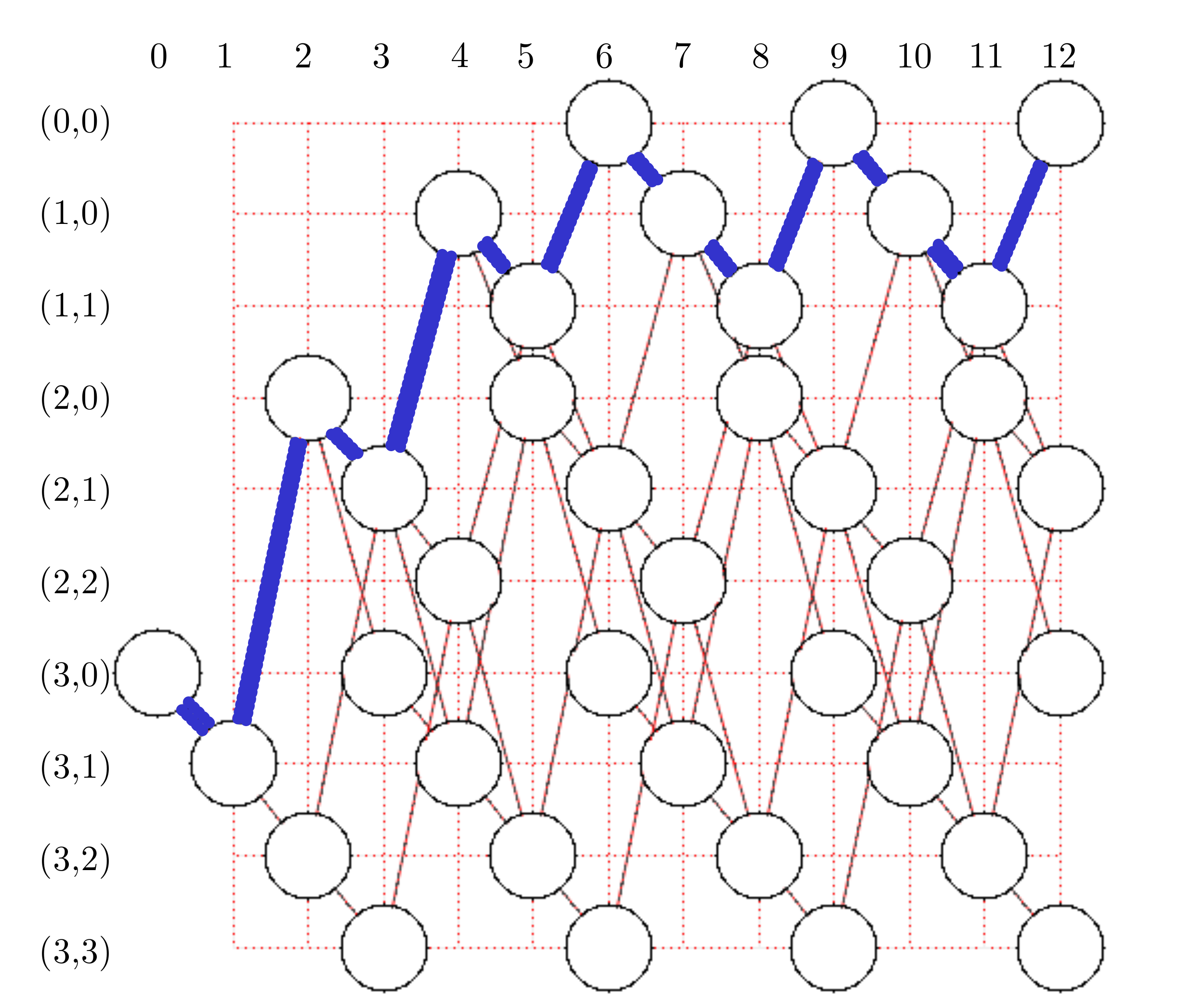}
  \caption{ Default path in profile $c=(0, 3, 0)$.  }
\label{figDef_path_ex_30}
\end{figure}

Between two successive pivots,
there is a mesh of paths in the path diagram.  
The one without any pivots in between 
is the one in which we choose the smallest possible shape 
in the lexicographic ordering for the shape of the next slice.  
An example of rank $r = 3$ and level $\ell = 3$ is shown
in Figure \ref{figMesh_and_upper_env_r3}.
\begin{figure}
  \centering
  \includegraphics[scale=0.1]{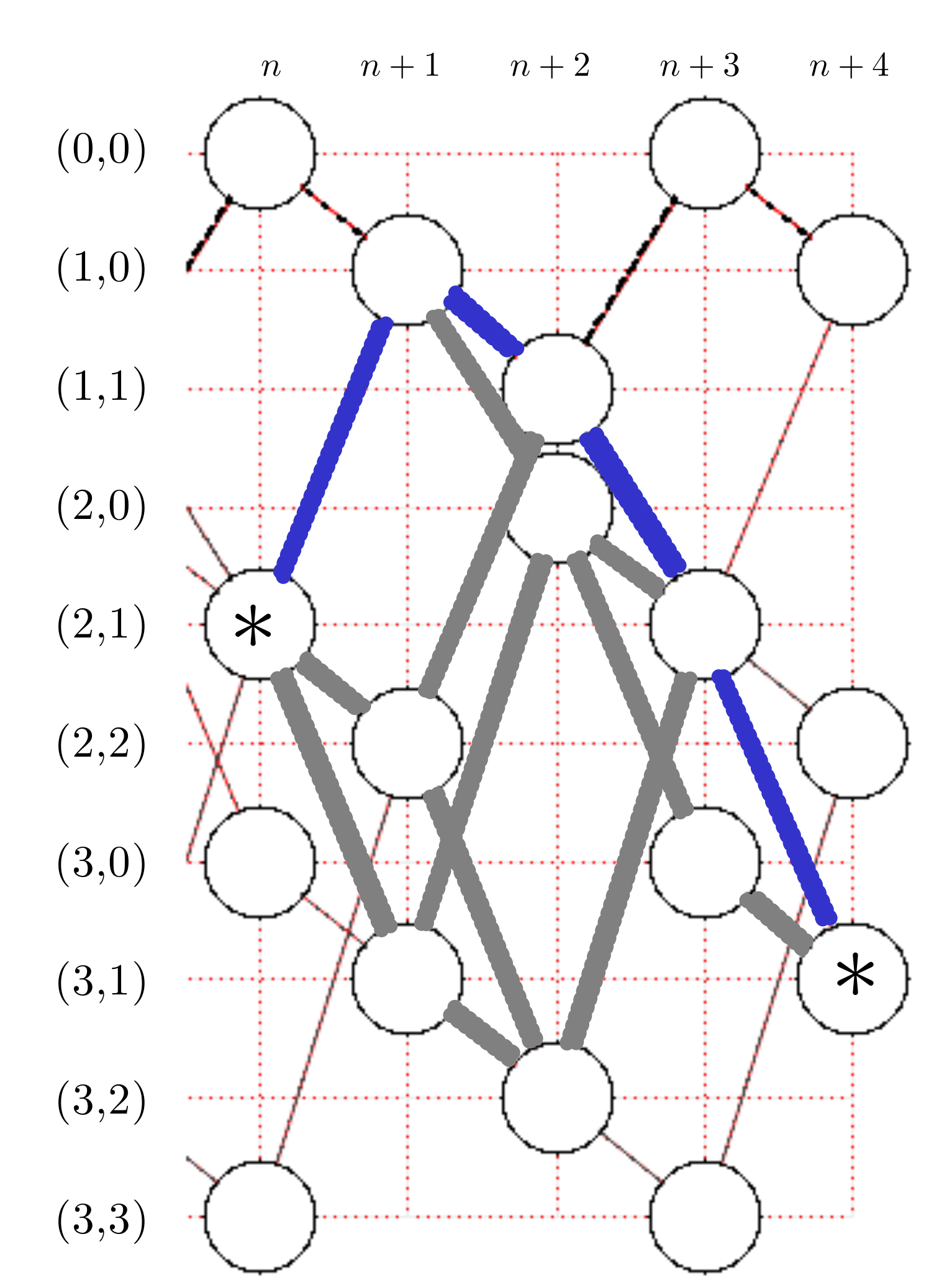}
  \hspace{20mm}
  \includegraphics[scale=0.1]{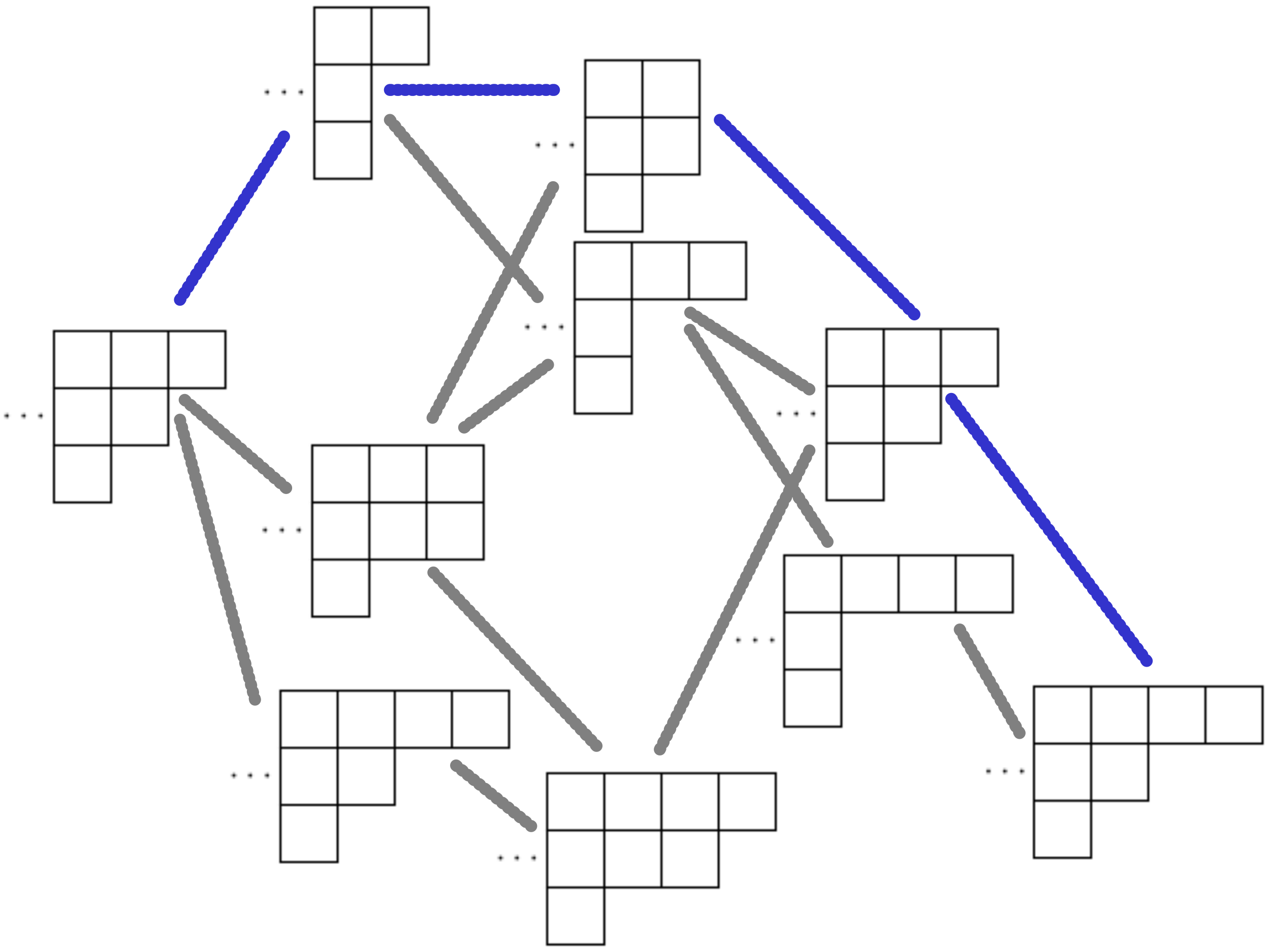}
  \caption{ The mesh and the upper envelope of paths between two successive pivots,
     and the corresponding portion of the (sideways) Hasse diagram.  }
\label{figMesh_and_upper_env_r3}
\end{figure}
Suppose that the slices with shapes $(2, 1)$ and $(3,1)$,
and weights $n$ and $n+4$, respectively,
are the only shown pivots
in the figure on the left in Figure \ref{figMesh_and_upper_env_r3}.
Then, the reader can check that any other path than the indicated thick one
will introduce a new pivot.
In fact, in some of those paths one or both of the indicated pivots
may not be pivots anymore.
This is more clearly seen if we just go back to the sideways Hasse diagram
and draw the slices appearing as vertices on the mesh of paths above,
as seen in the figure on the right in Figure \ref{figMesh_and_upper_env_r3}.

% \begin{align*}
% 	\vspace{5mm}
%  \cdots
%  \vcenter{\hbox{\ydiagram{0+3,0+2,0+1}}}
%  \quad
%  \cdots
%  \vcenter{\hbox{\ydiagram{0+2,0+1,0+1}}}
%  \quad
%  \cdots
%  \vcenter{\hbox{\ydiagram{0+3,0+3,0+1}}}
%  \quad
%  \cdots
%  \vcenter{\hbox{\ydiagram{0+4,0+2,0+1}}}
%  \quad \vspace{5mm} \\
%  \cdots
%  \vcenter{\hbox{\ydiagram{0+2,0+2,0+1}}}
%  \quad
%  \cdots
%  \vcenter{\hbox{\ydiagram{0+3,0+1,0+1}}}
%  \quad
%  \cdots
%  \vcenter{\hbox{\ydiagram{0+4,0+1,0+1}}}
%  \quad
%  \cdots
%  \vcenter{\hbox{\ydiagram{0+4,0+3,0+1}}}
%  \quad
% \end{align*}

% The details of justification are routine to write, and we leave it to the reader.

If $\sigma = (\sigma_1, \ldots, \sigma_{r-1})$ 
is the shape of a pivot, then $\sigma_1 \geq 2$.  
The default path falls back to the shapes 
\begin{align}
\nonumber 
  (0, \ldots, 0) \to (1, 0, \ldots, 0) \to 
  \cdots \to (1, \ldots, 1) \to (0, \ldots, 0)
\end{align}
as quickly as possible.  
As pointed out in the proof, none of these shapes can be a pivot.  

The difference between the above construction
and Tingley's bijection~\cite{Tingley, Tingley-Correction}
is that Tingley performs some \emph{moves} on the cylindric partition
to find the unrestricted partition,
while we do not alter the cylindric partition at all.

For the sake of completeness, we include the definition of a pivot slice,
summarizing the main idea in the proof of Theorem \ref{thmCylPtnVsPtnPairsFullCase}.

\begin{defn}
\label{defPivot}
  For a fixed profile $c$,
  let ${}_{1}\Sigma$ $> {}_{2}\Sigma$ $\cdots$ $> {}_{m}\Sigma$ $> {}_{m+1}\Sigma = E$
  be a chain of slices, not necessarily of consecutive weight.
  For each of the slices ${}_{j}\Sigma$, $j = $ 1, 2, \ldots, $m$,
  consider the space before it,
  i.e. the skew diagram that is inside ${}_{j}\Sigma$, but outside ${}_{j+1}\Sigma$;
  and for $j = $ 2, 3, \ldots, $m$, the space after it,
  i.e. the skew diagram that is outside ${}_{j}\Sigma$, but inside ${}_{j-1}\Sigma$.
  The space after ${}_{1}\Sigma$ is the infinite strip to the right of it.
  If the space after ${}_{j}\Sigma$ has any boxes
  strictly to the left of a box in the space before it,
  ${}_{j}\Sigma$ is called a \emph{pivot} slice in the given chain.
\end{defn}

\section{Generating functions for cylindric partitions into distinct parts}
\label{secDist}

We present the technical proof of Theorem \ref{thmDistPartsGeneral} here.  

\begin{proof}[Proof of Theorem \ref{thmDistPartsGeneral}]
Fix the profile $c$. 
Let $M$ be a cylindric partition with profile $c$ into distinct parts.  
Let 
\begin{align}
\nonumber
  f_1 \cdot {}_1\Sigma + f_2 \cdot {}_2\Sigma + \cdots 
  f_n \cdot {}_n\Sigma
\end{align}
be the decomposition of $M$ into slices such that 
\begin{align}
\nonumber 
  {}_1\Sigma > {}_2\Sigma > \cdots > {}_n\Sigma > E, 
\end{align}
and $f_j > 0$ for $j = 1, 2, \ldots, n$.  
We must have 
$ \vert {}_j\Sigma \vert - \vert {}_{j+1}\Sigma \vert = 1 $ 
for $j = 1, 2, \ldots, n$, and with the identification $E = {}_{n+1}\Sigma$.  
Because if that difference is greater than one, 
then there are two or more squares between slices 
${}_j\Sigma$ and ${}_{j+1}\Sigma$, causing a repeated part in $M$.  
Since $\vert {}_{n+1}\Sigma \vert = \vert E \vert = 0$, 
we deduce that 
\begin{align}
\nonumber 
  \vert {}_j\Sigma \vert = n+1 - j
\end{align}
for $j = 1, 2, \ldots, n$.  
In the path diagram, $M$ will be registered as
in Figure \ref{figDist_path_diag_ex_0}.
\begin{figure}
  \centering
  \includegraphics[scale=0.1]{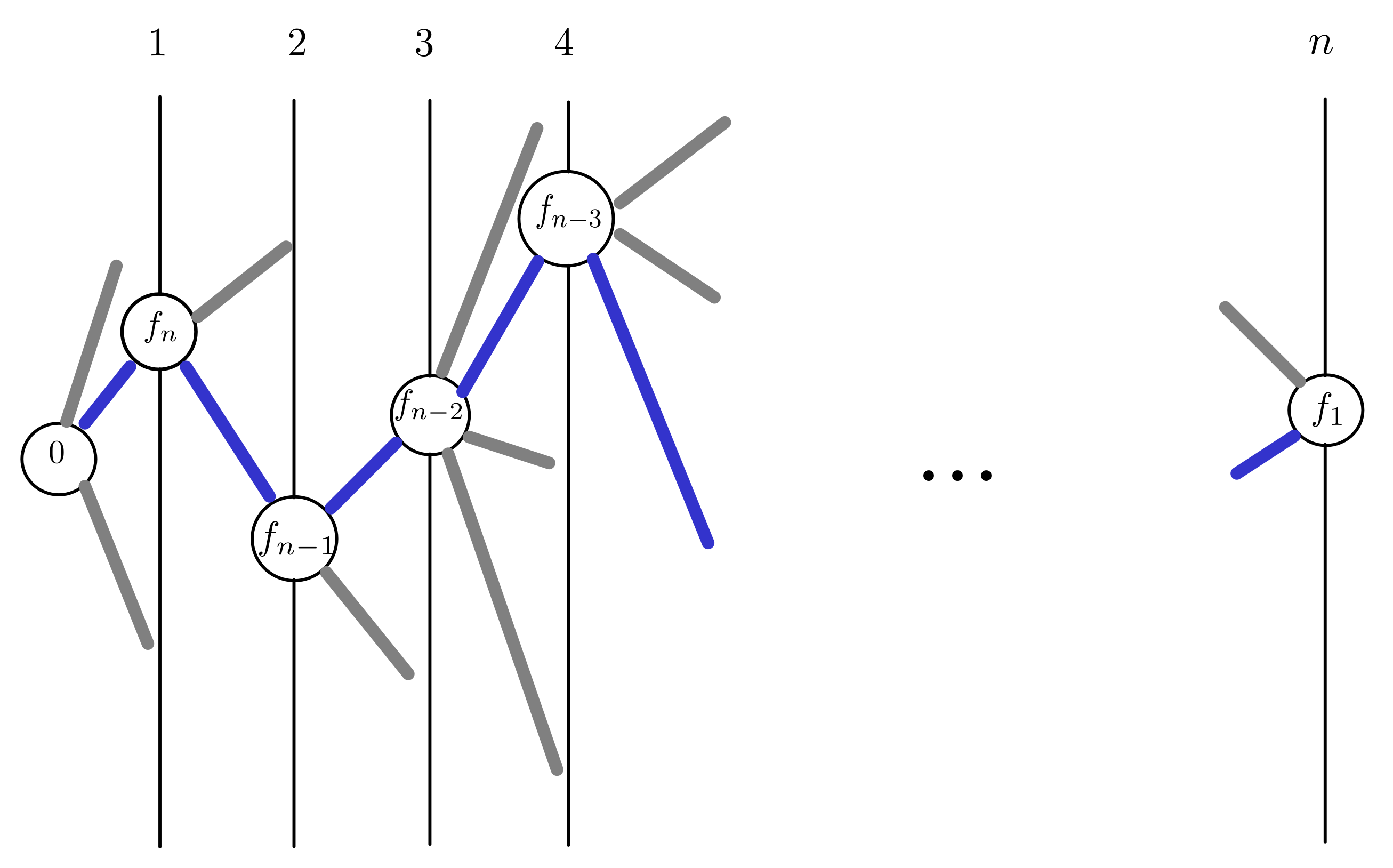}
  \caption{ A cylindric partition into distinct parts in a path diagram.  }
\label{figDist_path_diag_ex_0}
\end{figure}
% \begin{align*}
% 	\circled{\phantom{0}0\phantom{0}} \qquad
% 	\circled{\phantom{1}\fn\phantom{1}} \qquad
% 	\circled{\fno} \qquad
% 	\circled{\fnt} \qquad
% 	\circled{\fnth} \qquad
% 	\circled{\phantom{1}\fo\phantom{1}} \qquad
% \end{align*}
% \begin{align*}
% 	{ \scriptstyle \begin{array}{ccccc}
% 	 1 \hspace{6mm} &
% 	 2 \hspace{6mm} &
% 	 3 \hspace{6mm} &
% 	 4 \hspace{37mm} & n \\
% 	\end{array} \hspace{3mm}} \\
%  \begin{array}{c|c|c|c|c|c}
%   \phantom{1111} & \phantom{2222} & \phantom{3333} & \phantom{4444}
% 		& \phantom{5555555555555555555} & \phantom{6} \\
%   \phantom{1111} & \phantom{2222} & \phantom{3333} & \phantom{4444}
% 		& \phantom{55555555555} & \phantom{6} \\
%   \phantom{1111} & \phantom{2222} & \phantom{3333} & \phantom{4444}
% 		& \phantom{55555555555} & \phantom{6} \\
%   \phantom{1111} & \phantom{2222} & \phantom{3333} & \phantom{4444}
% 		& \phantom{55555555555} & \phantom{6} \\
%   \phantom{1111} & \phantom{2222} & \phantom{3333} & \phantom{4444}
% 		& \phantom{55555555555} & \phantom{6} \\
%   \phantom{1111} & \phantom{2222} & \phantom{3333} & \phantom{4444}
% 		& \phantom{55555555555} & \phantom{6} \\
%   \phantom{1111} & \phantom{2} & \phantom{3} & \phantom{4}
% 		& \phantom{55555555555} & \phantom{6} \\
%   \phantom{1} & \phantom{2} & \phantom{3} & \phantom{4}
% 		& \phantom{55555555555} & \phantom{6} \\
%   \phantom{1} & \phantom{2} & \phantom{3} & \phantom{4}
% 		& \phantom{55555555555} & \phantom{6} \\
%   \phantom{1} & \phantom{2} & \phantom{3} & \phantom{4}
% 		& \phantom{55555555555} & \phantom{6}
%  \end{array}
% \end{align*}

The weight of $M$ is $\sum_{j=1}^n (n+1-j)f_j$.  
This is the weight of a partition in which each part 1, 2, \ldots, $n$
appears at least once.  
As such, its conjugate is a partition into $n$ different parts, 
and is generated by $\frac{ q^{\binom{n+1}{2}}}{ (q; q)_n }$.  
Each such partition may be obtained multiple times 
on the path diagram.
To be more precise, let $a_n$ be the number of chains of length $n$ 
starting at the node corresponding to the zero slice.  
We now have 
\begin{align}
\label{eqDistPartGenFuncRaw}
  D_c(q) = \sum_{M} q^{M} = \sum_{n \geq 0} \frac{ a_n q^{ \binom{n+1}{2} } }{ (q; q)_n }.  
\end{align}
The middle sum is over all cylindric partitions $M$ with profile $c$ into distinct parts.  

% The rest of the proof is computational.  
We claim that depending on the value $s \equiv n \pmod{r}$, 
where $s = 0, 1, \ldots, r-1$, and $r$ is the rank of the cylindric partition, 
or the length of the composition $c$, 
\begin{align}
\nonumber 
  a_{nr+s} 
  = & \alpha_{s, 1}^0 \left( \gamma_1^0 \right)^n + \cdots 
    + \alpha_{s, A_0}^0 \left( \gamma_{A_0}^0 \right)^n \\ 
\nonumber 
  & + \alpha_{s, 1}^1 n \left( \gamma_1^1 \right)^n + \cdots 
    + \alpha_{s, A_1}^1 n \left( \gamma_{A_1}^1 \right)^n \\ 
\nonumber 
  & \vdots \\ 
\label{eqCharacterize_a_mod_r}
  & + \alpha_{s, 1}^m n^m \left( \gamma_1^m \right)^n + \cdots 
    + \alpha_{s, A_m}^m n^m \left( \gamma_{A_m}^m \right)^n 
\end{align}
for certain $m \in \mathbb{Z}^{+}$, $A_0, \ldots, A_m \in \mathbb{N}$, 
doubly indexed $\gamma_{\cdots}^{\cdots} \in \mathbb{C}$, 
and trebly indexed $\alpha_{\cdots, \cdots}^{\cdots} \in \mathbb{C}$,
except for a few initial values.
The $\alpha_{\cdots, \cdots}^{\cdots}$'s depend on $s$, 
but $\gamma_{\cdots}^{\cdots}$'s do not.  
We delay establishing this claim, 
and present the remaining part of the proof first.  
Substituting \eqref{eqCharacterize_a_mod_r} in \eqref{eqDistPartGenFuncRaw}, 
and adjusting for the first few terms with the help of a polynomial $rg_c(q)$,
we have 
\begin{align}
\nonumber
  D_c(q) = rg_c(q) + \sum_{s = 0}^{r-1} \sum_{n \geq 0}
    \frac{ a_{nr+s} q^{ \binom{nr+s+1}{2} } }{ (q; q)_{nr+s} } 
  = rg_c(q) + \sum_{s = 0}^{r-1} \sum_{j = 0}^m \sum_{k = 1}^{A_k} \alpha_{s, k}^j \sum_{n \geq 0}
    \frac{ n^j \left( \gamma_k^j \right)^n q^{ \binom{nr+s+1}{2} } }{ (q; q)_{nr+s} }.  
\end{align}
The outermost triple sum is finite, 
so it is enough to work on the innermost series on $n$.  
Simplifying notation, we will show that 
\begin{align}
\nonumber
  \sum_{n \geq 0} \frac{ n^j \gamma^n q^{ \binom{nr+s+1}{2} } }{ (q; q)_{nr+s} }
\end{align}
is a finite linear combination of infinite products, 
and infinite products multiplied by Lambert series.  
We will examine two different cases $j = 0$ and $j > 0$.  

When $j = 0$, observe that 
\begin{align}
\nonumber 
  \sum_{n \geq 0} \frac{ \gamma^n q^{ \binom{nr+s+1}{2} } }{ (q; q)_{nr+s} }
  = \left( \sqrt[r]{\gamma} \right)^{-s} \sum_{n \geq 0} 
  \frac{ \left( \sqrt[r]{\gamma} )\right)^{nr+s} q^{ \binom{nr+s+1}{2} } }{ (q; q)_{nr+s} }, 
\end{align}
where one can use the principal branch for the $r$th root.  
The latter series is the $s \pmod{r}$ indexed terms of 
\begin{align}
\nonumber 
  ( -\sqrt[r]{\gamma}q; q)_\infty = \sum_{n \geq 0}
  \frac{ \left( \sqrt[r]{\gamma} )\right)^{n} q^{ \binom{n+1}{2} } }{ (q; q)_{n} }, 
\end{align}
where the identity is Euler's~\cite{GR}.  

It is possible to extract the $s \pmod{r}$ indexed terms of the series
on the right-hand side of the displayed equation above.
This must be a well-known procedure for power series, 
but we outline it here.  
If $r$ is a prime number, we'll just add
\begin{align}
\nonumber 
  \xi^{-js}( -\xi^j \sqrt[r]{\gamma}q; q)_\infty = \sum_{n \geq 0}
  \frac{ \xi^{j(n-s)} \left( \sqrt[r]{\gamma} )\right)^n 
    q^{ \binom{n+1}{2} } }{ (q; q)_{n} }, 
\end{align}
side by side for $j = 0, 1, \ldots, r-1$ and divide by $r$ 
for a primitive $r$th root of unity, e.g. $\xi = \mathrm{e}^{ \frac{2 \pi i}{r} }$.  
We rely on the facts that $1 + \xi + \xi^2 + \cdots + \xi^{r-1} = 0$ and $\xi^r = 1$.
If $r$ is not a prime, then we place the desired indices 
in nested arithmetic progressions, 
and extract terms in one progression at a time.  
For example, for obtaining terms with indices $45n+23$, 
% we first extract the terms indexed $45n + 23$,
we rewrite $ 45n + 23 = 5( 3 (3n + 1) + 1 ) + 3 $,
and extract the $5n + 3$ terms and reindex first, and from among them
the $3n+1$ terms and reindex, and finally the $3n+1$ terms from inside the last batch.

When $j > 0$, we need another twist.  
This time, we have the series 
\begin{align}
\nonumber 
  \sum_{n \geq 0} 
  \frac{ n^j \left( \sqrt[r]{\gamma} )\right)^{nr+s} 
    q^{ \binom{nr+s+1}{2} } }{ (q; q)_{nr+s} }
\end{align}
to work on.  
We use the fact that it is possible to express $n^j$ as a linear combination 
of $(nr+s)^j$, $(nr+s)^{j-1}$, \ldots, $(nr+s)^0$.  
We regard $n$ as the indeterminate, and $r$ and $s$ as fixed parameters.  
We write 
\begin{align}
\nonumber 
  n^j = b_j (nr+s)^j + b_{j-1}(nr+s)^{j-1} + \cdots + b_0 (nr+s)^0, 
\end{align}
and identify coefficients of $n^j$, $n^{j-1}$, \ldots, $n^0$ 
after expansion of binomials and collecting the like powers.  
The resulting linear system for $b_{\cdots}$'s will have a lower triangular 
coefficient matrix with nonzero entries on the diagonal, 
and not all right hand sides are zeros.  
Thus, there is a unique non-trivial solution.  
Therefore, without loss of generality, 
we can work on 
\begin{align}
\nonumber 
  \sum_{n \geq 0} 
  \frac{ (nr+s)^j \left( \sqrt[r]{\gamma} )\right)^{nr+s} 
    q^{ \binom{nr+s+1}{2} } }{ (q; q)_{nr+s} }
  = \left[ \left( z \frac{\mathrm{d}}{\mathrm{d} z} \right)^{j} \sum_{n \geq 0}
  \frac{ \left( \sqrt[r]{\gamma} )\right)^{nr+s} z^{nr+s}
    q^{ \binom{nr+s+1}{2} } }{ (q; q)_{nr+s} } \right]_{z = 1}
\end{align}
Along with the linearity of the derivative, 
the rest is done similar to the $j = 0$ case.  

The proof will be finished once we show \eqref{eqCharacterize_a_mod_r}.  
We denote the number of paths in the path diagram 
ending in a specific slice $\Sigma$ by $a_n^\sigma$, 
where $n = \vert \Sigma \vert$ and $\sigma$ is the shape of $\Sigma$.  
It follows that
\begin{align}
\label{eq_a_n_as_a_n_sigma_sum}
  a_n = \sum_{\sigma} a_n^{\sigma}, 
\end{align}
where the sum is over all shapes.  
We also have 
\begin{align}
\label{eq_a_n_as_sum_of_a_n_minus_ones} 
  a_n^{\sigma} = a_{n-1}^{\tau_1} + a_{n-1}^{\tau_2} + \cdots + a_{n-1}^{\tau_t}.  
\end{align}
The sum on the right hand side is over all slices with weight $(n-1)$ 
connected to $\Sigma$ with an edge in the path diagram.  
When the rank $r > 1$, none of the $\tau_j$'s may be equal to $\sigma$, 
since their weights $(n-1)$ and $n$ are not congruent $\pmod{r}$.  
To incorporate all shapes, we can extend the sum as 
\begin{align}
\nonumber 
  a_n^{\sigma} = \varepsilon_1 a_{n-1}^{\tau_1} + \cdots + \varepsilon_u a_{n-1}^{\tau_u}, 
\end{align}
where $\varepsilon_j$ is 0 or 1 depending on the presence of an edge 
between the implied slices.  
Writing the equation for each slice of weight $n$, we first have
\begin{align}
\nonumber 
  a_n^{\sigma_1} = & \varepsilon_{11} a_{n-1}^{\tau_1} + \cdots 
                  + \varepsilon_{1u} a_{n-1}^{\tau_u} \\ 
\nonumber 
  & \vdots \\ 
\nonumber 
  a_n^{\sigma_v} = & \varepsilon_{v1} a_{n-1}^{\tau_1} + \cdots 
                  + \varepsilon_{vu} a_{n-1}^{\tau_u}, 
\end{align}
then, by iteration 
\begin{align}
\label{eq_a_n_shifted_r_1} 
  a_n^{\sigma_1} = & \alpha_{11} a_{n-r}^{\sigma_1} + \cdots 
                  + \alpha_{1u} a_{n-r}^{\sigma_v} \\ 
\nonumber 
  & \vdots \\ 
\label{eq_a_n_shifted_r_v} 
  a_n^{\sigma_v} = & \alpha_{v1} a_{n-r}^{\sigma_1} + \cdots 
                  + \alpha_{vv} a_{n-r}^{\sigma_v}, 
\end{align}
where the $\alpha_{ij}$'s are completely determined.  
This shows us the dependence of $a_n$ on the value $s \equiv n \pmod{r}$.
Different $s$'s may, and do, correspond to different systems of equations.
However; the left end, or the smaller end, of the path diagram
may not fit this characterization.
We noted that the path diagram is eventually periodic, rather than periodic.
Therefore, we may need to treat the first few values separately.
With this initial exception in mind,
\begin{align}
\nonumber 
  a_n \left[ = a_n^{\sigma_1} + \cdots + a_n^{\sigma_v} \right] 
  = \left( \alpha_{11} + \cdots + \alpha_{1v} \right) a_{n-r}^{\sigma_v} 
    + \cdots 
    + \left( \alpha_{v1} + \cdots + \alpha_{vv} \right) a_{n-r}^{\sigma_v}.  
\end{align}
Now we choose and fix $s$, replace $a_{nr+s}^{\sigma_j}$ by $b_n^j$, 
and, slightly abusing notation, focus on 
\begin{align}
\label{eq_b1} 
  b_n^1 - & \alpha_{11} b_{n-1}^1 - \cdots - \alpha_{1v} b_{n-1}^v = 0 \\ 
\nonumber 
  & \vdots \\ 
\label{eq_bv} 
  b_n^v - & \alpha_{v1} b_{n-1}^1 - \cdots - \alpha_{vv} b_{n-1}^v = 0
\end{align}
with $b_n = b_n^1 + \cdots + b_n^v$.  
Our goal is to find a recurrence for $b_n$.  
To construct that, 
we first add \eqref{eq_b1} through \eqref{eq_bv} side by side and obtain
\begin{align}
\label{eq_b}
  b_n - \left( \alpha_{11} + \cdots + \alpha_{v1} \right) b_{n-1}^1 
  - \cdots 
  - \left( \alpha_{1v} + \cdots + \alpha_{vv} \right) b_{n-1}^v 
  = 0.   
\end{align}
Then, we rewrite \eqref{eq_b1} through \eqref{eq_bv} replacing $n$ by $(n-1)$, 
vertically aligning $b_{n-1}^j$'s and $b_{n-2}^j$'s for convenience.  
Next, multiply \eqref{eq_b} through by $\beta^0$, 
and the 1-shifted \eqref{eq_b1} through \eqref{eq_bv} by 
$\beta_1^1$, \ldots, $\beta_v^1$, in their respective order.  
After that, we add these $v+1$ equations side by side, 
and identify the coefficients of $b_{n-1}^1$ through $b_{n-1}^v$, 
and those of $b_{n-2}^1$ through $b_{n-2}^v$.  
Because once those two sets of coefficients are identical among themselves, 
the sums of $b_{n-1}^j$'s and $b_{n-2}^j$'s may be replaced by 
$b_{n-1}$ and $b_{n-2}$, respectively.  
This will yield a recurrence for $b_n$.  

So far, we have $2v-2$ homogeneous equations for $v+1$ variables, 
namely $\beta^0$, $\beta_j^1$ for $j=$ 1, \ldots, $v$.  
For general $v$, $2v-2$ may be greater than or equal to $v+1$, 
so we are not guaranteed a non-trivial solution.  
In that case, we replace $n$ by $(n-2)$ in \eqref{eq_b1} through \eqref{eq_bv}, 
multiply them through $\beta_j^2$ for $j = $ 1, \ldots, $v$ in their respective order, 
and incorporate in the above picture.  
Then, we will have $3v-3$ homogeneous equations for $2v+1$ variables.  
If there are no non-trivial solutions yet, 
we keep repeating the same procedure, 
replacing $n$ by $(n-3)$, \ldots, $(n-f+1)$ in \eqref{eq_b1} through \eqref{eq_bv}, 
as many times as necessary to 
have greater number of variables than the number of equations, i.e. 
$(f-1)v + 1 > fv - f$.  
Equivalently, we want $f > v-1$, therefore $f = v$ will work.  
We are guaranteed to have a recurrence of degree $f = v$ for $b_n$.  
\begin{align}
\nonumber 
  A_0 b_n + A_1 b_{n-1} + \cdots + A_v b_{n-v} = 0.  
\end{align}
Finding $v$ initial values $b_0$, \ldots, $b_{v-1}$,
and deciding how many of them at the beginning do not fit the recurrence is routine.

By the well-known characterization of sequences 
satisfying linear recurrences with constant coefficients~\cite[Proposition 4.2.2]{Stanley},
we arrive at \eqref{eqCharacterize_a_mod_r}, concluding the proof.  
\end{proof}

We will give a concrete example a little further below,
backed up by some shortcuts.
We will also comment on why those shortcuts were not used in the proof.

Here, too, we can read more into the above proof.
% We will make some further observations and we give concrete examples.
The demonstration of \eqref{eqCharacterize_a_mod_r} requires 
the use of the fundamental theorem of algebra~\cite{DF}. 
So, despite all the construction, 
Theorem \ref{thmDistPartsGeneral} is an existence result.  
The explicit formulas depend on our capacity to solve 
the recurrence for $b_n$'s pertained to in the last part of the proof.  

We can adapt the last part of the proof 
to find a recurrence for $a_{nr+s}^{\sigma}$ for certain $s$ and $\sigma$, 
and hence obtain a similar formula 
to \eqref{eqCharacterize_a_mod_r} for $a_{nr+s}^{\sigma}$.  
Thanks to equations \eqref{eq_a_n_as_a_n_sigma_sum}
and \eqref{eq_a_n_as_sum_of_a_n_minus_ones}, 
the same $\gamma_{\cdots}^{\cdots}$ are shared by all $a_{nr+s}^{\sigma}$'s
and $a_{nr+s}$'s, 
independent of $s$ and $\sigma$.  

In the shape transition graph, 
if we reorder the shapes so that the ones with weights $\equiv 0 \pmod{r}$ 
are listed first, those with weight $\equiv 1 \pmod{r}$ are listed next, etc., 
the adjacency matrix~\cite{Harary-Palmer} $A$ 
will have the following $r \times r$ block diagonal form.   
\begin{align}
\nonumber 
  A = \begin{bmatrix}
    0 & \cdots & & 0 & \aast \\
    \aast & & & & 0 \\ 
    0 & \aast & & & 0 \\ 
    \vdots & & \ddots & & \vdots \\ 
    0 & \cdots & 0 & \aast & 0
  \end{bmatrix} 
\end{align}
In other words, it will be a block permutation matrix with square matrices, 
not necessarily of the same size, on the diagonal.  
For convenience, we can take weights of shapes as partitions.  
It is possible to incorporate the shape of zero, 
and offset all other weights by the same amount.  
But this will not essentially change the discussion.  

For example, if we reorder the shapes as
$(0, 0)$, $(2, 1)$, $(1, 0)$, $(2, 2)$, $(1, 1)$,  $(2, 0)$
in rank $r = 3$ and level $\ell = 2$ cylindric partitions,
we will obtain the block matrix
\begin{align}
\nonumber
	A = \begin{blockarray}{ccccccc}
	& (0, 0) & (2, 1) & (1, 0) & (2, 2) & (1, 1) & (2, 0) \\
		\begin{block}{c[cc|cc|cc]}
			(0, 0) & 0 & 0 & 0 & 0 & 1 & 0 \\
			(2, 1) & 0 & 0 & 0 & 0 & 1 & 1 \\
			\cmidrule{2-7}
			(1, 0) & 1 & 1 & 0 & 0 & 0 & 0 \\
			(2, 2) & 0 & 1 & 0 & 0 & 0 & 0 \\
			\cmidrule{2-7}
			(1, 1) & 0 & 0 & 1 & 1 & 0 & 0 \\
			(2, 0) & 0 & 0 & 1 & 0 & 0 & 0 \\
		\end{block}
	\end{blockarray}.
\end{align}

The adjacency matrix $A$ carries the same information 
as the equation \eqref{eq_a_n_as_sum_of_a_n_minus_ones}, 
and its $r$th power $A^r$, the following block diagonal matrix, 
as \eqref{eq_a_n_shifted_r_1} through \eqref{eq_a_n_shifted_r_v}.  
\begin{align}
\nonumber 
  A^r = \begin{bmatrix}
    \aast & & 0 \\ 
    & \ddots & \\ 
    0 & & \aast 
  \end{bmatrix}
\end{align}

In the running example,
\begin{align}
\nonumber
	A^3 = \begin{blockarray}{ccccccc}
	& (0, 0) & (2, 1) & (1, 0) & (2, 2) & (1, 1) & (2, 0) \\
		\begin{block}{c[cc|cc|cc]}
			(0, 0) & 1 & 2 & 0 & 0 & 0 & 0 \\
			(2, 1) & 2 & 3 & 0 & 0 & 0 & 0 \\
			\cmidrule{2-7}
			(1, 0) & 0 & 0 & 3 & 2 & 0 & 0 \\
			(2, 2) & 0 & 0 & 2 & 1 & 0 & 0 \\
			\cmidrule{2-7}
			(1, 1) & 0 & 0 & 0 & 0 & 3 & 2 \\
			(2, 0) & 0 & 0 & 0 & 0 & 2 & 1 \\
		\end{block}
	\end{blockarray}.
\end{align}

The establishment of $a_{nr+s}^{\sigma}$ as a closed formula 
and the following arguments above show that 
all blocks on the diagonal of $A^r$ have 
the same non-zero eigenvalues with the same multiplicities.  
Consequently, or by rank arguments, 
the bigger blocks have zeros, or more zeros, as their extra eigenvalues.  
For example, the eigenvalues of the above $A^3$ are $2 \pm \sqrt{5}$.

If we worked with rank $r = 3$ and level $\ell = 3$ cylindric partitions,
the blocks of $A^3$ would have been
\begin{align}
\nonumber
	\begin{bmatrix}
		1 & 2 & 0 & 1 \\ 2 & 6 & 2 & 2 \\ 1 & 2 & 1 & 0 \\ 0 & 2 & 1 & 1
	\end{bmatrix},
	\begin{bmatrix}
		3 & 3 & 2 \\ 2 & 3 & 3 \\ 3 & 2 & 3
	\end{bmatrix}, \textrm{ and }
	\begin{bmatrix}
		3 & 2 & 3 \\ 3 & 3 & 2 \\ 2 & 3 & 3
	\end{bmatrix}.
\end{align}
The eigenvalues of both of the smaller matrices are 8 and $\frac{1}{2} \pm \frac{\sqrt{3}}{2}i$,
and the larger matrix has the additional eigenvalue 0.

Back to the rank $r = 3$ and level $\ell = 2$ example,
we see that the matrix $A^3$ is diagonalizable.
This is because the eigenvalues of each block are distinct.
It follows that $a_{3n+j} = A_j (2 + \sqrt{5})^n + B_j (2 - \sqrt{5})^n$
for $j = 0, 1, 2$.
The generating function $D_{(2, 0, 0)}(q)$ will be
a linear combination of
\begin{align}
\nonumber
	( - \sqrt[3]{2 + \sqrt{5}} q ; q)_\infty,
	& ( - \sqrt[3]{2 - \sqrt{5}} q ; q)_\infty,
	( - \omega \sqrt[3]{2 + \sqrt{5}} q ; q)_\infty,
	( - \omega \sqrt[3]{2 - \sqrt{5}} q ; q)_\infty, \\
\nonumber
	& ( - \omega^2 \sqrt[3]{2 + \sqrt{5}} q ; q)_\infty, \textrm{ and }
	( - \omega^2 \sqrt[3]{2 - \sqrt{5}} q ; q)_\infty,
\end{align}
where $\omega^2 + \omega + 1 = 0$, i.e. $\omega$ is a primitive third root of unity.
We can find the unknown coefficients by comparing
the Maclaurin series of the linear combination of the above infinite products
and that of \eqref{eqDistPartGenFuncRaw}.
Therefore, we need some initial $a_n$'s.
We can do it by placing 1 in the zero node,
and calculating $a_n^\sigma$ for each slice
with weight $n$ and shape $\sigma$
as the sum of the numbers in the immediately preceding slices that are connected to it.
The final touch is adding the numbers in all slices with the same weight,
and hence finding $a_n$.
An example for profile $(2,0,0)$ is shown in Figure \ref{figNum_paths_r3l2}.
\begin{figure}
  \centering
  \includegraphics[scale=0.1]{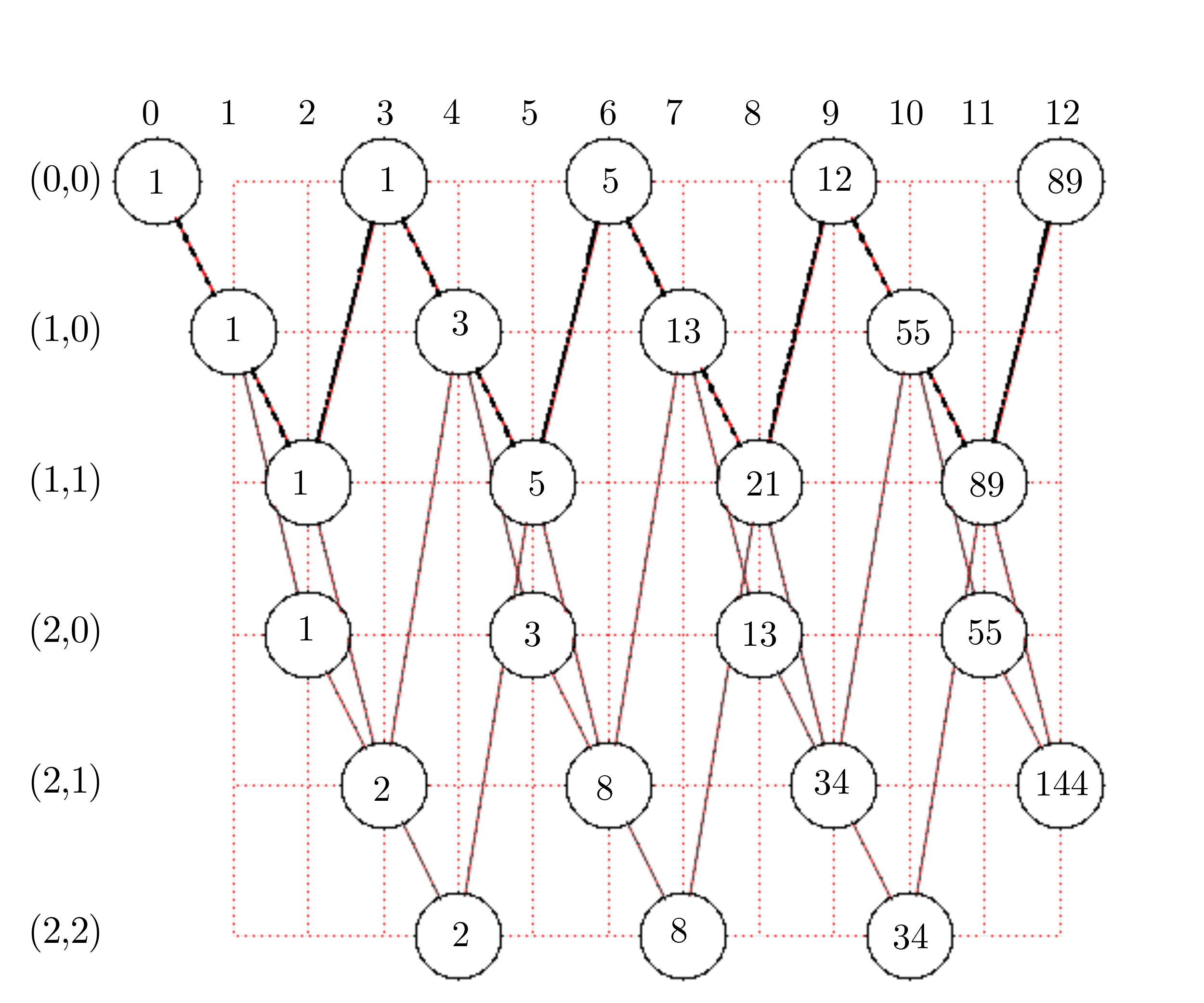}
  \caption{ The number of paths ending in each node.  }
\label{figNum_paths_r3l2}
\end{figure}
We caution the reader that Figure \ref{figNum_paths_r3l2} is a different
utilization of the path diagram.
In particular, different from the slice counting one.
Then, it can be proven that $a_n$'s are the (shifted)
Fibonacci numbers by strong induction~\cite{Rosen}.
So,
\begin{align}
\nonumber
	a_n = A \left(\frac{(1 + \sqrt{5})}{2}\right)^n + B \left(\frac{(1 - \sqrt{5})}{2}\right)^n.
\end{align}
This simplifies the computations, as it is now clear that
$D_{(2,0,0)}(q)$ is a linear combination of
\begin{align}
\nonumber
	\left( - \frac{(1 + \sqrt{5})}{2} q ; q \right)_\infty \textrm{ and }
	\left( - \frac{(1 - \sqrt{5})}{2} q  ; q \right)_\infty,
\end{align}
and there are no rule breaking terms at the beginning.
The comparison of the first few terms of the Maclaurin series mentioned above
will yield
\begin{align}
 D_{(2,0,0)}(q) =
 \left( \frac{1}{2} + \frac{\sqrt{5}}{10} \right) \left( - \frac{(1 + \sqrt{5})}{2} q ; q \right)_\infty
 + \left( \frac{1}{2} - \frac{\sqrt{5}}{10} \right) \left( - \frac{(1 - \sqrt{5})}{2} q ; q \right)_\infty .
\end{align}

The reason that we did not use this matrix characterization in the proof is that
we do not know up front if the constructed matrix blocks are diagonalizable,
let alone if their non-zero eigenvalues are distinct.
For example, a simple computer search reveals that 
the rank $r=5$ and level $\ell = 5$ case involves a 
repeated eigenvalue.  
The smallest block matrix has size $25 \times 25$ in this case.  
Some other examples for larger $r+\ell$ sums are 
$(r, \ell)=$ $(5, 7)$, $(5, 9)$, $(6, 6)$, $(6, 8)$, $(7, 5)$, $(7, 7)$, 
$(8, 6)$, $(9, 5)$, $(10, 5)$, 
with respective sizes of the smallest block matrix 
$66 \times 66$, $143 \times 143$, $75 \times 75$, $212 \times 212$, 
$66 \times 66$, $245 \times 245$, 
$212 \times 212$, $143 \times 143$, $200 \times 200$.  
We did the calculations for $r, \ell \leq 10$ and 
discarded matrices larger than $250 \times 250$.  

It is also an instructive exercise to find $a_0$, $a_1$, $a_2$, \ldots
for the profile $(1,1,1)$ case.
They are 1, 3, 6, 12, 24, 48, 96, 192,
i.e. $3 \cdot 2^{n-1}$ except for $a_0$.
This can be proven, again, by strong induction~\cite{Rosen},
exploiting the structure of the path diagram.
Then, not even requiring the corresponding eigenvalues listed above, one calculates
\begin{align}
\nonumber
	D_{(1,1,1)} = -\frac{1}{2} + \frac{3}{2} (- 2q; q)_\infty
\end{align}

%%%%%% birader salak mısın?  kim diyor bunu?
% Last, but not least, since $A$ has non-negative entries,
% any of its real eigenvalues must be non-negative.
% Same is true for $A^r$.

When the rank $r = 2$, the picture is much nicer.  
In the shape transition graph, 
there is an edge from $\sigma$ to $\tau$ 
if and only if there is an edge from $\tau$ to $\sigma$.  
One of those edges corresponds to placing a square 
at the right end of the upper row of a shape, 
and the other edge to that of the lower row of the other shape.  
Thus, the adjacency matrix 
$A = \begin{bmatrix} 0 & \aast \\ \aast & 0 \end{bmatrix}$ 
described above is a symmetric matrix.  
So is the block diagonal matrix $A^2$.  
As such, $A^2$ has real eigenvalues, and it is diagonalizable~\cite{LinAlg}.
Also, the blocks of $A^2$ have the same non-zero eigenvalues with the same multiplicities.
It follows that in equation \eqref{eqCharacterize_a_mod_r} 
$\alpha_{\cdots, \cdots}^j = 0$ if $j > 0$, 
and Theorem \ref{thmDistPartsGeneral} reduces to 

\begin{cor}
\label{corDistPartsRank2}
  With notation as in Theorem \ref{thmDistPartsGeneral}, 
	\begin{align}
	\nonumber
		D_{(a, b)}(q)
		= rg_{(a,b)}(q) & + \sum_{i = 1}^B A_i
			\left( \frac{ ( - \sqrt{\beta_i} q; q )_\infty + ( \sqrt{\beta_i} q; q )_\infty }
				{2} \right) \\
\nonumber
		& + \sum_{i = 1}^B B_i
			\left( \frac{ ( - \sqrt{\beta_i} q; q )_\infty - ( \sqrt{\beta_i} q; q )_\infty }
				{2\sqrt{\beta_i}} \right)
	\end{align}
	for certain $B \in \mathbb{Z}^{+}$, $\beta_i$, $A_i$, $B_i$ $\in \mathbb{R}$,
	$i = 1, 2$ \ldots, $B$, and a polynomial $rg_{(a,b)}(q)$.
\end{cor}

It is possible to write the above theorem as
\begin{align}
\nonumber
	D_{(a, b)}(q) = rg_{(a,b)}(q) + \sum_{i = 1}^B \alpha_i ( -\beta_i q; q )_\infty
\end{align}
for certain $B \in \mathbb{Z}^{+}$, $\alpha_i$, $\beta_i$ $\in \mathbb{R}$,
$i = 1, 2$ \ldots, $B$,
but the statement emphasizes the nature of the proof.
It is Theorem \ref{thmDistPartsGeneral}, the observations following it,
and the dissection of series as done in~\cite[Section 3]{KO}.  

Although we are guaranteed diagonalization in the rank $r = 2$ case, 
a small computer search detects no repeated eigenvalues for levels $\ell \leq 500$.  

As another example, we calculate $D_{(4,0)}(q)$.
We order the shapes as $(0)$, $(2)$, $(4)$, $(1)$, $(3)$,
write the adjacency matrix of the corresponding shape transition graph,
and calculate its square.
\begin{align}
\nonumber
	A = \begin{blockarray}{cccccc}
	& (0) & (2) & (4) & (1) & (3) \\
		\begin{block}{c[ccc|cc]}
			(0) & 0 & 0 & 0 & 1 & 0 \\
			(2) & 0 & 0 & 0 & 1 & 1 \\
			(4) & 0 & 0 & 0 & 0 & 1 \\
			\cmidrule{2-6}
			(1) & 1 & 1 & 0 & 0 & 0 \\
			(3) & 0 & 1 & 1 & 0 & 0 \\
		\end{block}
	\end{blockarray}, \qquad
	A^2 = \begin{blockarray}{cccccc}
	& (0) & (2) & (4) & (1) & (3) \\
		\begin{block}{c[ccc|cc]}
			(0) & 1 & 1 & 0 & 0 & 0 \\
			(2) & 1 & 2 & 1 & 0 & 0 \\
			(4) & 0 & 1 & 1 & 0 & 0 \\
			\cmidrule{2-6}
			(1) & 0 & 0 & 0 & 2 & 1 \\
			(3) & 0 & 0 & 0 & 1 & 2 \\
		\end{block}
	\end{blockarray}.
\end{align}
The smaller block in $A^2$ has eigenvalues 3 and $-1$.
Even if it had repeated eigenvalues,
we know that $A^2$ is diagonalizable because it is symmetric.
Either way, both $a_{2n}$ and $a_{2n+1}$ are linear combinations of $3^n$ and $(-1)^n$.
Next, we calculate the number of paths starting at the zero slice
and ending in a specific slice.
We record these numbers on the path diagram,
being careful not to confuse it with representations of a cylindric partition,
please see Figure \ref{figPath_diag_r2_l4}.
\begin{figure}
  \centering
  \includegraphics[scale=0.1]{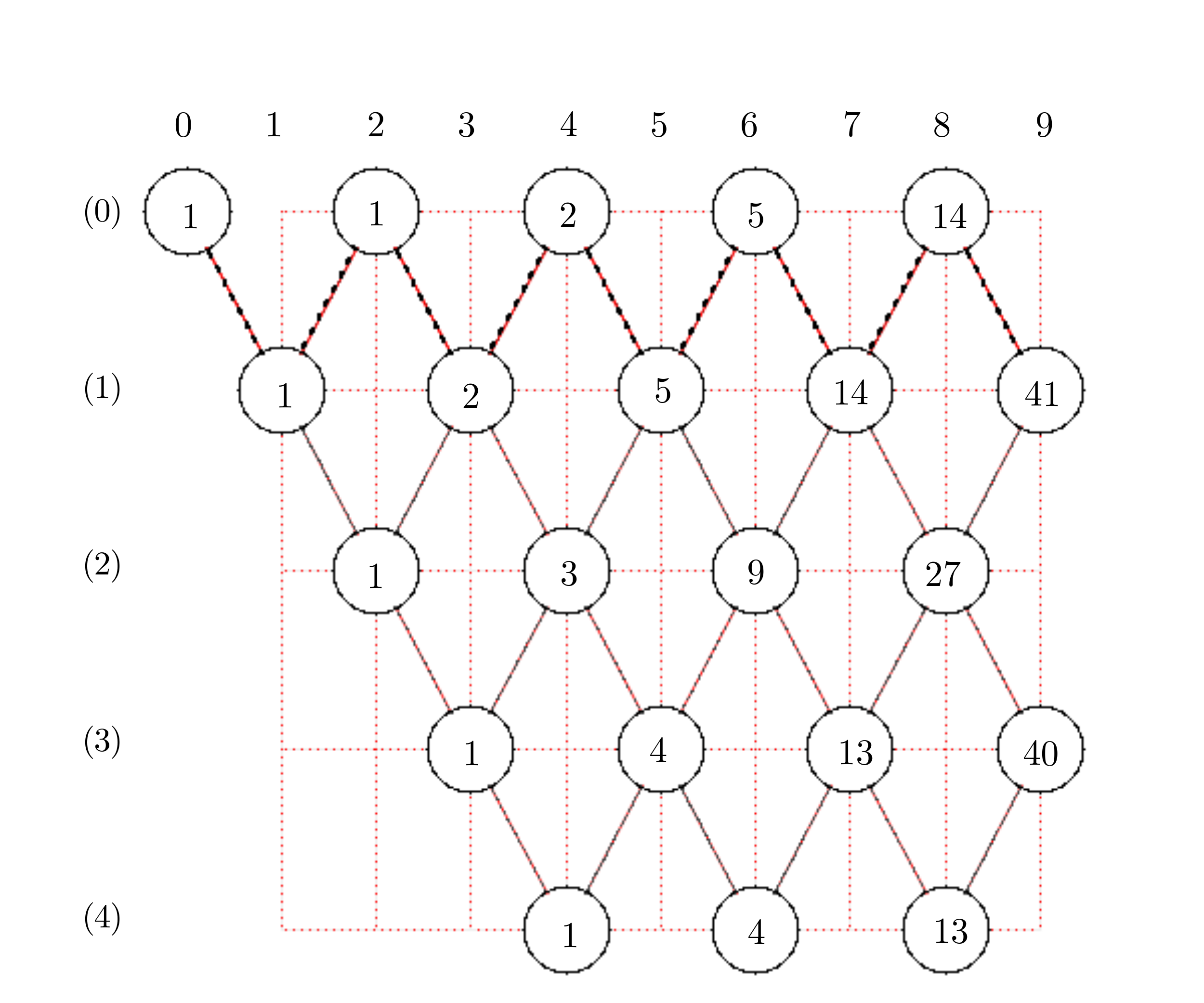}
  \caption{ Another example of the number of paths ending in each node.  }
\label{figPath_diag_r2_l4}
\end{figure}
Then, we see that $a_{2n} = 2\cdot3^{n-1}$ except for $a_0$,
and $a_{2n+1} = 3^n$.
It is possible, but not necessary to use Figure \ref{figPath_diag_r2_l4}
to find recurrences for $a_{2n}^\sigma$ or for $a_{2n+1}^\sigma$.
The rest is standard $q$-series manipulations.
\begin{align}
\nonumber
  D_{(4,0)}(q)
  = 1 + \sum_{n \geq 0} \frac{ 3^n q^{\binom{2n+2}{2}} }{(q; q)_{2n+1}}
  + 2 \sum_{n \geq 1} \frac{ 3^{n-1} q^{\binom{2n+1}{2}} }{(q; q)_{2n}}
\end{align}
\begin{align}
\nonumber
  = \frac{1}{3} + \frac{1}{\sqrt{3}}
    \sum_{n \geq 0} \frac{ (\sqrt{3})^{2n+1} q^{\binom{2n+2}{2}} }{(q; q)_{2n+1}}
  + \frac{2}{3} \left[ -1
    + \sum_{n \geq 0} \frac{ (\sqrt{3})^{2n} q^{\binom{2n+1}{2}} }{(q; q)_{2n}} \right]
\end{align}
\begin{align}
\nonumber
  = 1 + \frac{1}{\sqrt{3}} \left[ \frac{ ( \sqrt{3}q; q)_\infty - ( - \sqrt{3}q; q)_\infty }{2} \right]
  + \frac{2}{3} \left[ \frac{ ( \sqrt{3}q; q)_\infty + ( - \sqrt{3}q; q)_\infty }{2}  - 1 \right]
\end{align}
\begin{align}
\nonumber
  = \frac{1}{3}
  + \left( \frac{2 + \sqrt{3}}{6} \right) ( \sqrt{3}q; q)_\infty
  + \left( \frac{2 - \sqrt{3}}{6} \right) ( - \sqrt{3}q; q)_\infty
\end{align}

\section{Proof of part of a conjecture by Corteel, Dousse, and Uncu, and more}
\label{secCDU}

We revisit and relabel \eqref{PO_slices} for conformity with 
Corteel, Dousse, and Uncu~\cite{CDU}.  
\begin{align}
\label{PO_slices_again}
  {}_{1}\Sigma \geq {}_{2}\Sigma \geq \cdots \geq {}_{n}\Sigma > E 
\end{align}
These are the slices of a cylindric partition $\Lambda$ 
with profile $c = (c_1, c_2, \ldots, c_r)$
and $\mathrm{max}(\Lambda) = n$.  

Now, for each $j = 1, 2, \ldots, n$, in this order,
calculate 
\begin{align}
\nonumber 
  f_j := \mathrm{min}_{1 \leq i \leq r}
   \left\{ \ell({}_{j}\sigma^{(i)}) - \ell({}_{j+1}\sigma^{(i)}) \right\}.  
\end{align}
We take $E = {}_{n+1}\Sigma$ here.  
Then, in ${}_{1}\Sigma$, ${}_{2}\Sigma$, \ldots, ${}_{j}\Sigma$, 
delete $f_j$ 1's from the right ends of each ${}_{\cdots}\sigma^{(i)}$, 
$i = 1, 2, \ldots, r$.
We update $\Lambda$ accordingly, as 
$\Lambda = {}_{1}\Sigma + {}_{2}\Sigma + \cdots + {}_{n}\Sigma$.  
This transformation reduces $\vert \Lambda \vert$ 
(and hence $\vert {}_{1}\Sigma \vert+$ $\vert {}_{2}\Sigma \vert+$ $\cdots$ 
$+\vert {}_{n}\Sigma \vert$) by $r (j f_j)$.

In the end, \eqref{PO_slices_again} becomes
\begin{align}
\label{PO_slices_base}
  {}_{1}\Sigma \geq {}_{2}\Sigma \geq \cdots \geq {}_{n}\Sigma \geq E.
\end{align}
and a partition is produced on the side
into multiples of $r$ such that parts are at most $rn$.
The reader already noticed that
$(r \cdot 1 \cdot f_1, \ldots, r \cdot n \cdot f_n)$
is a partition into multiples of $r$, all $\leq rn$.
They are thus generated by $\frac{1}{(q^r; q^r)_n}$.

We can visualize this process as shrinking the slices as much as possible.
We keep the left ends, or the profile ends, fixed;
and we let the right ends, or the shape ends, move.
The profile, the shapes, and the inclusion are preserved.
Because we move the right ends as much as possible to the left,
we have
\begin{align}
\label{eqPbaseDiff}
  \vert {}_{j}\Sigma \vert - \vert {}_{j+1}\Sigma \vert
  = \Delta( c_{j+1}, c_j )
\end{align}
where $c_j$ is the shape of ${}_{j}\Sigma$,
for $j = 1, \ldots, n$.

As an example, consider the cylindric partition
\begin{align}
\nonumber
	\Lambda=((3,3,3,2,2,2,2,1,1,1,1,1), (3,3,3,2,2,2,1,1,1,1), (3,3,3,2,2,2,2,1,1,1,1))
\end{align}
with profile $c=(1,1,1)$.
The almost superimposed slices of $\Lambda$ are shown below.
Let us remember that the boxes in dark gray indicate the profile,
and they do not contribute to the weight.
\begin{align}
\label{fig_orig_slices}
 \vcenter{\hbox{\includegraphics[scale=1]{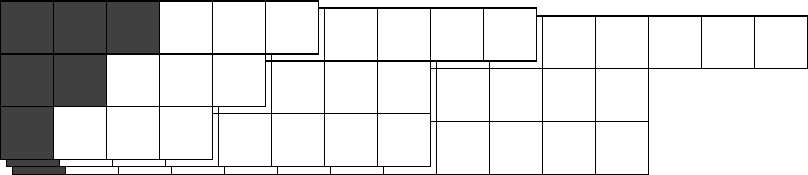}}}
\end{align}
Now, we take out the lightly shaded boxes, and record their numbers.
These numbers are 9+9+9, where 9 is for 3 slices $\times$ 3 rows,
and repetition three times for 3 columns.
The term column is used in a liberal sense, as they are askew because of the profile.
The slices become as the figure on the right as below.
\begin{center}
 \includegraphics[scale=1]{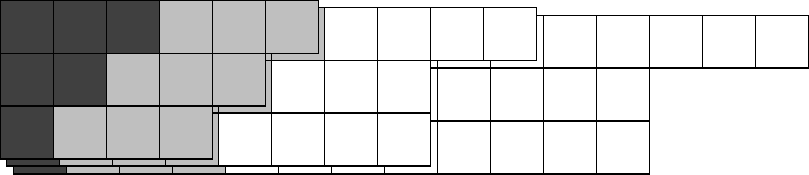}
 \raisebox{6mm}{$\longrightarrow$}
 \includegraphics[scale=1]{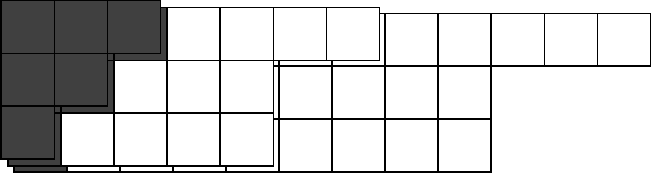}
\end{center}
The smallest slice cannot be shrunk anymore.
In fact, it incidentally became the zero slice.
Now we shrink the second smallest slice, along with the one larger than it.
Again, the lightly shaded boxes are to be taken out.
Their numbers this time are 6+6+6, for similar reasons as in the previous paragraph.
The intermediate stack of slices are shown in the next figure.
\begin{center}
 \includegraphics[scale=1]{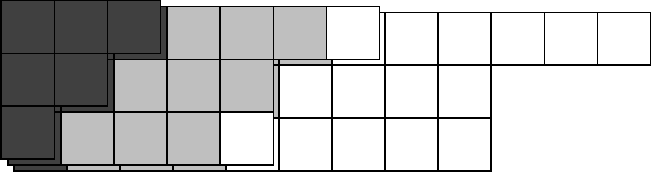}
 \raisebox{6mm}{$\longrightarrow$}
 \includegraphics[scale=1]{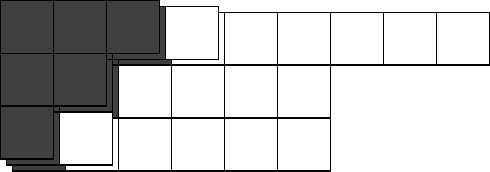}
\end{center}
Lastly, we shrink the largest slice,
and record 3+3+3+3.
\begin{align}
\label{fig_shrunk_slices_1}
 \vcenter{\hbox{\includegraphics[scale=1]{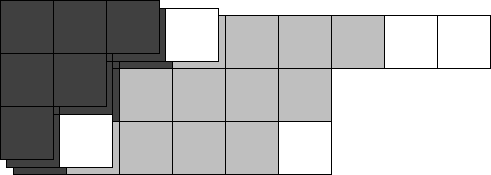}}}
 \longrightarrow
 \vcenter{\hbox{\includegraphics[scale=1]{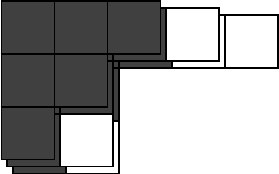}}}
\end{align}

Before concluding anything further,
we would like to draw the reader's attention to a detail here.
The requirement that $\mathrm{max}(\Lambda) = n$ for the original $\Lambda$ 
implies either $\mathrm{max}(\Lambda) = n$ for the final $\Lambda$ 
or $\mathrm{max}(\Lambda) < n$ and $f_n > 0$.  

With this proviso in mind, one can argue that the above process is reversible.  
Namely, given a partition into multiples of $r$ with parts at most $rn$,
and $n$ slices ${}_{1}\Sigma$, ${}_{2}\Sigma$, \ldots ${}_{n}\Sigma$ with profile $c$
satisfying \eqref{PO_slices_base} along with the conditions following it; 
it is possible to combine them and produce a cylindric partition $\Lambda$
of profile $c$ and $\mathrm{max}(\Lambda) = n$.  

There are two ways to make things simpler.  
One is starting with a weak final inequality 
instead of a strict one in \eqref{PO_slices_again}, 
and work with cylindric partitions with profile $c$
having parts at most $n$.  
Observe that the chain of inequalities \eqref{PO_slices_base} 
have extra conditions on the appearing slices.  

i.e. $f_{n, c}(q)$ is the generating function of 
chains of slices satisfying \eqref{PO_slices_again}, 
with the final inequality replaced by a weak one.  
Then, we have reproved the following theorem, borrowing notation from~\cite{CDU}.

\begin{theorem}
  \label{thmCDUconjPrep}
  Let $f_{n, c}(q)$ be the the generating function of cylindric partitions 
  with profile $c=(c_1, \ldots, c_r)$ and parts at most $n$.
  Then, 
  \begin{align}
  \nonumber
    f_{n, c}(q) = \frac{P_{n, c}(q)}{(q^r; q^r)_n},
  \end{align}
  where $P_{0, c}(q) = 1$, and 
  \begin{align}
  \nonumber 
    P_{n, c}(q) = \sum_{d} q^{n \Delta(d, c)} P_{n-1, d}(q), 
  \end{align}
  where the sum is over all possible shapes $d$.  
\end{theorem}

It is convenient to write the formula in Theorem \ref{thmCDUconjPrep} in matrix form.
For instance, for rank $r = 3$ and level $\ell = 2$, one can calculate:
\begin{align}
\label{eqCDUConjPrepMxEx}
  \begin{bmatrix}
   P_{n, (0,0)}(q) \\ P_{n, (1,0)}(q) \\ P_{n, (1,1)}(q) \\
   P_{n, (2,0)}(q) \\ P_{n, (2,1)}(q) \\ P_{n, (2,2)}(q)
  \end{bmatrix}
  = \begin{bmatrix}
     1 & q^{n} & q^{n} & q^{2n} & q^{3n} & q^{4n} \\
     q^{2n} & 1 & q^{n} & q^{n} & q^{2n} & q^{3n} \\
     q^{n} & q^{2n} & 1 & q^{3n} & q^{n} & q^{2n} \\
     q^{4n} & q^{2n} & q^{3n} & 1 & q^{n} & q^{2n} \\
     q^{3n} & q^{n} & q^{2n} & q^{2n} & 1 & q^{n} \\
     q^{2n} & q^{3n} & q^{4n} & q^{4n} & q^{2n} & 1 \\
    \end{bmatrix}
  \begin{bmatrix}
   P_{n-1, (0,0)}(q) \\ P_{n-1, (1,0)}(q) \\ P_{n-1, (1,1)}(q) \\
   P_{n-1, (2,0)}(q) \\ P_{n-1, (2,1)}(q) \\ P_{n-1, (2,2)}(q)
  \end{bmatrix}.
\end{align}

\begin{cor}
\label{coroCDUconj}
  $P_{n, c}(q)$ as described above is a polynomial 
  with positive coefficients, 
  and $P_{n, c}(1) = \binom{\ell + r-1}{r-1}^n$,
  where $\ell$ is the level of the cylindric partitions, 
  and $r$ is the rank.
\end{cor}

The other way to make things simpler is observing that 
${}_{n}\Sigma = E$ is possible only when the shape of ${}_{n}\Sigma$ 
is the shape of zero.  
We can eliminate that case by requiring 
\begin{align}
\nonumber 
  \ell({}_{n}\sigma^{(i)}) \geq 1
\end{align}
for at least one $i$ whenever ${}_{n}\Sigma$ has the shape of zero.
Then, only for the smallest slice,
and only when the smallest slice has the shape of zero,
do we have an additional column as a buffer.

$f_n$ will be one less than it would have been otherwise,
so we do not have the extra condition $f_n > 0$ anymore.
Also, \eqref{PO_slices_base} has ${}_{n}\Sigma > E$
instead of ${}_{n}\Sigma \geq E$,
along with the extra conditions listed afterwards.
Now we can talk about the cylindric partitions $\Lambda$
with $\mathrm{max}(\Lambda) = n$ instead of $\leq n$.

As an example, the cylindric partition
\begin{align}
\nonumber
	\Lambda=((3,3,3,2,2,2,2,1,1,1,1,1), (3,3,3,2,2,2,1,1,1,1), (3,3,3,2,2,2,2,1,1,1,1))
\end{align}
with profile $c=(1,1,1)$ will shrink to
\begin{align}
\label{fig_shrunk_slices_2}
 \vcenter{\hbox{\includegraphics[scale=1]{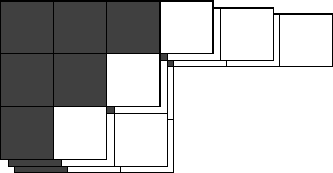}}},
\end{align}
and the partition on the side will be
$9+9+6+6+6+3+3+3+3$,
There is one less 9 than the previous example,
and the shrunk cylindric partition has the same maximum part, here 3.

\begin{theorem}
  \label{thmCDUconjPrepAlt}
  Let $f_{=n, c}(q)$ be the the generating function of cylindric partitions 
  with profile $c=(c_1, \ldots, c_r)$ and largest part $n$.
  Then, 
  \begin{align}
  \nonumber
    f_{=n, c}(q) = \frac{P_{=n, c}(q)}{(q^r; q^r)_n},
  \end{align}
  where $P_{=0, c}(q) = 1$, and 
  \begin{align}
  \nonumber 
    P_{=n, c}(q) = q^{nr}P_{n-1, c}(q) + \sum_{d \neq c} q^{n \Delta(d, c)} P_{n-1, d}(q),
  \end{align}
  where the sum is over all possible shapes $d$ except $c$, 
  and $P_{n, c}(q)$'s are as in Theorem \ref{thmCDUconjPrep}.  
\end{theorem}

Corollary \ref{coroCDUconj} applies to $P_{=n, c}(q)$'s, as well.  
They are polynomials with positive coefficients and 
$P_{=n, c}(1) = \binom{\ell + r-1}{r-1}^n$,
where $\ell$ is the level of cylindrical partitions, and $r$ is the rank.

Again, it is convenient to write the Theorem \ref{thmCDUconjPrepAlt}
in matrix form if the rank and the level are fixed.
For example, for rank $r=3$ and level $\ell = 2$, we have:
\begin{align}
\nonumber
  \begin{bmatrix}
   P_{=n, (0,0)}(q) \\ P_{=n, (1,0)}(q) \\ P_{=n, (1,1)}(q) \\
   P_{=n, (2,0)}(q) \\ P_{=n, (2,1)}(q) \\ P_{=n, (2,2)}(q)
  \end{bmatrix}
  = \begin{bmatrix}
     q^{3n} & q^{n} & q^{n} & q^{2n} & q^{3n} & q^{4n} \\
     q^{2n} & q^{3n} & q^{n} & q^{n} & q^{2n} & q^{3n} \\
     q^{n} & q^{2n} & q^{3n} & q^{3n} & q^{n} & q^{2n} \\
     q^{4n} & q^{2n} & q^{3n} & q^{3n} & q^{n} & q^{2n} \\
     q^{3n} & q^{n} & q^{2n} & q^{2n} & q^{3n} & q^{n} \\
     q^{2n} & q^{3n} & q^{4n} & q^{4n} & q^{2n} & q^{3n} \\
    \end{bmatrix}
  \begin{bmatrix}
   P_{(n-1), (0,0)}(q) \\ P_{(n-1), (1,0)}(q) \\ P_{(n-1), (1,1)}(q) \\
   P_{(n-1), (2,0)}(q) \\ P_{(n-1), (2,1)}(q) \\ P_{(n-1), (2,2)}(q)
  \end{bmatrix}.
\end{align}
The only difference of the last displayed square matrix from that in \eqref{eqCDUConjPrepMxEx}
is that there are $q^{3n}$'s on the diagonal instead of $1$'s.
We also remind the reader of the three-way identity \eqref{eqClyPtnGenFuncWithPs}, 
which follows from the above discussion.  

In addition, Theorem \ref{thmCDUconjPrep} can then be restated as the following.
\begin{align}
\label{eqFuncEqAltToCorteelWelsh}
  F_c(z, q)
  = \frac{(1 - z)}{(1 - zq)} F_c(zq, q)
  + z \sum_d \frac{ (1 - z)q^{\Delta(d,c)} }{ (1 - zq^{\Delta(d,c)}) }
    F_d(zq^{\Delta(d,c)}, q),
\end{align}
where the sum is over all possible shapes for fixed rank $r$ and level $\ell$.  
This is not exactly the same as the functional equations 
found in~\cite{CW}, as it has a uniform number of generating functions in each instance.  

The above construction shows greater resemblance to
Tingley's bijection~\cite{Tingley, Tingley-Correction}.
But, Tingley strips off \emph{ribbons}, one at a time,
while the askew columns pertained to above are never ribbons.  
Also, the tight configuration Tingley arrives at appears to be tighter 
than our terminal configurations.  

Unlike tight packing of slices,
tight packing of pivot slices is not always uniquely possible.
We know that a slice $\Pi$ with shape $\sigma$, where $\sigma_1 \geq 2$,
is a potential pivot.
When there is enough space to the left and to the right of $\Pi$,
the last box $s$ used for tiling the space to the left
is strictly to the right of the first box $b$ used for tiling the space
to the right of $\Pi$.
\begin{center}
 \includegraphics[scale=0.05]{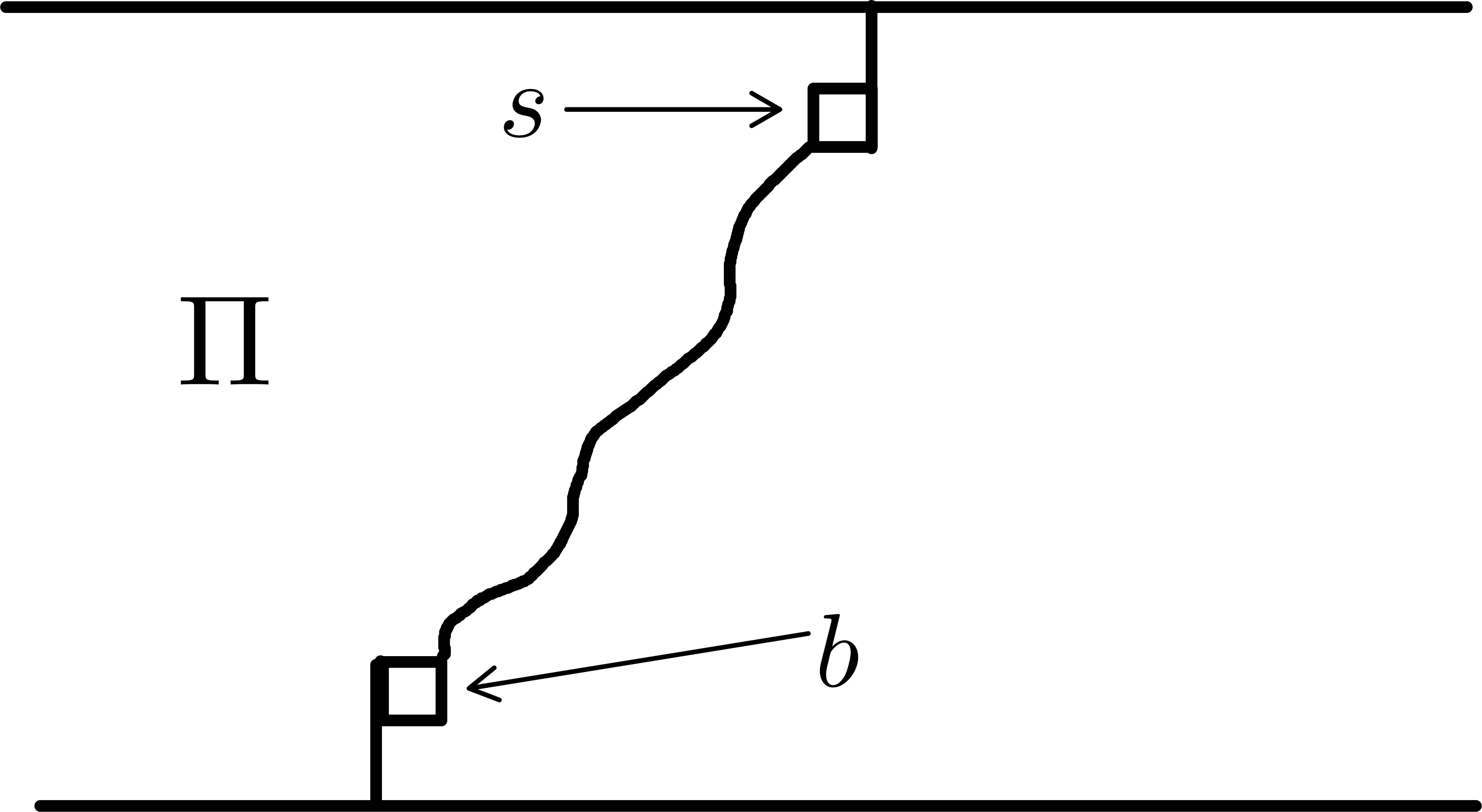}
\end{center}
However, when there are other potential pivots
to the right or to the left of $\Pi$, or both,
say ${}_r\Pi$ and ${}_l\Pi$,
with shapes ${}_r\sigma$ and ${}_l\sigma$, respectively,
the boxes $s$ and $b$ shown above may not be in the spaces
between ${}_l\Pi$ and $\Pi$, or $\Pi$ and ${}_r\Pi$ anymore.
\begin{center}
 \includegraphics[scale=0.05]{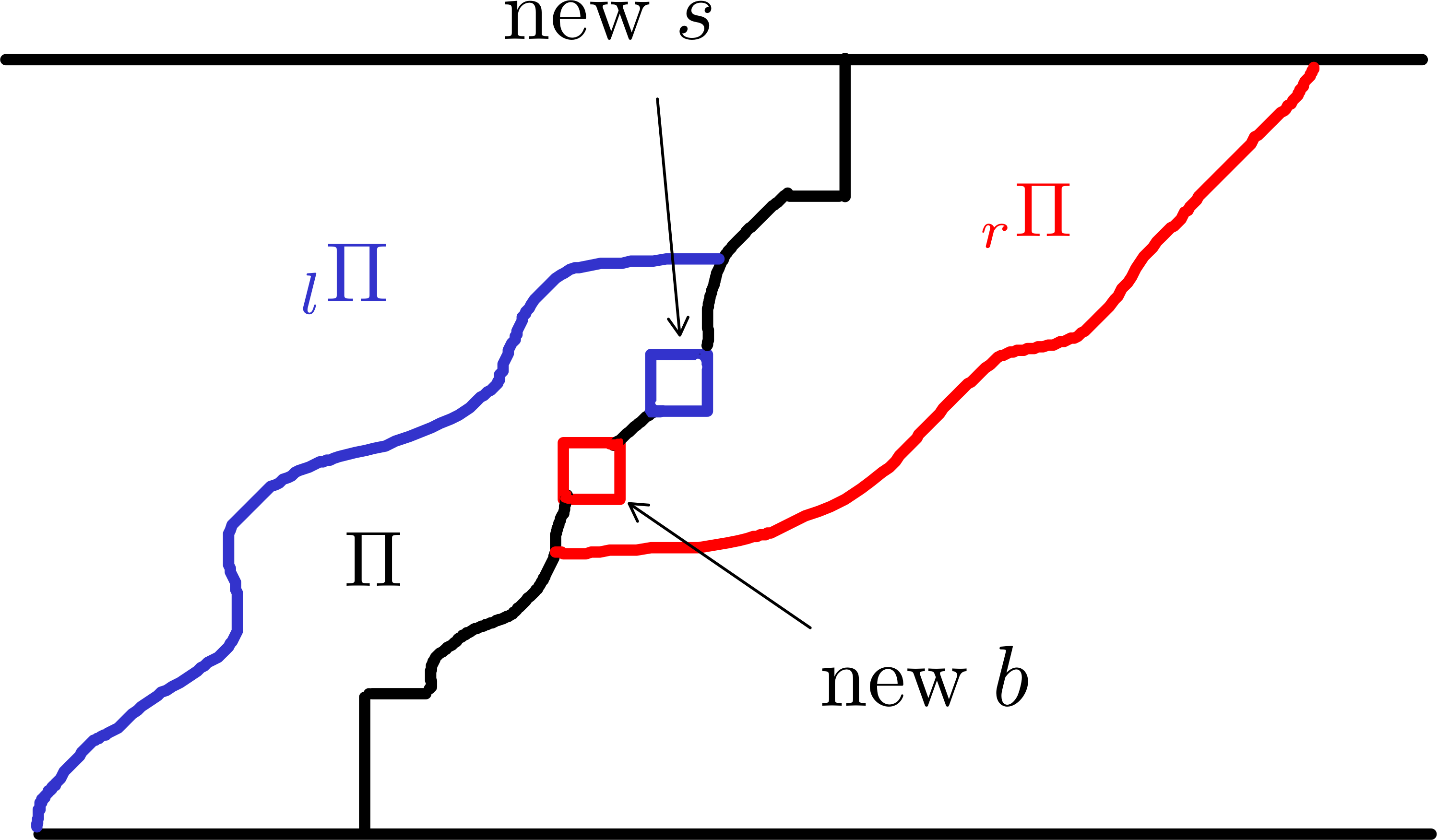}
\end{center}
If the new $s$ not strictly to the right of the new $b$,
then $\Pi$ may not be a pivot in this lineup anymore.
For similar reasons, ${}_l\Pi$ or ${}_r\Pi$ may not be pivots, either.
Please note that, in this case,
\begin{align}
\nonumber
  \left\vert \Pi \right\vert - \left\vert {}_l\Pi \right\vert = \Delta( {}_l\sigma, \sigma ),
  \quad \textrm{ and } \quad
  \left\vert {}_r\Pi \right\vert - \left\vert \Pi \right\vert = \Delta( \sigma, {}_r\sigma ),
\end{align}
necessarily.

But, we can expand $\Pi$ and ${}_r\Pi$ by one box in each row,
so that the old $s$ is restored.
We can instead expand only ${}_r\Pi$ by one box in each row,
so that the old $b$ is restored.
We may even need to do both of these operations.
\begin{center}
 \includegraphics[scale=0.05]{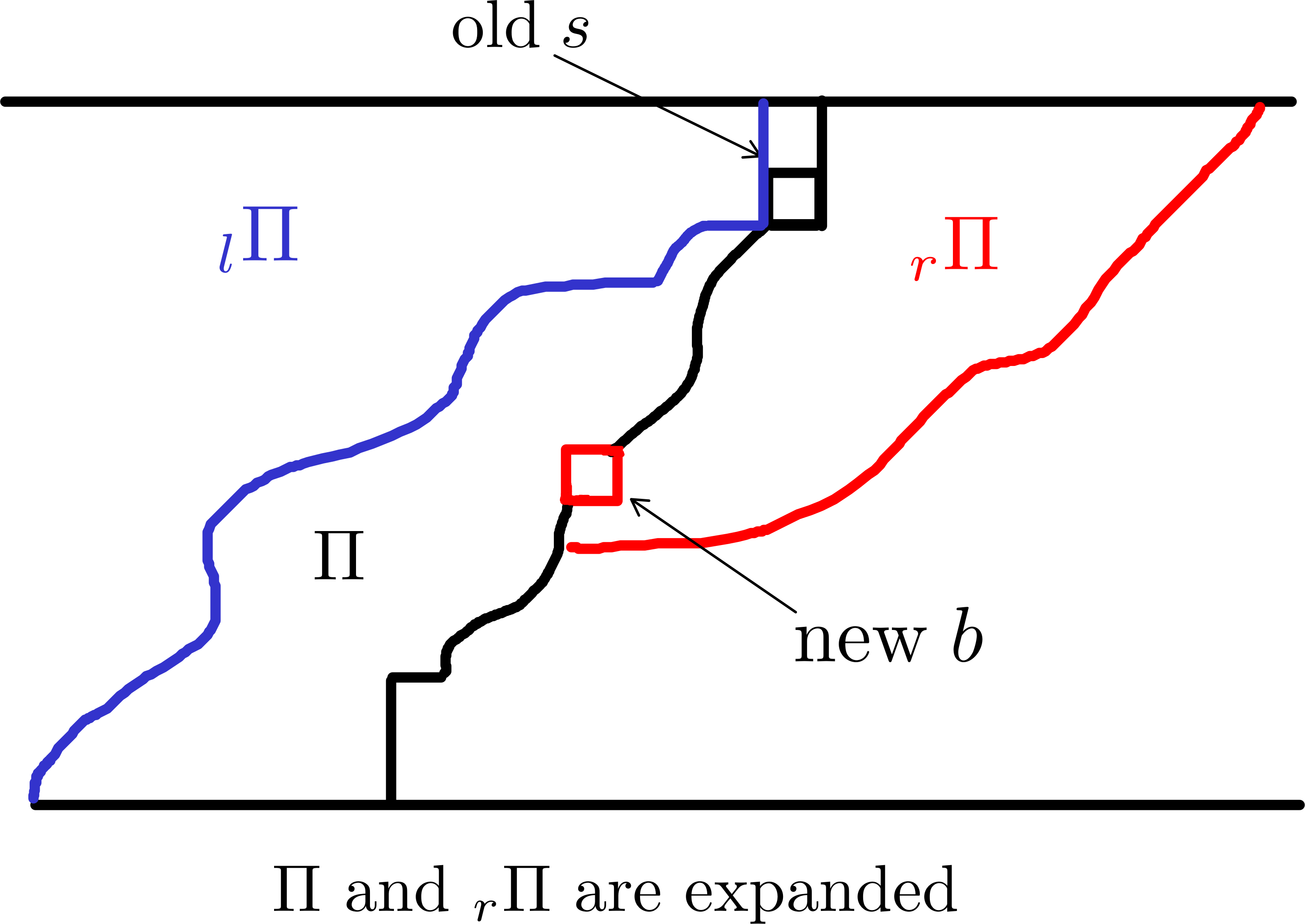}
 \hspace{15mm}
 \includegraphics[scale=0.05]{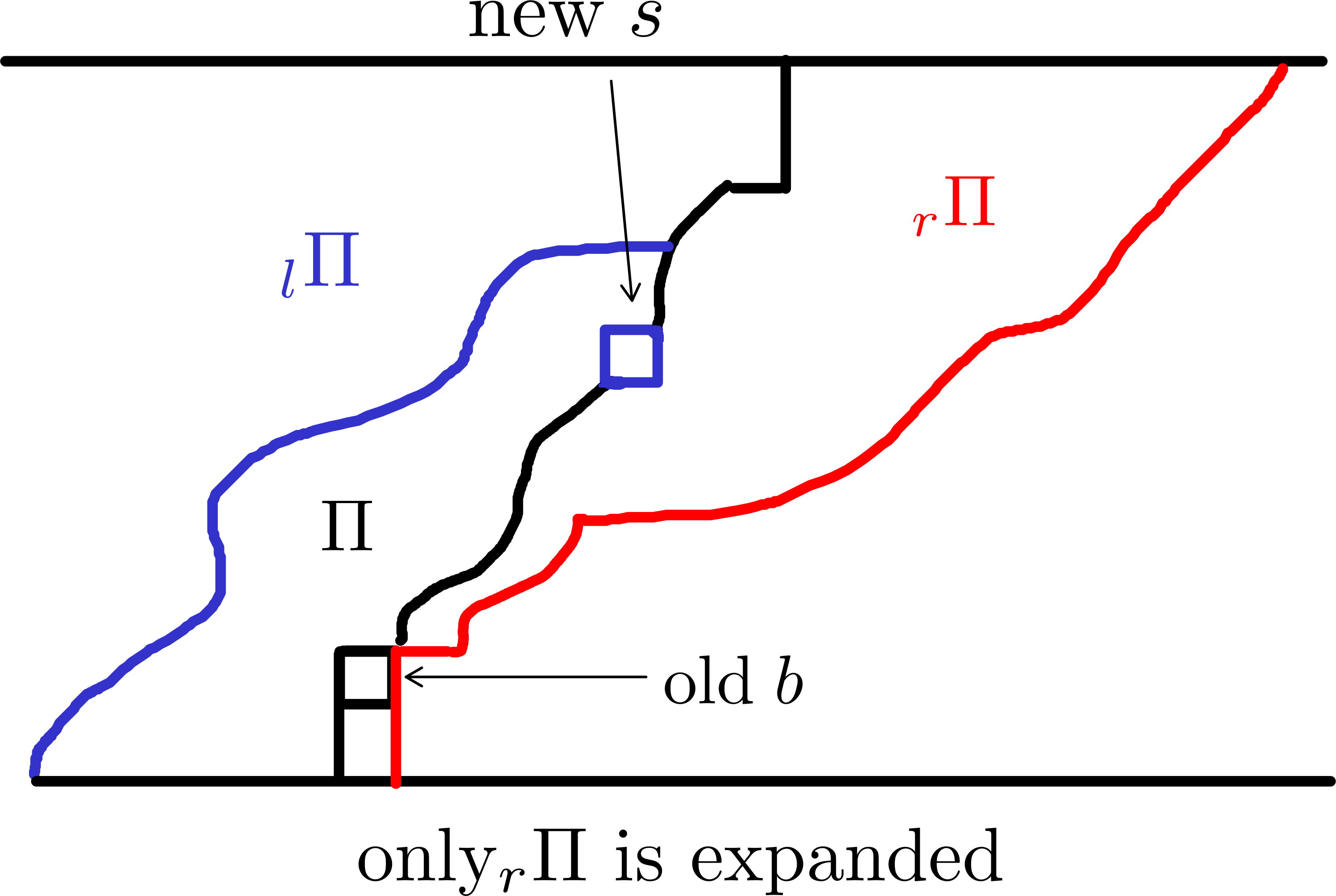}
\end{center}
For example, in profile $(1,1,0)$, all of the choices $(2, 0)$, $(2, 1)$, $(2, 2)$
as shape $\sigma$ are potential pivots.
But the tight packing as in Theorem \ref{thmCDUconjPrep}
will result in $1^{(2, 0)}$, $2^{(2, 1)}$, $3^{(2, 2)}$,
where the middle $2^{(2, 1)}$ is not a pivot.
The boxes for $2^{(2, 1)}$ are shown, as well.
\begin{center}
 \includegraphics[scale=0.6]{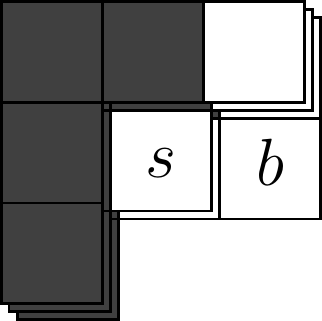}
\end{center}
% \begin{align*}
% 		\begin{ytableau}
% 			*(darkgray) & *(darkgray) & \\
% 			*(darkgray) \\
% 			*(darkgray) \\
% 		\end{ytableau} \qquad
% 		\begin{ytableau}
% 			*(darkgray) & *(darkgray) & \\
% 			*(darkgray) & s \\
% 			*(darkgray) \\
% 		\end{ytableau} \qquad
% 		\begin{ytableau}
% 			*(darkgray) & *(darkgray) & \\
% 			*(darkgray) & & b \\
% 			*(darkgray) \\
% 		\end{ytableau}
% \end{align*}

If we expand $2^{(2, 1)}$ and $3^{(2, 2)}$, we get
$1^{(2, 0)}$, $5^{(2, 1)}$, $6^{(2, 2)}$,
where the $5^{(2, 1)}$ is still not a pivot.
\begin{center}
 \includegraphics[scale=0.6]{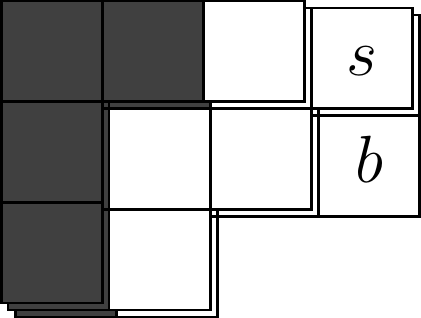}
\end{center}
% \begin{align*}
% 		\begin{ytableau}
% 			*(darkgray) & *(darkgray) & \\
% 			*(darkgray) \\
% 			*(darkgray) \\
% 		\end{ytableau} \qquad
% 		\begin{ytableau}
% 			*(darkgray) & *(darkgray) & & s \\
% 			*(darkgray) & & \\
% 			*(darkgray) & \\
% 		\end{ytableau} \qquad
% 		\begin{ytableau}
% 			*(darkgray) & *(darkgray) & & \\
% 			*(darkgray) & & & b \\
% 			*(darkgray) & \\
% 		\end{ytableau}
% \end{align*}

If we expanded $3^{(2, 2)}$ only,
we'd have $1^{(2, 0)}$, $2^{(2, 1)}$, $6^{(2, 2)}$,
where the $2^{(2, 1)}$ is still not a pivot.
\begin{center}
 \includegraphics[scale=0.6]{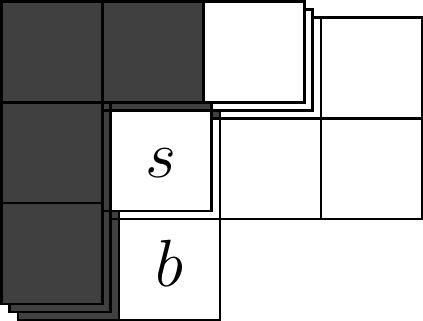}
\end{center}
% \begin{align*}
% 		\begin{ytableau}
% 			*(darkgray) & *(darkgray) & \\
% 			*(darkgray) \\
% 			*(darkgray) \\
% 		\end{ytableau} \qquad
% 		\begin{ytableau}
% 			*(darkgray) & *(darkgray) & \\
% 			*(darkgray) & s \\
% 			*(darkgray) \\
% 		\end{ytableau} \qquad
% 		\begin{ytableau}
% 			*(darkgray) & *(darkgray) & & \\
% 			*(darkgray) & & & \\
% 			*(darkgray) & b \\
% 		\end{ytableau}
% \end{align*}

Instead, if we expand $2^{(2, 1)}$ and $3^{(2, 2)}$,
and then $6^{(2, 2)}$ again, we would end up with
$1^{(2, 0)}$, $5^{(2, 1)}$, $9^{(2, 2)}$,
where $5^{(2, 1)}$ is a pivot.
In fact, all slices in this lineup are pivots.
\begin{center}
 \includegraphics[scale=0.6]{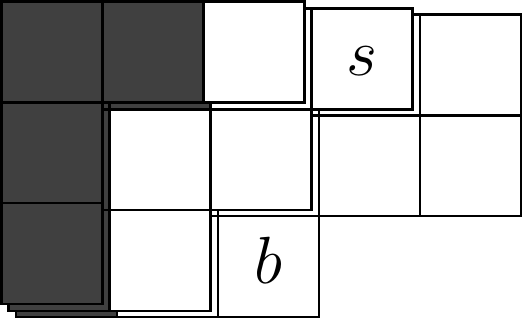}
\end{center}
% \begin{align*}
% 		\begin{ytableau}
% 			*(darkgray) & *(darkgray) & \\
% 			*(darkgray) \\
% 			*(darkgray) \\
% 		\end{ytableau} \qquad
% 		\begin{ytableau}
% 			*(darkgray) & *(darkgray) & & s \\
% 			*(darkgray) & & \\
% 			*(darkgray) & \\
% 		\end{ytableau} \qquad
% 		\begin{ytableau}
% 			*(darkgray) & *(darkgray) & & & \\
% 			*(darkgray) & & & & \\
% 			*(darkgray) & & b \\
% 		\end{ytableau}
% \end{align*}

We welcome the reader to verify that the slice $1^{(2, 1)}$ by itself
is not a pivot in the profile $(0, 2, 0)$,
but its once expanded version $4^{(2, 1)}$ is.

The next example is in profile $(4, 0, 0)$.
The potential pivots from smallest to largest
will have shapes $\sigma = $ $(4, 1)$, $(4, 2)$, and $(4, 3)$ this time.
Per
\begin{align}
\nonumber
  \Delta( (0, 0), (4, 1) ) = 5,
  \quad
  \Delta( (4, 1), (4, 2) ) = 1,
  \quad \textrm{ and } \quad
  \Delta( (4, 2), (4, 3) ) = 1,
\end{align}
the tight configuration in the sense of Theorem \ref{thmCDUconjPrep} will be
$5^{(4, 1)}$, $6^{(4, 2)}$, $7^{(4, 3)}$,
Here, $6^{(4, 2)}$ is not a pivot, but the others are.
The boxes $s$ and $b$ are shown for $6^{(4, 2)}$.
\begin{center}
 \includegraphics[scale=0.6]{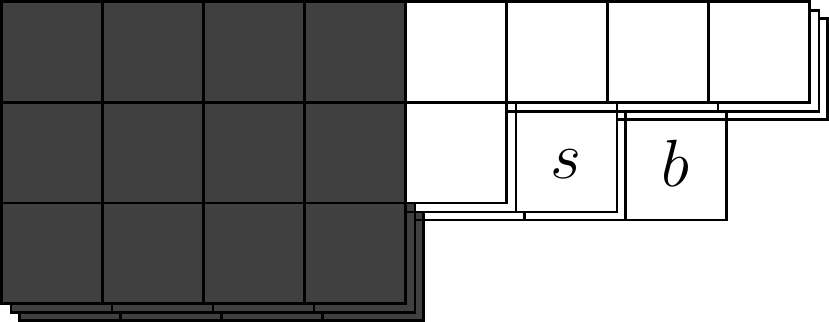}
\end{center}
% \begin{align*}
% 		\begin{ytableau}
% 			*(darkgray) & *(darkgray) & *(darkgray) & *(darkgray) & & & & \\
% 			*(darkgray) & *(darkgray) & *(darkgray) & *(darkgray) & \\
% 			*(darkgray) & *(darkgray) & *(darkgray) & *(darkgray)
% 		\end{ytableau} \qquad
% 		\begin{ytableau}
% 			*(darkgray) & *(darkgray) & *(darkgray) & *(darkgray) & & & & \\
% 			*(darkgray) & *(darkgray) & *(darkgray) & *(darkgray) & & s \\
% 			*(darkgray) & *(darkgray) & *(darkgray) & *(darkgray)
% 		\end{ytableau} \qquad
% 		\begin{ytableau}
% 			*(darkgray) & *(darkgray) & *(darkgray) & *(darkgray) & & & & \\
% 			*(darkgray) & *(darkgray) & *(darkgray) & *(darkgray) & & & b \\
% 			*(darkgray) & *(darkgray) & *(darkgray) & *(darkgray)
% 		\end{ytableau}
% \end{align*}

In either $5^{(4, 1)}$, $9^{(4, 2)}$, $10^{(4, 3)}$
or $5^{(4, 1)}$, $6^{(4, 2)}$, $10^{(4, 3)}$,
all slices are pivots.
\begin{center}
 \includegraphics[scale=0.6]{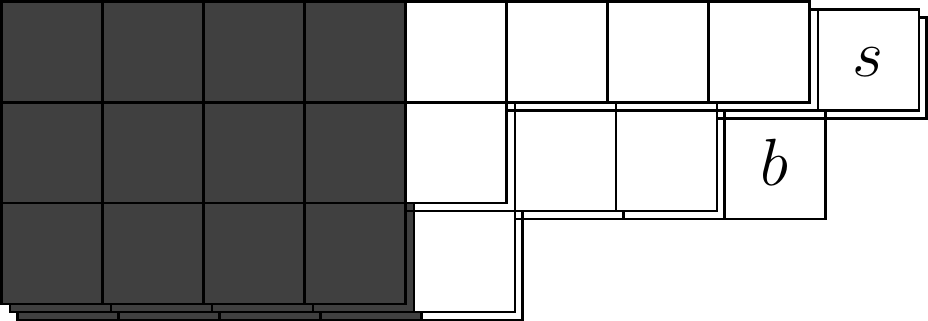}
 \raisebox{6mm}{, or}
 \includegraphics[scale=0.6]{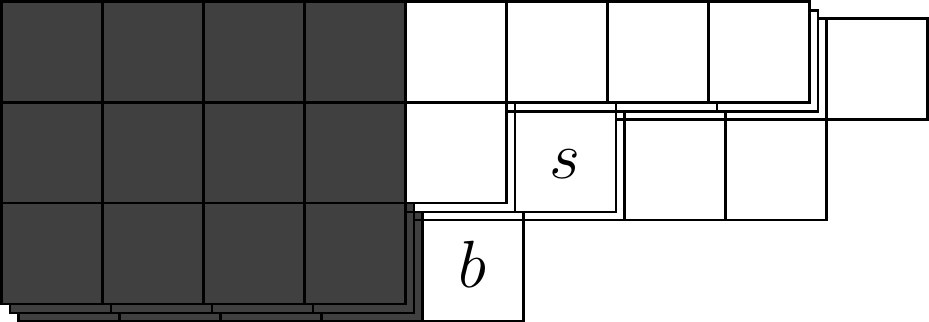}
\end{center}
% \begin{align*}
% 		\begin{ytableau}
% 			*(darkgray) & *(darkgray) & *(darkgray) & *(darkgray) & & & & \\
% 			*(darkgray) & *(darkgray) & *(darkgray) & *(darkgray) & \\
% 			*(darkgray) & *(darkgray) & *(darkgray) & *(darkgray)
% 		\end{ytableau} \qquad
% 		\begin{ytableau}
% 			*(darkgray) & *(darkgray) & *(darkgray) & *(darkgray) & & & & & s \\
% 			*(darkgray) & *(darkgray) & *(darkgray) & *(darkgray) & & & \\
% 			*(darkgray) & *(darkgray) & *(darkgray) & *(darkgray) &
% 		\end{ytableau} \qquad
% 		\begin{ytableau}
% 			*(darkgray) & *(darkgray) & *(darkgray) & *(darkgray) & & & & & \\
% 			*(darkgray) & *(darkgray) & *(darkgray) & *(darkgray) & & & & b \\
% 			*(darkgray) & *(darkgray) & *(darkgray) & *(darkgray) &
% 		\end{ytableau}
% \end{align*}
% \begin{align*}
% 		\begin{ytableau}
% 			*(darkgray) & *(darkgray) & *(darkgray) & *(darkgray) & & & & \\
% 			*(darkgray) & *(darkgray) & *(darkgray) & *(darkgray) & \\
% 			*(darkgray) & *(darkgray) & *(darkgray) & *(darkgray)
% 		\end{ytableau} \qquad
% 		\begin{ytableau}
% 			*(darkgray) & *(darkgray) & *(darkgray) & *(darkgray) & & & & \\
% 			*(darkgray) & *(darkgray) & *(darkgray) & *(darkgray) & & s \\
% 			*(darkgray) & *(darkgray) & *(darkgray) & *(darkgray)
% 		\end{ytableau} \qquad
% 		\begin{ytableau}
% 			*(darkgray) & *(darkgray) & *(darkgray) & *(darkgray) & & & & & \\
% 			*(darkgray) & *(darkgray) & *(darkgray) & *(darkgray) & & & & \\
% 			*(darkgray) & *(darkgray) & *(darkgray) & *(darkgray) & b
% 		\end{ytableau}
% \end{align*}

Let's summarize the observations and examples above as follows.
We consider
\begin{align}
\label{ineq_QConjPrep_lineup0}
  {}_{1}\Pi > {}_{2}\Pi > \cdots > {}_{n}\Pi > {}_{n+1}\Pi = E,
\end{align}
where the slices have shapes ${}_{1}\sigma$, ${}_{2}\sigma$, \ldots, ${}_{n}\sigma$, ${}_{n+1}\sigma$,
in their respective order.
${}_{j}\sigma_1 \geq 2$ for $j = $ 1, 2, \ldots, $n$, but not necessarily $j = n+1$.
In other words, ${}_j\Pi$ for $j = $ 1, 2, \ldots, $n$ are potential pivots.
If the inequality
\begin{align}
\label{ineq_QConjPrep_star}
  \left\vert {}_j\Pi \right\vert - \left\vert {}_{j+1}\Pi \right\vert
  \geq \Delta( {}_{j+1}\sigma, {}_{j}\sigma ) + r
\end{align}
is satisfied for all $j = $ 1, 2, \ldots, $n$,
then ${}_{j}\Pi$'s for $j = $ 1, 2, \ldots, $n$ are pivots.
If
\begin{align}
\label{eq_QConjPrep_doubleStar}
  \left\vert {}_j\Pi \right\vert - \left\vert {}_{j+1}\Pi \right\vert
  = \Delta( {}_{j+1}\sigma, {}_{j}\sigma )
\end{align}
for some $j$, $1 \leq$ $j$ $\leq n$,
then ${}_{j}\Pi$, and if $j < n$ ${}_{j+1}\Pi$, too,
must be separately checked to see if they are still pivots.

If \eqref{ineq_QConjPrep_star} is satisfied for all $j = $ 1, 2, \ldots, $n$,
then the lineup \eqref{ineq_QConjPrep_lineup0} is called a
\emph{loose lineup}.
In addition, if the equality is attained in \eqref{ineq_QConjPrep_star}
for all $j = $ 1, 2, \ldots, $n$,
then \eqref{ineq_QConjPrep_lineup0} is called a \emph{minimal loose lineup}.
In fact, if we fix the shapes ${}_{1}\sigma$, ${}_{2}\sigma$, \ldots,
${}_{n}\sigma$, then the minimal loose lineup is unique.
If \eqref{eq_QConjPrep_doubleStar} holds for some $j$, $1 \leq $ $j$, $\leq n$,
and ${}_{j}\Pi$'s are still pivot slices for all $j = $ 1, 2, \ldots, $n$,
then \eqref{ineq_QConjPrep_lineup0} is called a \emph{jammed lineup}.
If \eqref{eq_QConjPrep_doubleStar} holds for some $j$, $1 \leq $ $j$, $\leq n$,
and the equality is attained in \eqref{ineq_QConjPrep_star} for the remaining $j$'s,
and ${}_{j}\Pi$'s are still pivot slices for all $j = $ 1, 2, \ldots, $n$,
then \eqref{ineq_QConjPrep_lineup0} is called a \emph{minimal jammed lineup}.
Even if one fixes the shapes, minimal jammed lineups are not unique.
$1^{(2, 0)}$, $5^{(2, 1)}$, $9^{(2, 2)}$ for profile $(1, 1, 0)$,
and $5^{(4, 1)}$, $9^{(4, 2)}$, $10^{(4, 3)}$, and
$5^{(4, 1)}$, $6^{(4, 2)}$, $10^{(4, 3)}$ for profile $(4, 0, 0)$
are all jammed lineups.
In fact, all are minimal jammed lineups.
We invite the reader to verify that $4^{(2, 0)}$, $8^{(2, 1)}$, $12^{(2, 2)}$ in profile $(1, 1, 0)$
is the minimal loose lineup for potential pivot shapes $(2, 0)$, $(2, 1)$, $(2, 2)$
from the smallest to largest slices,
and $4^{(2, 0)}$, $11^{(2, 1)}$, $15^{(2, 2)}$ in profile $(1, 1, 0)$
is a loose lineup which is not minimal.

If \eqref{ineq_QConjPrep_lineup0} is a jammed lineup,
let
\begin{align}
\nonumber
  \Lambda = {}_{1}\Pi + {}_{2}\Pi + \cdots + {}_{n}\Pi,
\end{align}
and set
\begin{align}
\nonumber
  \iota(\Lambda) = \textrm{the collection of }
    j\textrm{'s for which \eqref{eq_QConjPrep_doubleStar} is satisfied.  }
\end{align}
For instance,
\begin{align}
\nonumber
  & \iota( 5^{(4, 1)} + 9^{(4, 2)} + 10^{(4, 3)} ) = \{ 1, 3 \} \\
\nonumber
  & \iota( 5^{(4, 1)} + 6^{(4, 2)} + 10^{(4, 3)} ) = \{ 2, 3 \}
  \textrm{ (profile } (4, 0, 0) \textrm{ )}, \\
\nonumber
  & \iota( 1^{(2, 0)} + 5^{(2, 1)} + 9^{(2, 2)} ) = \{ 3 \}
  \textrm{ (profile } (1, 1, 0) \textrm{ )}.
\end{align}
\begin{align}
\nonumber
  \iota( 4^{(2, 0)} + 8^{(2, 1)} + 12^{(2, 2)} ) = \varnothing
\end{align}
because $4^{(2, 0)}$, $8^{(2, 1)}$, $12^{(2, 2)}$ is not a jammed lineup
in profile $(1, 1, 0)$.

For fixed profile $c$ and number of pivot slices $n$,
the collection of loose lineups will be denoted by $\mathcal{LL}_{n, c}$.
In the same setting, we let $\mathcal{MLL}_{n, c}$, $\mathcal{JL}_{n, c}$, $\mathcal{MJL}_{n, c}$
denote the set of minimal loose lineups, jammed lineups, and minimal jammed lineups, respectively.
We will later show that
\begin{align}
\nonumber
  \left\vert \mathcal{MLL}_{n, c} \right\vert < \infty,
  \textrm{ and }
  \left\vert \mathcal{MJL}_{n, c} \right\vert < \infty
\end{align}
for any $c$ and $n$.

\begin{lemma}
\label{lemmaQConjPrep}
  For integer $n \geq 1$ and fixed profile $c = (c_1, \ldots, c_r)$,
  \begin{align}
  \nonumber
    \sum_{ {}_{1}\Pi > {}_{2}\Pi > \cdots > {}_{n}\Pi > E}
      q^{ \vert {}_{1}\Pi \vert + \vert {}_{2}\Pi \vert + \cdots + \vert {}_{n}\Pi \vert }
    & = \sum_{ \Lambda \in \mathcal{MLL}_{n, c} }
        \frac{ q^{ \vert \Lambda \vert } }{ (q^r; q^r)_n }
    + \sum_{ \Lambda \in \mathcal{MJL}_{n, c} }
        \frac{ q^{ \vert \Lambda \vert } }{
          \displaystyle \prod_{ \substack{j = 1 \\ j \not\in \iota(\Lambda)} }^n (1 - q^{jr} )  } \\
  \nonumber
    & = \frac{ \displaystyle \sum_{ \Lambda \in \mathcal{MLL}_{n, c} } q^{ \vert \Lambda \vert }
            + \sum_{ \Lambda \in \mathcal{MJL}_{n, c} }
              \left( \prod_{j \in \iota(\Lambda)} (1 - q^{jr}) \right)
                q^{ \vert \Lambda \vert }
        }{ (q^r; q^r)_n },
  \end{align}
  where the first sum is over all chains of $n$ pivots.
\end{lemma}
We must emphasize that the first sum in Lemma \ref{lemmaQConjPrep}
is over actual, not potential, pivots.

\begin{proof}
 We will imitate the proof of Theorem \ref{thmCDUconjPrep}.
 For \eqref{ineq_QConjPrep_lineup0} a loose lineup,
 we can shrink ${}_{j}\Pi$'s so that
 \begin{align}
 \nonumber
  \left\vert {}_{j}\Pi \right\vert - \left\vert {}_{j+1}\Pi \right\vert
  = \Delta( {}_{j+1}\sigma, {}_{j}\sigma ) + r,
 \end{align}
 and collect the number of removed boxes
 as a partition into multiples of $r$, all at most $r \cdot n$.
 The process is clearly reversible.
 This gives us a first sums in the second and the third slots.
 For a jammed lineup, we want to keep
 \begin{align}
 \nonumber
  \left\vert {}_{j}\Pi \right\vert - \left\vert {}_{j+1}\Pi \right\vert
  = \Delta( {}_{j+1}\sigma, {}_{j}\sigma )
 \end{align}
 at all times for the indices $j$ satisfying \eqref{eq_QConjPrep_doubleStar}.
 For any such index $j$,
 ${}_{1}\Pi$, ${}_{2}\Pi$, \ldots, ${}_{j}\Pi$ must not be expanded
 while leaving ${}_{j+1}\Pi$ as is.
 In other words, the auxiliary partition accounted for
 by $\frac{1}{(q^r; q^r)_n}$ must not have any $j \cdot r$'s in it.
 Likewise, the process is reversible,
 i.e. given a minimal jammed lineup and a partition into
 multiples of $r$'s, all at most $r \cdot n$,
 having no parts $j \cdot r$ that would ``loose the jammed pivots'',
 we can obtain a jammed lineup by expanding the pivot slices accordingly.
 This explains the second sums in the second and the third slots.
\end{proof}

For profile $c = (c_1, \ldots, c_r)$ with rank $r$ and level $\ell$,
there are $\binom{\ell + r - 1}{r - 1}$ possible shapes.
Among these, $r$ of them,
in particular $(0, \ldots, 0)$, $(1, 0, \ldots, 0)$, $(1, 1, 0, \ldots, 0)$,
\ldots, $(1, 1, \ldots, 1)$ cannot be pivots,
but the remaining $\binom{\ell + r - 1}{r - 1} - r$ are potential pivots.
Thus, there are $\left( \binom{\ell + r - 1}{r - 1} - r \right)^n$
possibilities for \eqref{ineq_QConjPrep_lineup0} to be a minimal loose lineup.
If we try to shrink some or all of the $n$ gaps between the pivots
in a minimal loose lineup,
we may or may not end up with a jammed lineup.
This is because the pivothood of each slice is not guaranteed to be preserved.
Thus, we can only talk about an upper bound for the minimal jammed lineups for each $n$,
which is $(2^n - 1)\left( \binom{\ell + r - 1}{r - 1} - r \right)^n$.
We have shown that
\begin{align}
\nonumber
  \left\vert \mathcal{MLL}_{n, c} \right\vert = \left( \binom{\ell + r - 1}{r - 1} - r \right)^n,
  \textrm{ and }
  \left\vert \mathcal{MJL}_{n, c} \right\vert \leq (2^n - 1)
    \left( \binom{\ell + r - 1}{r - 1} - r \right)^n,
\end{align}
for any $c$ and $n$.

\begin{theorem}
\label{thmQconjPrep}
  For any profile $c = (c_1, \ldots, c_r)$, the identity \eqref{eqQconjPrepGenFunc}
  holds with $\widetilde{P}_{0, c}(q) = 1$, and
  \begin{align}
  \nonumber
    \widetilde{P}_{n, c}(q) = \sum_d q^{ n \Delta(c, d) + nr } \widetilde{P}_{n-1, d}(q)
  \end{align}
  for $n \geq 1$, where the sum is over all potential pivot shapes $d$.
\end{theorem}

\begin{proof}
 Combine Theorem \ref{thmCDUconjPrep} and Lemma \ref{lemmaQConjPrep}.
 Any cylindric partition with profile $c$ breaks up
 into an ordinary partition and a chain of pivot slices.
 The ordinary partition is generated by $\frac{1}{ (zq; q)_\infty }$,
 and the chain of pivots is found with the help of Lemma \ref{lemmaQConjPrep}.
\end{proof}

The only missing detail in the the proofs above is the placement of $z$
in the factor $\frac{1}{(zq; q)_\infty}$,
contributing to the maximum part in the generated cylindric partitions.
But this is implicit in the proof of Theorem \ref{thmCylPtnVsPtnPairsFullCase}.
The slices whose weights are stored in $\mu$ are either duplicates of pivot slices,
or the non-pivot slices in between.
In short, they are additional slices.
As such, they contribute one each to the maximum part.

Again, if one fixes the rank and the level,
it is convenient to rewrite the recurrence in
Theorem \ref{thmQconjPrep} in terms of matrices and vectors.
For instance, for rank $r = 3$ and level $\ell = 2$, one has
\begin{align}
\nonumber
  \begin{bmatrix}
   \widetilde{P}_{n, (2,0)} \\ \widetilde{P}_{n, (2,1)} \\ \widetilde{P}_{n, (2,2)}
  \end{bmatrix}
  = \begin{bmatrix}
   q^{3n} & q^{4n} & q^{5n} \\
   q^{5n} & q^{3n} & q^{4n} \\
   q^{7n} & q^{5n} & q^{3n}
  \end{bmatrix}
  \begin{bmatrix}
   \widetilde{P}_{n-1, (2,0)} \\ \widetilde{P}_{n-1, (2,1)} \\ \widetilde{P}_{n-1, (2,2)}
  \end{bmatrix},
  \textrm{ and }
  \begin{bmatrix}
   \widetilde{P}_{n, (0,0)} \\ \widetilde{P}_{n, (1,0)} \\ \widetilde{P}_{n, (1,1)}
  \end{bmatrix}
  = \begin{bmatrix}
   q^{5n} & q^{6n} & q^{7n} \\
   q^{4n} & q^{5n} & q^{6n} \\
   q^{6n} & q^{4n} & q^{5n}
  \end{bmatrix}
  \begin{bmatrix}
   \widetilde{P}_{n-1, (2,0)} \\ \widetilde{P}_{n-1, (2,1)} \\ \widetilde{P}_{n-1, (2,2)}
  \end{bmatrix}.
\end{align}

We also have the following corollary.
\begin{cor}
\label{coroQconj}
  Let $r$ be the rank, and $\ell$ be the level
  of the cylindric partitions that are being considered.  
  $\widetilde{P}_{n, c}(q)$ is a polynomial with positive coefficients, and 
  \begin{align}
  \nonumber 
    \widetilde{P}_{n, c}(1) = \left( \binom{\ell+r-1}{r-1} - r \right)^n.
  \end{align}
\end{cor}

Relabeling as
\begin{align}
\nonumber
  \widetilde{Q}_{n, c}(q) = \widetilde{P}_{n, c}(q)
    + \displaystyle \sum_{ \Lambda \in \mathcal{MJL}_{n, c} }
      \left( \prod_{j \in \iota(\Lambda)} (1 - q^{rj}) \right) q^{\vert \Lambda \vert},
\end{align}
we still have $\widetilde{Q}_{n, c}(q)$'s as polynomials, and
\begin{align}
\nonumber
  \widetilde{Q}_{n, c}(1) = \left( \binom{\ell+r-1}{r-1} - r \right)^n.
\end{align}
Although we cannot calculate the $\widetilde{Q}_{n, c}(q)$'s directly,
we can use the matrix formulation of Theorems \ref{thmCDUconjPrep} and \ref{thmCDUconjPrepAlt}
together with a corollary of the $q$-binomial theorem~\cite{GR} as follows,
for arbitrary but fixed profile $c$.
\begin{align}
\nonumber
  \sum_{n \geq 0} \frac{ \widetilde{Q}_{n, c}(q) z^n }{ (q^r; q^r)_n }
  = (zq; q)_\infty \sum_{k \geq 0} \frac{ P_{=k, c}(q) z^k }{ (q^r; q^r)_k }
  = \left\{ \sum_{m \geq 0} \frac{ (-1)^m q^{ \binom{m+1}{2} z^m } }{ (q; q)_m } \right\}
  \left\{ \sum_{k \geq 0} \frac{ P_{=k, c}(q) z^k }{ (q^r; q^r)_k } \right\}
\end{align}
\begin{align}
\nonumber
  = \left\{ \sum_{m \geq 0} \frac{ z^m }{ (q^r; q^r)_m }
    \left[ \prod_{j = 1}^m (1 + q^j + q^{2j} + \cdots + q^{(r-1)j}) \right]
      (-1)^m q^{ \binom{m+1}{2} } \right\}
  \left\{ \sum_{k \geq 0} \frac{ z^k }{ (q^r; q^r)_k }  P_{=k, c}(q) \right\}
\end{align}
\begin{align}
\nonumber
  \sum_{n \geq 0} \frac{ z^n }{ (q^r; q^r)_n }
  \sum_{ k + m = n } \begin{bmatrix} n \\ k \end{bmatrix}
    \left( \prod_{j = 1}^m (1 + q^j + q^{2j} + \cdots + q^{(r-1)j}) \right)
    (-1)^m q^{ \binom{m+1}{2} } P_{=k, c}(q)
\end{align}
Therefore,
\begin{align}
\label{eqQconjPrepRec}
  \widetilde{Q}_{n, c}(q)
  = \sum_{ k + m = n } \begin{bmatrix} n \\ k \end{bmatrix}
    \left( \prod_{j = 1}^m (1 + q^j + q^{2j} + \cdots + q^{(r-1)j}) \right)
    (-1)^m q^{ \binom{m+1}{2} } P_{=k, c}(q).
\end{align}
By construction, all of $P_{n, c}(q)$, $P_{=n, c}(q)$, and $\widetilde{P}_{n, c}(q)$'s
have positive coefficients.
This is not clear for $\widetilde{Q}_{n, c}(q)$.
Na\"{\i}ve approaches such as fixing the number and the shapes of pivots,
then calculating the contribution in the polynomial, do not work.
% Still, experimentation with the computer suggests the following.
% \begin{conj}
% \label{conjQConjPrep}
%   For any $n \in \mathbb{N}$ and any profile $c$,
%   the polynomials $\widetilde{Q}_{n, c}(q)$'s have positive coefficients.
% \end{conj}

% This is a weaker conjecture than Warnaar's~\cite[Conjecture 8.4]{Warnaar-A2AG},

% The conjecture by Corteel, Dousse, and Uncu~\cite[Conjecture 4.2, part 2]{CDU}
% and Warnaar~\cite[Conjecture 2.7]{Warnaar-A2AG},
% per the above computations seem to extend to Conjecture \ref{conjQConjExtended}.

% Computer experimentation suggests that the $Q_{n, c}(q)$'s are not polynomials
% when $\mathrm{gcd}(r, \ell) > 1$, even if $r \vert \binom{\ell + r - 1}{r - 1}$.
% They are rational functions with power series expansions having integer coefficients.
% Those coefficients necessarily include negative ones.
% Conjecture \ref{conjQConjExtended} implies Conjecture \ref{conjQConjPrep}
% for $r \geq 3$ and $\mathrm{gcd}(r, \ell) = 1$, but not vice versa.
% In fact, the positivity in Conjecture \ref{conjQConjExtended}
% implies conditions much stronger than positivity for $\widetilde{Q}_{n, c}(q)$'s, since
% \begin{align}
% \nonumber
%   \widetilde{Q}_{n, c}(q)
%   = \left( \prod_{j = 1}^n (1 + q^j + q^{2j} + \cdots + q^{(r-1)j}) \right)
%     Q_{n, c}(q).
% \end{align}

\section{Some comments and future work}
\label{secComments}

Theorem \ref{thmCDUconjPrep} seems to be naturally related to 
the framework of the method of weighted words 
introduced by Alladi and Gordon~\cite{Alladi-Gordon}, 
developed and extensively used 
by Dousse~\cite{Dousse-Schur-gen, Dousse-Siladic, Dousse-Capparelli-Primc}, 
and automated by Ablinger and Uncu~\cite{AbU}.  
In particular, one can easily translate 
the matrix form of Theorem \ref{thmCDUconjPrep} into a \emph{gap matrix}.  
Then, one can look for refinements in Borodin's Theorem \ref{theorem-Borodin}~\cite{Borodin}, 
at least for small ranks and levels.  

Theorem \ref{thmDistPartsGeneral} for rank $r = 2$ 
reduces to Corollary \ref{corDistPartsRank2}.  
It is noted afterwards that although diagonalization is guaranteed regardless, 
it appears that there are no repeated eigenvalues in this case.  
This may have a graph-theoretic explanation 
in conjunction with the corresponding shape transition graphs.  

% For rank $r=2$, $Q_{n, (a, b)}(q)$'s are polynomials with positive coefficients
% regardless of the parity of $a+b$, as Warnaar showed in~\cite{Warnaar-A2AG}.
% For $a+b$ even, $Q_{n, (a, b)}(q)$'s seem to yield nice combinatorial constructions
% using path diagrams.

Another open problem is the search for explicit, preferably closed, formulas
for $P_{n, c}(q)$'s, and thus for $P_{=n, c}(q)$'s and $\widetilde{P}_{n, c}(q)$,
based on the matrix characterization of Theorems \ref{thmCDUconjPrep} and \ref{thmCDUconjPrepAlt}.

Theorem \ref{thmDistPartsGeneral}, in the absence of Lambert series,
i.e. for $k = 0$ case,
suggests the study of asymptotics of $(-\alpha q; q)_\infty$
for fixed and known $\alpha \in \mathbb{C}$.
Possible papers to start with are~\cite{Boyer-Goh, Boyer-Goh-2, BFG, Parry}.

The functional equation \eqref{eqFuncEqAltToCorteelWelsh} 
must be equaivalent to Corteel and Welsh's 
functional equation~\cite[Proposition 3.1]{CW}.  

The interplay between loose lineups and jammed lineups 
is reminiscent of \emph{jagged partitions}~\cite{jagged}.  

\section*{Acknowledgements}

The authors thank Walter Bridges for the suggestions 
of references~\cite{Boyer-Goh, Boyer-Goh-2, BFG, Parry} 
for the remark about the asymptotics in Section \ref{secComments}, 
Ole Warnaar for correcting many inaccuracies in the exposition and 
pointing out~\cite{Tingley} and~\cite{Warnaar-BaileyTree}, 
and Ali Uncu for helpful remarks and corrections.  

% \newpage

\bibliographystyle{amsplain}

\begin{thebibliography}{10}

\bibitem{AbU} Jakob Ablinger and Ali K. Uncu,
qFunctions – A Mathematica package for $q$-series and partition theory applications,
\emph {Journal of Symbolic Computation}, {\bf 107}:145--166, (2021).  

\url{https://doi.org/10.1016/j.jsc.2021.02.003}


\bibitem{Alladi-Gordon} Krishnaswami Alladi and Basil Gordon,
Generalizations of Schur's partition theorem,  
\emph{Manuscripta Math} {\bf 79}:113-–126, (1993).  

\url{https://doi.org/10.1007/bf02568332}


\bibitem{GEA-AG} George E. Andrews, 
An Analytic Generalization of the Rogers-Ramanujan Identities for Odd Moduli, 
\emph{Proceedings of the National Academy of Sciences}, {\bf 71}(10):4082--4085, (1974).  

\url{https://doi.org/10.1073/pnas.71.10.4082}


\bibitem{TheBlueBook} George E. Andrews, 
\emph{The theory of partitions}, 
No. 2. Cambridge university press, 1998. 


\bibitem{GEA-E} George E. Andrews and Kimmo Eriksson, 
\emph{Integer partitions}, Cambridge University Press, (2004).  


\bibitem{ASW} George E. Andrews, Anne Schilling and S. Ole Warnaar, 
An $A_2$ Bailey lemma and Rogers-Ramanujan-type identities, 
\emph{J. Amer. Math. Soc.}, {\bf 12}:677--702, (1999).  

\url{https://doi.org/10.1090/S0894-0347-99-00297-0}


\bibitem{LinAlg} Howard Anton and Chris Rorres, 
\emph{Elementary linear algebra: applications version}, 
John Wiley \& Sons, (2013).


\bibitem{Birkhoff-Lattices} Garrett Birkhoff, 
\emph{Lattice Theory, revised ed.}, 
American Mathematical Society Colloquium Publications. Vol. 25., (1948).


\bibitem{NT-Rama} Bruce C. Berndt, 
\emph{Number theory in the spirit of Ramanujan} (Vol. 34), 
American Mathematical Soc., (2006).  


\bibitem{Boyer-Goh} Robert P. Boyer and  William M. Y. Goh,
Partition polynomials: asymptotics and zeros.Tapas in experimental mathematics, 99–111.
\emph{Contemp. Math.}, {\bf 457},
American Mathematical Society, Providence, RI, (2008).

\url{https://doi.org/10.1090/conm/457/08904}


\bibitem{Boyer-Goh-2} Robert P. Boyer and  William M. Y. Goh,
Polynomials associated with partitions: asymptotics and zeros.
Special functions and orthogonal polynomials, 33--45.
\emph{Contemp. Math.}, {\bf 471}
American Mathematical Society, Providence, RI, (2008).

\url{https://doi.org/10.1090/conm/471/09204}


\bibitem{BFG} Walter Bridges, Johann Franke, and Taylor Garnowski,
Asymptotics for the twisted eta-product and applications to sign changes in partitions.
\emph{Res Math Sci} {\bf 9}, 61 (2022).

\url{https://doi.org/10.1007/s40687-022-00355-x}

\bibitem{BU} Walter Bridges and Ali K. Uncu, 
Weighted cylindric partitions, 
\emph{J Algebr Comb}, {\bf 56}:1309--1337, (2022).  

\url{https://doi.org/10.1007/s10801-022-01156-9}


\bibitem {Borodin}  Alexei Borodin., 
Periodic Schur process and cylindric partitions, 
\emph{Duke Math. J.}, {\bf 140}(3):391--468, 2007. 
	
\url{https://doi.org/10.1215/S0012-7094-07-14031-6}


\bibitem{Corteel-RR-RSK} Sylvie Corteel, 
Rogers-Ramanujan identities and the Robinson-Schensted-Knuth correspondence, 
\emph{Proc. Amer. Math. Soc.}, {\bf 145}:2011--2022, 2017.  

\url{http://dx.doi.org/10.1090/proc/13373}


\bibitem{CDU}  Sylvie Corteel, Jehanne Dousse and Ali K. Uncu
Cylindric partitions and some new Rogers–Ramanujan identities, 
\emph{Proc. Amer. Math. Soc.}, {\bf 150}:481--497, (2022).  

\url{https://doi.org/10.1090/proc/15570}


\bibitem{CSV} Sylvie Corteel, Cyrille Savelief, Mirjana Vuletić,
Plane overpartitions and cylindric partitions,
\emph{Journal of Combinatorial Theory, Series A}, {\bf 118}(4):1239--1269, 2011.  

\url{https://doi.org/10.1016/j.jcta.2010.12.001}


\bibitem{CW} Sylvie Corteel and Trevor Welsh, 
The $A_2$ Rogers–Ramanujan Identities Revisited, 
\emph{Ann. Comb.}, {\bf 23}:683--694, (2019).  

\url{https://doi.org/10.1007/s00026-019-00446-7}


\bibitem{Dousse-Schur-gen} Jehanne Dousse, 
Unification, Refinements and Companions of Generalisations of Schur’s Theorem, 
In: \emph{Andrews, G., Garvan, F. (eds) Analytic Number Theory, 
Modular Forms and q-Hypergeometric Series. ALLADI60 2016}, 
Springer Proceedings in Mathematics \& Statistics, vol 221. Springer, Cham., (2017).  

\url{https://doi.org/10.1007/978-3-319-68376-8_14}


\bibitem{Dousse-Siladic} Jehanne Dousse, 
Siladić’s theorem: Weighted words, refinement and companion, 
\emph{Proc. Amer. Math. Soc.}, {\bf 145}:1997--2009, (2017).  

\url{https://doi.org/10.1090/proc/13376}


\bibitem{Dousse-Capparelli-Primc} Jehanne Dousse, 
On partition identities of Capparelli and Primc, 
\emph{Advances in Mathematics}, {\bf 370}:107245, (2020) 

\url{https://doi.org/10.1016/j.aim.2020.107245}


\bibitem{DF} David S.Dummit and Richard M. Foote, 
\emph{Abstract algebra}, Vol. 3, Hoboken: Wiley, (2004).


\bibitem{FFW} Boris Feigin, Omar Foda, and Trevor Welsh, 
Andrews–Gordon type identities from combinations of Virasoro characters, 
\emph{Ramanujan J}, {\bf 17}:33--52, (2008).  

\url{https://doi.org/10.1007/s11139-006-9011-7}


\bibitem{Fomin88} Sergey V. Fomin, 
The generalized Robinson-Schensted-Knuth correspondence, 
\emph{J Math Sci} {\bf 41}:979–-991, 1988.  

\url{https://doi.org/10.1007/BF01247093}


\bibitem{Fomin95} Sergey Fomin,
Schur operators and Knuth correspondences,
\emph{Journal of Combinatorial Theory, Series A}, 
{\bf 72}(2):277--292, 1995.  

\url{https://doi.org/10.1016/0097-3165(95)90065-9}


\bibitem{jagged} Jean-Fran\c{c}ois Fortin, Patrick Jacob, and Pierre Mathieu, 
Jagged Partitions, 
\emph{The Ramanujan Journal} {\bf 10}(2):215--235, 2005.  

\url{https://doi.org/10.1007/s11139-005-4848-8}


\bibitem{GR} George Gasper and Mizan Rahman, 
\emph{Basic hypergeometric series} (Vol.96), Cambridge university press, (2004).  


\bibitem {GesselKrattenthaler}  Ira M. Gessel and C. Krattenthaler, 
Cylindric partitions, 
\emph{Trans. Amer. Math. Soc.}, {\bf 349}(2):429--479, 1997. 
	
\url{http://dx.doi.org/10.1090/S0002-9947-97-01791-1}


\bibitem{Harary-Palmer} Frank Harary, and Edgar M. Palmer, 
\emph{Graphical enumeration}, Elsevier, (2014).  


\bibitem{KR-completeASW} Shashank Kanade and Matthew C. Russell, 
Completing the $\mathrm {A}_2$ Andrews-Schilling-Warnaar identities, 
arXiv preprint arXiv:2203.05690, (2022).  

\url{https://arxiv.org/pdf/2203.05690}


\bibitem{Krattenthaler-growth-diag} Christian Krattenthaler, 
Growth diagrams, and increasing and decreasing chains in fillings of Ferrers shapes, 
\emph{Advances in Applied Mathematics}, {\bf 37}(3):404--431, 2006.  

\url{https://doi.org/10.1016/j.aam.2005.12.006}


\bibitem{KO} Ka\u{g}an Kur\c{s}ung\"{o}z and Halime \"{O}mr\"{u}uzun Seyrek, 
Combinatorial Constructions of Generating Functions of Cylindric Partitions with Small Profiles into Unrestricted or Distinct Parts, 
\emph{The Electronic Journal of Combinatorics}, {\bf 30}(2):P2.29, (2023).  

\url{https://doi.org/10.37236/11287}


\bibitem{Langer-I} Robin Langer, 
Enumeration of cylindric plane partitions, 
\emph{Discrete Mathematics \& Theoretical Computer Science}, (Proceedings), 2012.  

\url{https://doi.org/10.46298/dmtcs.3083}


\bibitem{Langer-II} Robin Langer, 
Enumeration of cylindric plane partitions-part II. arXiv preprint arXiv:1209.1807, 2012.  

\url{https://arxiv.org/pdf/1209.1807}


\bibitem{Leeuwen} Marc A. A. van Leeuwen
Spin-Preserving Knuth Correspondences for Ribbon Tableaux, 
\emph{The Electronic Journal of Combinatorics}, {\bf 12}:R10, 2005.  

\url{https://doi.org/10.37236/1907}


\bibitem{Parry} Daniel Parry,
\emph{A Polynomial Version of Meinardus' Theorem},
PhD Thesis,  Drexel University, (2014).

\url{https://doi.org/10.17918/etd-4563}


\bibitem{RR} S.Ramanujan, L.J.Rogers,
Proof of certain identities in combinatory analysis,
\emph{Proc.Cambridge Phil.Soc.},{\bf 19}, 211--216 (1919).  


\bibitem{Umbral} Steven Roman, 
\emph{The umbral calculus}, Springer New York, (2005).  


\bibitem{Rosen} Kenneth H. Rosen, 
\emph{Discrete mathematics and its applications}, 
The McGraw Hill Companies, (2007).


\bibitem{Stanley} Richard P. Stanley, 
\emph{Enumerative Combinatorics Volume 1 second edition}, 
Cambridge studies in advanced mathematics, (2011).


\bibitem{Tingley} Peter Tingley,
Three Combinatorial Models for $\widehat{sl}_n$ Crystals,
with Applications to Cylindric Plane Partitions,
\emph{International Mathematics Research Notices}, {\bf 2008}, rnm143, (2008).

\url{https://doi.org/10.1093/imrn/rnm143}


\bibitem{Tingley-Correction} Peter Tingley,
Erratum: Three combinatorial models for $\widehat{sl}_n$ crystals,
with applications to cylindric plane partitions,
\emph{International Mathematics Research Notices}, {\bf 2011}(10), 2374–-2375, (2011).

\url{https://doi.org/10.1093/imrn/rnq095}


\bibitem{Tsu} Shunsuke Tsuchioka, 
An example of $A_2$ Rogers-Ramanujan bipartition identities of level 3, 
arXiv preprint arXiv:2205.04811, (2022).

\url{https://arxiv.org/pdf/2205.04811}


\bibitem{AU23} Ali K. Uncu, 
Proofs of Modulo 11 and 13 Cylindric Kanade-Russell Conjectures 
for $A_2$ Rogers-Ramanujan Type Identities. 
arXiv preprint arXiv:2301.01359, (2023).  

\url{https://arxiv.org/pdf/2301.01359}


\bibitem{Warnaar-A2AG} S. Ole Warnaar, 
The $A_2$ Andrews–Gordon identities and cylindric partitions, 
\emph{Trans. Amer. Math. Soc. Ser. B}, {\bf 10}:715--765, 2023.  

\url{https://doi.org/10.1090/btran/147} 


\bibitem{Warnaar-BaileyTree} S. Ole Warnaar,
An $\mathrm {A} _2 $ Bailey tree and $\mathrm {A} _2^{(1)} $ Rogers-Ramanujan-type identities.
arXiv preprint arXiv:2303.09069, (2023).

\url{https://arxiv.org/pdf/2303.09069}


\end{thebibliography}

\end{document}